\documentclass[aop, preprint, 12pt]{imsart}

\RequirePackage[OT1]{fontenc}
\RequirePackage{amsthm,amsmath}
\RequirePackage[numbers]{natbib}
\usepackage{amsfonts}
\usepackage{color}
\usepackage{pifont}
\usepackage{amssymb}
\usepackage{wrapfig}

\usepackage[english]{babel}
\usepackage[T1]{fontenc}
\usepackage{graphicx}
\usepackage{latexsym}
\usepackage{amstext}
\usepackage{mathrsfs}
\usepackage{dsfont}
\usepackage{graphics}
\usepackage{enumerate}
\usepackage{paralist}
\usepackage{color}
\usepackage{fullpage}
\usepackage{url}
\newtheorem{prop}{Proposition}[section]

\newtheorem{lem}[prop]{Lemma}
\newtheorem{theo}[prop]{Theorem}
\newtheorem{cor}[prop]{Corollary}
\newtheorem{alg}[prop]{Algorithm}

\newtheorem{example}[prop]{Example}

\numberwithin{equation}{section}

\newcommand{\beq}{\begin{eqnarray}}
\newcommand{\beqq}{\begin{eqnarray*}}
\newcommand{\eeq}{\end{eqnarray}}
\newcommand{\eeqq}{\end{eqnarray*}}
\newcommand{\eps}{\varepsilon}

\newcommand{\kp}[1]{\kappa_{[#1]}}
\newcommand{\kpt}[1]{\tilde{\kappa}_{[#1]}}

\usepackage{mathtools}
\usepackage[english]{babel}
\usepackage[utf8]{inputenc}
\usepackage{amsmath}
\usepackage{graphicx}
\usepackage[colorinlistoftodos]{todonotes}
\usepackage{polynom}
\usepackage{amsfonts}
\usepackage{dsfont}
\usepackage{amsthm}
\usepackage{graphicx}
\usepackage{fullpage}
\usepackage[T1]{fontenc}
\usepackage{xcolor}

\usepackage{graphicx}
\usepackage{amsmath,amsthm,amssymb}
\usepackage{bbm}
\usepackage{tabularx}

\newcommand{\R}{\mathbb{R}}
\newcommand{\N}{\mathbb{N}}

\newcommand{\1}{\mathbbm{1}}

\newcommand{\bU}{{\texttt U}}

\newcommand{\me}{{\rm e}}

\newcommand{\eg}{e.g.}
\newcommand{\cf}{c.f.}

    \def\d{{\textnormal d}}

\newenvironment{eqnarr}{\begin{IEEEeqnarray}{rCl}}{\end{IEEEeqnarray}\ignorespacesafterend}

\DeclareMathOperator{\sgn}{sgn}

\newcommand*{\norm}[1]{\lVert #1 \rVert}

    \def\beq{\begin{eqnarr}}
    \def\eeq{\end{eqnarr}}
    \def\beqq{\begin{eqnarray*}} 
    \def\eeqq{\end{eqnarray*}} 

        \def\d{{\rm d}}

    \def\d{{\textnormal d}}
\newtheorem{remark}{Remark}[section]

\renewcommand{\ln}{\log}

\renewcommand{\ln}{\log}

\newcommand{\un}{\mathbf{1}}

  \newcommand{\D}{{\rm d}}

\definecolor{amethyst}{rgb}{0.6, 0.4, 0.8}
\definecolor{applegreen}{rgb}{0.55, 0.71, 0.0}
\definecolor{aqua}{rgb}{0.0, 1.0, 1.0}
\definecolor{asparagus}{rgb}{0.53, 0.66, 0.42}
\definecolor{amber(sae/ece)}{rgb}{1.0, 0.49, 0.0}
 	\definecolor{armygreen}{rgb}{0.29, 0.33, 0.13}
	\definecolor{shitbrown}{rgb}{0.43, 0.21, 0.1}
	\definecolor{brightpink}{rgb}{1.0, 0.0, 0.5}
	\definecolor{brightube}{rgb}{0.82, 0.62, 0.91}
	 	\definecolor{byzantine}{rgb}{0.74, 0.2, 0.64}
		\definecolor{chartreuse(web)}{rgb}{0.5, 1.0, 0.0}

\newcommand{\bT}{{\texttt T}}

\newcommand{\U}{\texttt U}
\newcommand{\eU}{\emph{\texttt U}}

\newcommand{\bS}{{\texttt S}}
\newcommand{\ebS}{{\emph{\texttt S}}}

\newcommand{\bF}{{\texttt F}}
\newcommand{\ebF}{{\emph{\texttt F}}}

\newcommand{\bL}{{\texttt L}}
\newcommand{\ebL}{{\emph{\texttt L}}}

\newcommand{\bA}{{\texttt A}}

\newcommand{\bJ}{\texttt J}
\newcommand{\ebJ}{\emph{\texttt J}}

\startlocaldefs
\numberwithin{equation}{section}
\theoremstyle{plain}

\endlocaldefs

\begin{document}

\begin{frontmatter}
\title{Monte Carlo Methods for the Neutron \\ Transport Equation}


\begin{aug}
\author{\fnms{Alexander M.G. Cox}\thanksref{t3}\ead[label=e1]{a.m.g.cox@bath.ac.uk}},
\author{\fnms{Simon C. Harris}\thanksref{t3}\ead[label=e2]{simon.harris@auckland.ac.nz}},\\
\author{\fnms{Andreas E. Kyprianou}\thanksref{t3}\ead[label=e3]{a.kyprianou@bath.ac.uk}}
\and 
\author{\fnms{Minmin Wang}\thanksref{t3}\ead[label=e4]{wangminmin03@gmail.com}},

\thankstext{t3}{Supported by EPSRC grant EP/P009220/1.}



\address{
A. M. G. Cox,  and \\
A.E. Kyprianou \\
University of Bath\\
Department of Mathematical Sciences \\
Bath, BA2 7AY\\
 UK.\\
\printead{e1}\\
\printead{e3}\\
}
\address{
M. Wang\\
School of Mathematical\\
 and Physical Sciences\\
University of Sussex\\
Sussex House, Falmer\\
Brighton, BN1 9RH\\
UK.\\
\printead{e4}}

\address{S. C. Harris\\
Department of Statistics\\
University of Auckland\\
Private Bag 92019\\
Auckland 1142\\
New Zealand\\
\printead{e2}
}
\end{aug}

\begin{abstract}\hspace{0.1cm}
  This paper continues our treatment of the Neutron Transport Equation (NTE) building on the work in \cite{MultiNTE}, \cite{SNTE} and \cite{SNTE-II}, which describes the flux of neutrons through inhomogeneous fissile medium. Our aim is to analyse existing and novel Monte Carlo (MC) algorithms, aimed at simulating the lead eigenvalue associated with the underlying model. This quantity is of principal importance in the nuclear regulatory industry for which the NTE must be solved on complicated inhomogenous domains corresponding to nuclear reactor cores, irradiative hospital equipment, food irradiation equipment and so on.  We include a complexity analysis of such MC algorithms, noting that no such undertaking has previously appeared in the literature. The new MC algorithms offer a variety of advantages and disadvantages of accuracy vs cost, as well as the possibility of more convenient computational parallelisation.
\end{abstract}



\end{frontmatter}

\setcounter{tocdepth}{1}

\section{Introduction:  The Neutron Transport Equation} In this paper, we continue the stochastic analysis of the neutron transport equation (NTE), cf. \cite{SNTE, SNTE-II, MultiNTE}, but now focusing mainly on Monte Carlo methods. We begin in  this first section with a review of the classical and recent contributions to the mathematical framework of the NTE, out of which we will build completely new Monte Carlo algorithms in the subsequent sections. 

\smallskip

We recall that the NTE describes the flux of neutrons across a planar cross-section in an inhomogeneous fissile medium (typically measured in number of neutrons per cm$^2$ per second) where nuclear fission is actively creating additional neutrons.  Neutron flux is described as a function of time, $t$, Euclidian location, $r$, neutron velocity, $\upsilon$ and neutron energy $E$. It is not uncommon in the physics literature, as indeed we shall do here, to assume that energy is a function of velocity, thereby reducing the number of variables by one. This allows us to describe the dependency of flux more simply in terms of time and, what we call, the {\it configuration variables} $ (r, v) \in D \times V$ where $D\subseteq\mathbb{R}^3$ is a non-empty, open, smooth, bounded and convex domain, and $V$ is the velocity space, which we will assume 
to be $V = \{\upsilon\in \mathbb{R}^3: \upsilon_{\texttt{min}}\leq |\upsilon|\leq \upsilon_{\texttt{max}}\}$, where $0<\upsilon_{\texttt{min}}<\upsilon_{\texttt{max}}<\infty$.  \smallskip

Whilst the NTE is usually stated in the applied mathematics and physics literature in its forwards form, we will introduce it here in its backwards form. A fuller description of the relationship between the two is given in the accompanying papers \cite{MultiNTE} and \cite{SNTE}.  From a mathematical perspective, the backwards NTE is given\footnote{The differential operator $\nabla\psi_t(r, \upsilon)$ is the gradient in the spatial variable, $r$, only, so that $\upsilon\cdot\nabla\psi_t(r, \upsilon)$ is the gradient of $\psi$ in the direction of $\upsilon$, multiplied by the norm of the velocity.} by
\begin{align}
\frac{\partial}{\partial t}\psi_t(r, \upsilon) &= \upsilon\cdot\nabla\psi_t(r, \upsilon)  -\sigma(r, \upsilon)\psi_t(r, \upsilon)\notag\\
&+ \sigma_{\texttt{s}}(r, \upsilon)\int_{V}\psi_t(r, \upsilon') \pi_{\texttt{s}}(r, \upsilon, \upsilon')\d\upsilon' + \sigma_{\texttt{f}}(r, \upsilon) \int_{V}\psi_t(r, \upsilon') \pi_{\texttt{f}}(r, \upsilon, \upsilon')\d\upsilon',
\label{bNTE}
\end{align}
where the different components (or {\it cross-sections}\footnote{There is a discrepancy between what we mean by cross-sections here  and what is typical in the physics and engineering literature. Here, cross-sections $\sigma_\mathtt{s}, \sigma_\mathtt{f}$ are considered as rates per unit time, which is natural for probabilistic considerations, whereas the norm  is to consider them as rates per unit neutron track length, which is natural when thinking of modelling how physical interaction occurs. The difference is a factor of $\upsilon$, which converts rate per unit track length to rate per unit time. This also explains why \eqref{bNTE} is usually written with a factor $1/\upsilon$ preceding the time derivative.} as they are known in the nuclear physics literature) have the following interpretation:
\begin{align*}
\sigma_{\texttt{s}}(r, \upsilon) &: \text{ the rate at which scattering occurs from incoming velocity $\upsilon$,}\\
\sigma_{\texttt{f}}(r, \upsilon) &: \text{  the rate at which fission occurs from incoming velocity $\upsilon$,}\\
\sigma(r, \upsilon) &: \text{ the sum of the rates } \sigma_{\texttt{f}}+ \sigma_{\texttt{s}} \text{ and is known as the total cross section,}\\
\pi_{\texttt{s}}(r, \upsilon, \upsilon')\d\upsilon' &: \text{  the scattering yield at velocity $\upsilon'$ from incoming velocity }  \upsilon, \\
 &\hspace{0.5cm}\text{ satisfying }\textstyle{\int_V}\pi_{\texttt{s}}(r, \upsilon, \upsilon'){\rm d}\upsilon'=1,\text{ and }\\
 \pi_{\texttt{f}}(r, \upsilon, \upsilon')\d\upsilon' &:  \text{  the neutron yield at velocity $\upsilon'$ from fission with incoming velocity }   \upsilon,\\
 &\hspace{0.5cm}\text{ satisfying }{\color{black} m_{\texttt{f}}(r,\upsilon) := \textstyle{\int_V}\pi_{\texttt{f}}(r, \upsilon, \upsilon')\d\upsilon' <\infty.}
\end{align*}
Note that scattering is the physical process that occurs when a neutron comes into close proximity with an atomic nucleus causing a change in velocity.  A justification of the structure of the backwards (and indeed the forwards) NTE  form from the physics of nuclear scattering and fission is given in a number of classical texts e.g. \cite{DL6}  and \cite{M-K}, see also \cite{MultiNTE} for a multi-species treatment. 
 
 \smallskip
 
 It is also usual to assume the additional boundary conditions 
\begin{equation}
\left\{
\begin{array}{ll}
\psi_0(r, \upsilon) = g(r, \upsilon) &\text{ for }r\in D, \upsilon\in{V},
\\
&
\\
\psi_t(r, \upsilon) = 0& \text{ for }r\in \partial D
\text{ if }\upsilon
\cdot{\bf n}_r>0.
\end{array}
\right.
\label{BC}
\end{equation}
where  ${\bf n}_r$ is the outward facing normal of $D$ at $r\in \partial D$ and $g: D\times {V}\to [0,\infty)$ is a bounded, measurable function which we will later assume has some additional properties. Roughly speaking, this means that neutrons at the boundary which are travelling in the direction of the exterior of the domain are lost to the system.
The second of the two conditions in \eqref{BC} is often written $\psi_t|_{{\partial D}^+} = 0$, where ${\partial D}^+ : = \{(r,\upsilon)\in D\times V :r\in \partial D
\text{ if }\upsilon
\cdot{\bf n}_r>0 \}$.

\smallskip
 
For mathematical reasons, in the forthcoming analysis, we will assume that 

\smallskip

{\bf (H1): Cross-sections $\sigma_{\texttt{s}}$, $\sigma_{\texttt{f}}$, $\pi_{\texttt{s}}$ and $\pi_{\texttt{f}} $ are uniformly bounded away from   infinity.}

\smallskip

{\bf (H2): The mixed cross-sections satisfy}
\[
\inf_{r\in D, \upsilon, \upsilon'\in V} \sigma_{\texttt{s}}(r,\upsilon)\pi_{\texttt{s}}(r, \upsilon, \upsilon') + \sigma_{\texttt{f}}(r,\upsilon) \pi_{\texttt{f}}(r, \upsilon, \upsilon')>0.
\]

\smallskip

{\bf  (H3): There exists an open ball $B$, compactly embedded in $D$, such that}
$$
\inf_{r \in B, \upsilon, \upsilon' \in V}\sigma_{\texttt{f}}(r, \upsilon)\pi_{\texttt{f}}(r, \upsilon, \upsilon') > 0.
$$
Roughly speaking, (H2) says that either scattering or fission, or both, must be able to occur everywhere at all energies. On the other hand, (H3) says that fission must be able to occur at least in some region at all energies.
\smallskip

In addition, the maximum number of neutrons that can be emitted during a fission event with positive probability is capped by physical constraints. For example in an environment where the heaviest nucleus is {\it Uranium-235}, there are at most 143 neutrons that can be released in a fission event, albeit, in reality it is more likely that 2 or 3 are released in any fission event. We will thus occasionally work with the assumption:
 
\smallskip

 {\bf (H4): Fission offspring are bounded in number  by the constant $n_{\texttt{max}}> 1$.}
 
 \smallskip

In particular this means that 
$
\textstyle{\sup_{ r\in D, \upsilon\in V}m_{\texttt{f}}(r,\upsilon)\leq n_\texttt{max}}.
$

\smallskip

It turns out that equation \eqref{bNTE} is ill-defined unless considered in an appropriate sense. There are in fact two ways this can be done. The first approach is to formulate the problem as an abstract Cauchy problem 
(ACP) in an appropriate Banach space; see e.g. \cite{D, DL6}, which is typically the approach taken in the applied mathematics literature. The second approach is to consider the mild form of the equation. This approach has seen somewhat less attention, although it  lends itself better to Monte Carlo methods.

\section{Mild NTE and stochastic representation} In order for us to give a clear statement of what a mild version of \eqref{bNTE} should look like and how it relates to Monte Carlo simulation, we must first consider an auxiliary stochastic process.

\subsection{The physical process}
Consider a neutron branching process (NBP), which at time $t\geq0$ is represented by a configuration of particles which are specified via their physical location and velocity in $D\times V$, say $\{(r_i(t), \upsilon_i(t)): i = 1,\dots , N_t\}$, where $N_t$ is the number of particles alive at time $t\ge 0$.
In order to describe the process, we will represent it as a process in the space of the random counting measures
\[
X_t(A) = \sum_{i=1}^{N_t}\delta_{(r_i(t), \upsilon_i(t))}(A), \qquad A\in\mathcal{B}(D\times V), \;t\ge 0,
\]
where $\delta_{(r, v)}$ is the Dirac measure of unit mass at $(r,v)$ defined on $\mathcal{B}(D\times V)$, the Borel subsets of $D\times V$.
The evolution of $(X_t, t\geq 0)$ is a stochastic process valued in the space of atomic measures
$
\mathcal{M}(D\times V): = \{\textstyle{\sum_{i = 1}^n}\delta_{(r_i,\upsilon_i)}: n\in \mathbb{N}, (r_i,\upsilon_i)\in D\times V, i = 1,\cdots, n\}
$
which evolves randomly as follows.

\smallskip

A particle positioned at $r$ with velocity $\upsilon$ will continue to move along the trajectory $r + \upsilon t$ for $t\geq 0$, until one of the following things happens. 
\begin{enumerate}
\item[(i)] The particle leaves the physical domain $D$, in which case it is instantaneously killed. 
\smallskip

\item[(ii)] Independently of all other neutrons, a scattering event occurs when a neutron comes in close proximity to an atomic nucleus and, accordingly, makes an instantaneous change of velocity. For a neutron in the system with position and velocity $(r,\upsilon)$, if we write $T_{\texttt{s}}$ for the random time that scattering may occur, then independently of any other physical event that may affect the neutron, 
$
\Pr(T_{\texttt{s}}>t) = \exp\{-\textstyle{\int_0^t} \sigma_{\texttt{s}}(r+\upsilon s, \upsilon){\rm d}s \}, $ for $t\geq0$ with $(r+vs : s\in[0,t])\subset D$, that is, as long as the particle has remained within the domain up to time $t$.
\smallskip

When scattering occurs at space-velocity $(r,\upsilon)$, the new velocity $\upsilon'$ is selected in $V$ independently with probability $\pi_{\texttt{s}}(r, \upsilon, \upsilon')\d\upsilon'$. 
\smallskip

\item[(iii)] Independently of all other neutrons, a fission event occurs when a neutron smashes into an atomic nucleus. 
For a neutron in the system  with initial position and velocity $(r,\upsilon)$, if we write $T_{\texttt{f}}$ for the random time that scattering may occur, then independently of any other physical event that may affect the neutron, 
$
\Pr(T_{\texttt{f}}>t) = \exp\{-\textstyle{\int_0^t} \sigma_{\texttt{f}}(r+\upsilon s, \upsilon){\rm d}s \},$ for $t\geq 0
$  with $(r+vs : s\in[0,t])\subset D$.
\smallskip

When fission occurs,  the smashing of the atomic nucleus produces  lower mass isotopes and releases  a random number of neutrons, say $N\geq 0$, which are ejected from the point of impact with randomly distributed, and possibly correlated, velocities, say $\{\upsilon_i: i=1,\cdots, N\}$. The outgoing velocities are described by  the atomic random measure 
\begin{equation}
\label{PP}
\mathcal{Z}(A): = \sum_{i= 1}^{N } \delta_{\upsilon_i}(A), \qquad A\in\mathcal{B}(V).
\end{equation} 

When fission occurs at location $r\in\mathbb{R}^d$ from a particle with incoming velocity $\upsilon\in{V}$, we denote by ${\mathcal P}_{(r,\upsilon)}$ the law of $\mathcal{Z}$.  The probability laws ${\mathcal P}_{(r,\upsilon)}$, with their corresponding expectations ${\mathcal E}_{(r,\upsilon)}$, are such that, for $\upsilon'\in{V}$ and bounded measurable $g: V\to[0,\infty)$,
\begin{align}
\int_V g(\upsilon')\pi_{\texttt{f}}(r, v, \upsilon')\d\upsilon' &= {\mathcal E}_{(r,\upsilon)}\left[\int_V g(\upsilon')\mathcal{Z}(\d \upsilon')\right]
=: {\mathcal E}_{(r,\upsilon)}[\langle g, \mathcal{Z}\rangle].
\label{Erv}
\end{align}
Note the possibility that $\Pr(N = 0)>0$, which will be tantamount to neutron capture (that is, where a neutron slams into a nucleus but no fission results and the neutron is absorbed into the nucleus). 
\end{enumerate}
\smallskip

In essence, one may think of the process $X: = (X_t, t\geq 0)$ as a typed spatial Markov branching process, where  the type of each particle is the velocity  $\upsilon\in{V}$  and the underlying 
motion is nothing more than movement in a straight line with velocity $\upsilon$ 
until the next jump in type (scattering) or branching event (fission), with these occurring independently for each particle at rates depending only on their current spatial position and type.
\begin{remark}\label{nonuniqueNBP}\rm
  The NBP is thus parameterised by the quantities $\sigma_{\texttt s}, \pi_{\texttt s}, \sigma_{\texttt f}$ and the family of measures ${\mathcal P} =({\mathcal P}_{(r,\upsilon)} , r\in D,\upsilon\in V)$ and accordingly we refer to it as a $(\sigma_{\texttt s}, \pi_{\texttt s}, \sigma_{\texttt f}, \mathcal{P})$-NBP. It is associated to the NTE via the relation \eqref{Erv}, however this association does not uniquely identify the NBP. Said another way, the quantities $\sigma_{\texttt s}, \pi_{\texttt s}, \sigma_{\texttt f}, \pi_{\texttt f}$ in the NTE do not uniquely identify an underlying NBP. In particular, only the {average} number of particles created at fission events is specified via $\pi_{\texttt f}$, but not their probability distributions as required for ${\mathcal P}$. See Remark 2.1 of \cite{SNTE}.

\smallskip

Later on, we will see that it is possible to simulate a NBP with only the parameters $\sigma_{\texttt s}, \pi_{\texttt s}, \sigma_{\texttt f}, \pi_{\texttt f}$ given. Such processes are not uniquely determined, however the aforementioned cross-sections are sufficient to identify at least one NBP which is associated to the NTE with those cross-sections. In that case, we will abuse our terminology and refer to it as a $(\sigma_{\texttt s}, \pi_{\texttt s}, \sigma_{\texttt f}, \pi_{\texttt f})$-NBP.
\hfill$\diamond$\end{remark}

Write $\mathbb{P}_{\delta_{(r, \upsilon)}}$ for the the law of $X$ when issued from a single particle  with space-velocity configuration $(r,\upsilon)\in {D}\times V$.
More generally, for $\mu\in\mathcal{M}(D\times V)$, 
we understand 
$
\mathbb{P}_{\mu} : = \mathbb{P}_{\delta_{(r_1,\upsilon_1)}}\otimes\cdots\otimes\mathbb{P}_{\delta_{(r_n,\upsilon_n)}}$
 when $\mu = \textstyle{\sum_{i = 1}^n} \delta_{(r_i,\upsilon_i)}.
$
In other words, the process $X$ when issued from initial configuration $\mu$ is equivalent to issuing $n$ independent copies of $X$, where the $i^{\textrm{th}}$ copy starts from a single particle with configuration $(r_i, \upsilon_i)$, for $i = 1,\cdots, n$.
\smallskip

Like all spatial Markov branching processes, $(X, \mathbb{P})$, where $\mathbb{P} : = (\mathbb{P}_{\mu}, \mu\in\mathcal{M}(D\times V))$, respects the {\it Markov branching property} with respect to the filtration 
$\mathcal{F}_t: = \sigma(  (r_i(s),\upsilon_i(s)) : i = 1,\cdots, N_s, s\leq t)$, $t\geq 0$.  That is to say, for all bounded and measurable $g: {D}\times V\to [0,\infty)$ and $\mu\in\mathcal{M}(D\times V)$ written $\mu = \textstyle{\sum_{i = 1}^n\delta_{(r_i,\upsilon_i)}}
$, for $s,t\geq0$, on the event $\{X_t = \mu\}\in\mathcal{F}_t$,
we have
$
\textstyle{\mathbb{E}_{\mu}[\prod_{i=1}^{N_{t+s}}g(r_i(t+s), \upsilon_i(t+s))|\mathcal{F}_t] = \prod_{i=1}^{n}
u_s[g](r_i, \upsilon_i)},$ for  $t\geq 0, r_i\in D, \upsilon_i\in{V},
$
where $\textstyle{u_s[g](r,\upsilon): = \mathbb{E}_{\delta_{(r, \upsilon)}}[\prod_{i=1}^{N_{s}}g(r_i(s), \upsilon_i(s))]}.
$ 

\smallskip

What is of particular interest to us in the context of the NTE is what we call the  {\it expectation semigroup} of the neutron branching process. More precisely, and with pre-emptive notation, we are interested in 
\begin{equation}
\psi_t[g](r,\upsilon) : = \mathbb{E}_{\delta_{(r, \upsilon)}}[\langle g, X_t \rangle], \qquad t\geq 0, r\in \bar{D}, \upsilon\in{V},
\label{semigroup}
\end{equation}
for $g\in L^+_\infty(D\times V)$, the space of non-negative uniformly bounded measurable functions on $D\times V$.  Here we have made a slight abuse of notation (see $\langle \cdot,\cdot\rangle$ as it appears in \eqref{Erv}) and written $\langle g, X_t \rangle$ to mean $\textstyle{\int_{D\times V}} g(r,\upsilon)X_t(\d r,\d \upsilon)$, that is, $\langle g, X_t \rangle = \sum_{i=1}^{N_t} g(r_i(t), v_i(t) )$.  \smallskip

Thus, $\psi_t[g](r,\upsilon)$ gives the average value of a weighted sum over all the neutrons present at time $t$ when started from a single neutron with a configuration $(r,\upsilon)$, where the weight of a neutron with configuration $(r',\upsilon')$ is given by $g(r',\upsilon')$.
\smallskip

Heuristically, if we take $g(r',\upsilon')=\delta_{(r', \upsilon')}$, then $\psi_t$ recovers the expected density of neutrons that have a position-velocity configuration $(r', \upsilon')$ at time $t$.  In particular, the expected flux of neutrons across any planar cross-section could then be recovered from these quantities.

\smallskip

To see why $(\psi_t, t\geq 0)$ deserves the name of expectation semigroup, it is a straightforward exercise with the help of the  Markov branching property to show that 
\begin{equation}
\psi_{t+s}[g](r,\upsilon) =\psi_t[\psi_s[g]](r,\upsilon)\qquad s,t\geq 0.
\label{branchingsemigroup}
\end{equation}

\subsection{Mild form of the Neutron Transport Equation}
The reason why the choice of notation \eqref{semigroup} is pre-emptive is because we can now show how it relates to the aforementioned mild form of the NTE.
 In order to do so, let us momentarily introduce some more notation:
the deterministic evolution
$
\texttt{U}_t[g]( r,\upsilon) = g( r+\upsilon t, \upsilon)\mathbf{1}_{\{t<\kappa^D_{r,\upsilon}\}}, t\geq 0,
$
where
$
\kappa_{r,\upsilon}^{D} := \inf\{t>0 : r+\upsilon t\not\in D\},
$
represents the advection semigroup associated with a single neutron travelling at velocity $\upsilon$ from $r$, killed on exit from $D$.

\begin{lem}[\cite{MultiNTE}]\label{NBPrep}Under (H1) and (H2), for $g\in L^+_\infty(D\times V)$, there exist  constants $C_1,C_2>0$ such that  $\psi_t[g]$, as given in \eqref{semigroup}, is uniformly bounded by $ C_1\exp(C_2 t)$, for all $t\geq 0$. Moreover, $(\psi_t[g], t\geq 0)$ is the unique solution, which is bounded in time, to  the so-called mild equation (also called a {\it Duhamel solution} in the PDE literature):
\begin{equation}
\psi_t[g] = \emph{\texttt{U}}_t[g] + \int_0^t \emph{\texttt{U}}_s[({\ebS} + {\ebF})\psi_{t-s}[g]]\d s, \qquad t\geq 0,
\label{mild}
\end{equation}
for which \eqref{BC} holds.
\end{lem}

In \eqref{mild}, $\texttt{S}$ is the corresponding scattering operator with action 
\[
\texttt{S} f(r,\upsilon) = \sigma_{\texttt{s}}(r,\upsilon)\int_V \Big(f(r, \upsilon')  -f(r,\upsilon)\Big)\pi_{\texttt{s}}(r,\upsilon,\upsilon')\d \upsilon', \qquad f\in L^+_\infty(D\times V),
\] and $\texttt{F}$ the fission operator with action
\[
\texttt{S} f(r,\upsilon) = \sigma_{\texttt{f}}(r,\upsilon)\int_V f(r, \upsilon')\pi_{\texttt{f}}(r,\upsilon,\upsilon')\d \upsilon'
- \sigma_{\texttt{f}}(r,\upsilon)f(r,\upsilon),\qquad  f\in L^+_\infty(D\times V),
\] 
see \cite{MultiNTE} for further details. 
The connection of the expectation semigroup \eqref{semigroup} with the NTE \eqref{bNTE} was explored in the recent articles \cite{MultiNTE, SNTE} (see also older work in \cite{D, LPS}).

\subsection{The many-to-one representation}\label{m21subsection}
It turns out that there is a second stochastic representation of the unique solution to \eqref{mild}. In order to describe it, we need to introduce the notion of a {\it neutron random walk} (NRW).

\smallskip

A NRW on $D$, which is defined by its scatter rates,
$\alpha(r,\upsilon)$, $r\in D, \upsilon\in V$, and scatter probability
densities $\pi(r,\upsilon,\upsilon')$,
$r\in D, \upsilon,\upsilon'\in V$ such that
$\textstyle{\int_V \pi(r,\upsilon, \upsilon')\d\upsilon'}=1$ for all
$r\in D, \upsilon\in V$. Simply, when issued with a velocity
$\upsilon$, the NRW will propagate linearly with that velocity until
either it exits the domain $D$, in which case it is killed, or at the
random time $T_{\texttt{s}}$ a scattering occurs, where
$ \Pr(T_{\texttt{s}}>t) = \exp\{-\textstyle{\int_0^t}
\alpha(r+\upsilon t, \upsilon){\rm d}s \}, $ for $t\geq0.$ When the
scattering event occurs in position-velocity configuration
$(r,\upsilon)$, a new velocity $\upsilon'$ is selected with
probability $\pi(r,\upsilon,\upsilon')\d\upsilon'$. If we denote by
$(R,\Upsilon) = ((R_t, \Upsilon_t), t\geq 0)$, the position-velocity
of the resulting continuous-time random walk on $D\times V$ with an
additional cemetery state $\{\dagger\}$ for when it leaves the domain
$D$, then it is easy to show that $(R,\Upsilon)$ is a Markov process.
(Note, neither $R$ nor $\Upsilon$ alone is Markovian.) We call the
process $(R,\Upsilon)$ an $\alpha\pi$-NRW. It is worth remarking that
when $\alpha\pi$ is given as a single rate function, the density
$\pi$, and hence the rate $\alpha$, is uniquely identified by 
normalising  the given product form $\alpha\pi$ to make it a probability distribution.

\smallskip

To  describe the second stochastic representation of \eqref{mild},  we shall take 
\begin{align}
\alpha(r,\upsilon)\pi(r, \upsilon, \upsilon') = \sigma_{\texttt{s}}(r,\upsilon)\pi_{\texttt{s}}(r, \upsilon, \upsilon') + \sigma_{\texttt{f}}(r,\upsilon) \pi_{\texttt{f}}(r, \upsilon, \upsilon')\qquad r\in D, \upsilon, \upsilon'\in V.
\label{alpha}
\end{align}
We also need to define
\begin{equation}
\beta(r,\upsilon)=\sigma_{\texttt{f}}(r,\upsilon)\left(\int_V\pi_{\texttt{f}}(r, \upsilon,\upsilon')\d\upsilon'-1\right)\geq -\sup_{r\in D, \upsilon\in V}\sigma_{\texttt{f}}(r,\upsilon)>-\infty,
\label{betadef}
\end{equation}
where the uniform lower bound is due to assumption (H1).
The following result was established in Lemma 6.1 of \cite{MultiNTE}.
\begin{lem}[Many-to-one formula, \cite{MultiNTE}]\label{NRWrep} Under the assumptions of Lemma 
\ref{NBPrep}, we have the second representation 
\begin{equation}
\psi_t[g](r,\upsilon)  = \mathbf{E}_{(r,\upsilon)}\left[{\rm e}^{\int_0^t\beta(R_s, \Upsilon_s)\D s}g(R_t, \Upsilon_t) \mathbf{1}_{(t < \tau^D)}\right], \qquad t\geq 0,r\in D, \upsilon\in V,
\label{phi}
\end{equation}
where $\tau^D = \inf\{t>0 : R_t\not\in D\}$ and 
 ${\bf P}_{(r, v)}$ for the law of the $\alpha\pi$-NRW  starting from a single neutron with configuration $(r, \upsilon)$. 
\end{lem}

Define $\overline{\beta} := \textstyle{\sup_{r\in D, \upsilon\in V}\beta(r,\upsilon)}$ and note that it is uniformly bounded thanks to (H1) and (H4).  Let us now introduce  $\texttt{P}^\dagger: = (\texttt{P}^\dagger_t, t\geq 0)$ for the expectation semigroup of the $\alpha\pi$-neutron random walk with potential $\beta$, such as is represented by the semigroup \eqref{phi}, but now killed at rate $ (\overline{\beta}-\beta)$. More precisely, for $g\in L^+_\infty (D\times V)$,
\begin{align}
\texttt{P}^\dagger_t[g](r,\upsilon) &= \psi_t[g](r,\upsilon){\rm e}^{-\overline{\beta} t}\notag\\
& =\mathbf{E}_{(r,\upsilon)}\left[{\rm e}^{\int_0^t (\beta (R_s, \Upsilon_s) - \overline{\beta} )\d s}g(R_t, \Upsilon_t) \mathbf{1}_{(t<\tau^D)}\right]\notag \\
&=\mathbf{E}_{(r,\upsilon)}\left[g(R_t, \Upsilon_t) \mathbf{1}_{(t<\texttt{k})}\right] \notag\\
&= :\mathbf{E}^\dagger_{(r,\upsilon)}\left[g(R_t, \Upsilon_t)\right] ,%
\qquad t\geq 0, r\in D, \upsilon\in V,
\label{boldPdagger}
\end{align}
 where 
 \[
 \texttt{k} =\inf\{t>0: \int_0^t (\overline{\beta} - \beta(R_s, \Upsilon_s) )\d s >\mathbf{e}\}\wedge \tau^D,
 \]
and $\mathbf{e}$ is an independent exponentially distributed random variable with mean 1.
We will henceforth write  $\mathbf{P}^{\dagger}: = (\mathbf{P}^{\dagger}_{(r,\upsilon)}, r\in \bar D, \upsilon\in V)$ for  the associated  (sub)probability measures associated to $\mathbf{E}^{\dagger}_{(r,\upsilon)}$, $r\in \bar D, \upsilon\in V$. The family $\mathbf{P}^{\dagger}$  now defines a Markov family of probability measures on the path space of the neutron random walk with cemetery state $\{\dagger\}$; note, paths are sent to the cemetery state either when  hitting the boundary $\partial D$ or the clock associated to the killing rate $\overline{\beta} - \beta$ rings. Note also that this process is a killed Piecewise Deterministic Markov Process, as introduced in \cite{davis_piecewise-deterministic_1984}. \smallskip


\begin{remark}\label{CVtheoremupgrade}\rm It is worthy of  discussion  that the proof of Theorem \ref{CVtheorem} in \cite{SNTE} uses the representation \eqref{boldPdagger}. Moreover, examination of the proof there reveals that, despite the setting of \eqref{boldPdagger}  forcing a relationship between  the value $\beta$ and the cross sections in the NTE, this is not needed in the proof. The only requirement is that $\beta$ is uniformly bounded from above. We will use this fact later on in this paper.
\hfill$\diamond$\end{remark}

\section{Spectral asymptotics}
Theorem 5.3 in \cite{MultiNTE} also gives the leading order behaviour of the mild solution of the NTE, mirroring similar results for the setting where the NTE is cited as an abstract Cauchy problem (c.f. Theorem 7.1 in \cite{SNTE} and \cite{M-K}); for convenience we reproduce it here.

\begin{theo}\label{CVtheorem}
Suppose that (H1) and (H2) hold. Then, for the semigroup  $(\psi_t,t\geq0)$ identified by \eqref{mild},  there exists  a $\lambda_*\in\mathbb{R}$, a  positive\footnote{To be precise, by a positive eigenfunction, we mean a mapping from $D\times V\to (0,\infty)$. This does not prevent it being valued zero on $\partial D$, as $D$ is an open bounded, convex domain.} right eigenfunction $\varphi \in L^+_\infty(D\times V)$ and a left eigenmeasure which is absolutely continuous with respect to Lebesgue measure on $D\times V$ with density $\tilde\varphi\in L^+_\infty(D\times V)$, both having associated eigenvalue ${\rm e}^{\lambda_* t}$, and such that $\varphi$  (resp. $\tilde\varphi$) is uniformly (resp. a.e. uniformly) bounded away from zero on each compactly embedded subset of $D\times V$. In particular, for all $g\in L^+_{\infty}(D\times V)$,
\begin{equation}
\langle\tilde\varphi, \psi_t[g]\rangle = {\rm e}^{\lambda_* t}\langle\tilde\varphi,g\rangle\quad  \text{(resp. } 
\psi_t[\varphi] = {\rm e}^{\lambda_* t}\varphi
\text{)} \quad t\ge 0.
\label{leftandright}
\end{equation}
Moreover, there exists $\varepsilon>0$ such that 
\begin{equation}
\sup_{g\in L^+_\infty(D\times V): \norm{g}_\infty\leq 1}  \left\|{\rm e}^{-\lambda_* t}{\varphi}^{-1}{\psi_t[g]}-\langle\tilde\varphi, g\rangle\right\|_\infty = O({\rm e}^{-\varepsilon t}) \text{ as $t\rightarrow\infty$.}
\label{spectralexpsgp}
\end{equation}
%
\end{theo}

In parallel with the asymptotic spectral decomposition of the semigroup given in \eqref{spectralexpsgp},  \cite{SNTE, SNTE-II} also considered analogous stochastic behaviour, which we briefly discuss.

\smallskip

An important object that emerges in this context is the process 
 \[
 W_t: = {\rm e}^{-\lambda_* t}\frac{\langle \varphi, X_t\rangle}{\langle \varphi, \mu\rangle},\qquad  t\geq 0,
 \]
  which is a martingale with unit mean under $\mathbb{P}_\mu$.  As a non-negative martingale, it automatically has an almost sure limit, and, from \cite{SNTE}, we have the following  understanding of its asymptotic behaviour:
  \begin{theo}[\cite{SNTE}]
The martingale limit $W_\infty$ is identically zero almost surely if and only if $\lambda_*\leq 0$, otherwise $(W_t, t\geq0)$ is $L_2(\mathbb{P}_\mu)$-convergent with $\mathbb{E}_\mu[W_\infty]=1$ and $\mathbb{E}_\mu[W_\infty^2]<\infty$.
\end{theo}
  
In essence, we can think of the martingale $W$ as give us information about the stochastic growth of $\langle g, X_t\rangle $ for the special case that $g = \varphi$. The full picture was given in \cite{SNTE-II}.
  
\begin{theo}[\cite{SNTE-II}]\label{SLLN} 
For all $g\in L^+_\infty(D\times V)$ such that, up to a multiplicative constant, $g\leq \varphi$, under the assumptions of Theorem \ref{CVtheorem}, by normalising $\varphi, \tilde\varphi$ such that $\langle\varphi, \tilde\varphi\rangle =1$, we have 
\begin{equation} \label{eq:1}
\lim_{t\to\infty} {\rm e}^{-\lambda_* t}\frac{\langle g, X_t\rangle}{\langle\varphi, \mu\rangle}=  \langle g,\tilde{\varphi}\rangle W_\infty.
\end{equation}
$\mathbb{P}_\mu$-almost surely, for $\mu\in\mathcal{M}(D\times V)$. 
\end{theo}

\smallskip

\begin{wrapfigure}{R}{0.5\textwidth}
\vspace{-22pt}
  \begin{center}
\includegraphics[height=7cm]{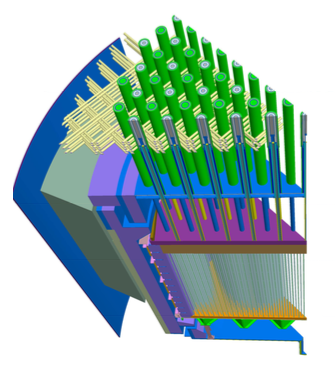}
  \end{center}
  \vspace{-20pt}
\caption{The geometry of a nuclear reactor core representing a physical domain $D$, on to which the different cross-sectional values of  $ \sigma_{\emph{\texttt{s}}},  \sigma_{\emph{\texttt{f}}},  \pi_{\emph{\texttt{s}}},  \pi_{\emph{\texttt{f}}}$ are mapped as numerical values.}
\label{fig}
\end{wrapfigure}
The left-eigenfunction $\tilde\varphi$ is called the {\it importance  map} and offers a quasi-stationary  distribution of radioactive activity (unless $\lambda_*=0$, in which case it is a stationary profile). We shall refer to the right-eigenfunction as the {\it heat map} and gives a principal representation of flux which grows at rate $\lambda_*$.
In
modern nuclear reactor  design and safety regulation, it is usually the case that
virtual reactor models such as the one seen in Figure \ref{fig} 
 are
designed such that $\lambda_* = 0$ (in fact, slightly above zero\footnote{As can be seen from \cite{SNTE-II}, reproduced in Theorem \ref{SLLN} below, at criticality, the underlying process of fission will die out with probability one, despite the semigroup $(\psi_t, t\geq 0)$ having an asymptotic limit.}) and the behaviour of $\tilde\varphi$ and $\varphi$
within spatial domains on the human scale remain within regulated levels.
Existing physics and engineering literature with focus on applications in the nuclear regulation industry, has largely been concerned with different numerical methods for estimating the value of the eigenvalue $\lambda_*$ as well as the eigenfunction $\varphi$ and eigenmeasure $\tilde{\varphi}(r,\upsilon)\d r\d\upsilon$.  Further discussion on this aspect of the NTE can be found in \cite{SNTE, LM, DS, PR, GPS, MT, M-K, PP, Bell, BG, D, DL6}

\smallskip


In this article we look at various existing and new Monte Carlo estimates of the principal quantity  $\lambda_*$, which opens the door to the much more complicated problem of numerically constructing the  eigenfunctions  $\varphi$ and $\tilde\varphi$. In particular, we are interested in how Monte Carlo methods based on the recently introduced so-called 
`{\it many-to-one}' representation of the solution to \eqref{mild} compare to those based on the classical representation of the underlying physical process. We will exploit the probabilistic perspective that has only recently been developed in the accompanying papers \cite{MultiNTE, SNTE, SNTE-II}. In particular, we will use the probabilistic perspective to build up complexity analysis of each of the Monte Carlo methods that we consider; considering the balance of computational cost against accuracy.

\smallskip

The remainder of the paper will proceed as follows. In Section \ref{MCsect} we show how classical methods use the NBP representation  to simulate the lead eigenvalue $\lambda_*$, as well as the new NRW representation of $(\psi_t, t\geq 0)$. In Sections \ref{complexitySection} and \ref{NRWMCCA}, we provide a complexity analysis of the two respective methods, which we believe does not currently exist in the literature. From here, a mixture of conclusions can be drawn concerning the relative benefits of the two methods. In Section \ref{importance} we show how the NRW Monte Carlo method can be adapted using a method of Doob $h$-transforms, which is a stochastic process version of the method of importance sampling for random variables. This leads to the so-called $h$-NRW. In Section \ref{h-complexity} we consider the complexity of the natural Monte Carlo method associated with the $h$-NRW. Additionally we introduce a candidate class of functions to play the role of $h$, called {\it Urts} functions. Finally, in Section \ref{discussion} we discuss the relative benefits of each method and look at some numerical experiments. Almost all of the proofs are deferred to the Appendix.

\section{Basic Monte Carlo}\label{MCsect}
In Appendix A we outline two algorithmic approaches (Algorithms \ref{A1} and \ref{A2}) for the Monte Carlo simulation of an $\alpha\pi$-NRW and a $(\sigma_{\texttt s}, \pi_{\texttt s}, \sigma_{\texttt f}, \pi_{\texttt f})$-NBP, which were denoted $((\texttt{r}_t, \texttt{v}_t), t\leq \texttt{t}_{\rm end})$ and $(\mathcal{X}_t, t\geq 0)$, respectively. With either of these simulators in hand, it is possible to build simple statistical estimators of the eigen-triple $(\lambda_*, \varphi, \tilde\varphi)$ based on certain easy limiting procedures that can be observed from Theorem \ref{CVtheorem}. For practical reasons noted below, we will in fact focus  on estimation of $\me^{\lambda_{\ast}}$, rather than $\lambda_{\ast}$ directly.
\smallskip

Below we introduce such statistical estimators. Notationally we use $\Psi_k[g](t, r, \upsilon)$ to mean one of two statistics that can be built from either a NBP or NRW simulator, where $t\geq0$, $r\in D$ and $\upsilon\in V$. In the case of NBP, we understand that either 
\begin{equation}
\Psi_k[g](t, r, \upsilon) = \Psi^{\texttt{br}}_k[g](t, r, \upsilon) = \frac{1}{k}\sum_{i= 1}^k \langle g,  \mathcal{X}^i_t\rangle,
\label{def: Psi_br}
\end{equation}
where $\mathcal{X}^i$ are independent simulations of the NBP described in Algorithm \ref{A2}, or 
\begin{equation}
\Psi_k[g](t, r, \upsilon) = \Psi^{\texttt{rw}}_k[g](t, r, \upsilon) =\frac{1}{k}\sum_{i= 1}^k  {\rm e}^{\int_0^t\beta(\texttt{r}_s^i, \texttt{v}_s^i)\D s}g(\texttt{r}_t^i, \texttt{v}_t^i) \mathbf{1}_{(t <\texttt{t}_{\rm end}^i)},
\label{Psirw}
\end{equation}
where $((\texttt{r}_t^i, \texttt{v}_t^i), t\leq \texttt{t}_{\rm end}^i)$ are independent simulations of the $\alpha\pi$-NRW described in Algorithm \ref{A1}. Note that the integral $\int_0^t\beta(\texttt{r}_s^i, \texttt{v}_s^i)\D s$ must be computed along simulated paths. In general applications, this might be carried out using standard numerical integration routines, although in nuclear applications, the function $\beta$ will typically be piecewise constant, and so it can be feasible to use an exact computation methods by breaking the path into lengths along which $\beta$ is constant.

\subsection{Estimating the lead eigenvalue}

Taking account of Theorem \ref{CVtheorem}, Algorithm \ref{A1} can be used to make the following estimates of ${\lambda_*}$, $\varphi$ and $\tilde\varphi$, we have in both the NBP and NRW setting that\footnote{In the case where $\lambda_{\ast}$ is estimated directly, the corresponding estimate would be
  \begin{equation*}
    \lim_{t\to\infty}\lim_{k\to\infty}\frac{1}{t}\log
\Psi_k[g](t, r, \upsilon)
 = {\lambda_*},
  \end{equation*}
  however this estimate has the undesirable numerical property that for large $k,t$, there is a positive probability that $\Psi_k[g](t, r, \upsilon) = 0$, and this will cause problems when taking logarithms.}
\begin{equation}
\label{lambda}
\lim_{t\to\infty}\lim_{k\to\infty}\left(
\Psi_k[g](t, r, \upsilon)\right)^{\frac{1}{t}}
 = \me^{\lambda_*},
\end{equation}
where the limit holds for any $(r,\upsilon)\in{D}\times V$ and bounded measurable $g: D\times {V} \to [0,\infty)$ such that $\langle g, \tilde\varphi\rangle>0$.  Moreover, in the branching setting, if it is {\it a priori} known that $\lambda_*>0$, then thanks to Theorem \ref{SLLN}, if we take limits in $t$ first, there is no need to involve cycles in the Monte Carlo approach since (conditional on survival of the branching process, which is an event of positive probability)
\begin{equation}
\label{lambda2}
\lim_{t\to\infty}\left( \Psi_k[g](t, r, \upsilon)\right)^{\frac{1}{t}} = \me^{\lambda_*},
\end{equation}
almost surely. Either way, we can base a Monte Carlo approach on the fact that 
\[
\left( \Psi_k[g](t, r, \upsilon)\right)^{\frac{1}{t}} \approx \me^{\lambda_*}
\]
 for $k,t$ sufficiently large and an arbitrary initial configuration $r\in D$, $\upsilon\in V$.

\subsection{Estimating the lead eigenfunctions}

Again, thanks to Theorem \ref{CVtheorem} and Algorithm \ref{A1}, for $(r,\upsilon)$, $(r_0,\upsilon_0)\in D\times V$, in both the NBP and NRW setting we have 
\[
\lim_{t\to\infty}\lim_{k\to\infty}\frac{\Psi_k[g](t, r, \upsilon)}{\Psi_k[g](t, r_0, \upsilon_0)}=\frac{\varphi(r,\upsilon)}{\varphi(r_0,\upsilon_0) },
\]
thereby giving  an estimate for $\varphi$ at $(r,\upsilon)$ up to a multiplicative constant (consistent for all $(r,\upsilon)$ by keeping  $(r_0, \upsilon_0)$ fixed). \smallskip

Finally, for fixed $(r,\upsilon)\in  D\times V$ and bounded measurable  $g: D\times {V} \to [0,\infty)$, if we assume without loss of generality that $\langle\tilde\varphi,1 \rangle=1$, then 
\begin{equation}
\label{tildephi}
\lim_{t\to\infty}\lim_{k\to\infty}\frac{\Psi_k[g](t, r, \upsilon)}{\Psi_k[1](t, r,\upsilon)} = \langle g, \tilde{\varphi}\rangle ,
\end{equation}
for $t$ and $k$ sufficiently large, which gives an estimate for $\tilde\varphi$ at $(r,\upsilon)$ up to a multiplicative constant.
In the setting of the NBP we can again take advantage of Theorem \ref{SLLN} and note that, when $\lambda_*>0$ and extinction does not occur, 
\begin{equation}
\label{tildephibr}
\lim_{t\to\infty}\frac{\Psi^{\texttt{br}}_1[g](t, r, \upsilon)}{\Psi^{\texttt{br}}_1[1](t, r,\upsilon)} = \langle g, \tilde{\varphi}\rangle .
\end{equation}

\subsection{Occupation measure eigenfunction estimates at criticality} \label{sec:occupation}

In the setting that it is {\it a priori} known that $\lambda_*=0$, another approach to estimating $\varphi$, and more importantly $\tilde\varphi$, involves the use of an occupation measure. 
Note that, on the one hand, for $g\in L_\infty^+(D\times V)$, thanks to Theorem \ref{CVtheorem}, it is easy to show that 
\[
\lim_{t\to\infty}\frac{1}{t}\int_0^t \psi_s[g](r,\upsilon)\d s 
= \langle \tilde\varphi ,g \rangle \varphi(r,\upsilon).  
\]
This implies that  one can build an alternative estimate of $\tilde\varphi$ by monitoring the empirical occupation of simulations. Specifically, fixing $r_0\in D$, $\upsilon_0\in V$ and $g\in L^+_\infty(D\times V)$,  
\begin{equation}
\lim_{t\to\infty}\lim_{k\to\infty}\lim_{M\to\infty}\frac{1}{tM}\sum_{m = 1}^M \Psi_k[g](m t/M, r_0,\upsilon_0)
=\langle\tilde\varphi, g\rangle\varphi(r_0,\upsilon_0).
\label{lefteigenestimate}
\end{equation}

The reader will note that, if we take for example $g(r,\upsilon) = g_{D_0,V_0}(r,\upsilon) = \mathbf{1}_{(r\in D_0, \upsilon\in V_0)}$, where $D_0$ is a small region of $D$ and $V_0$ a small region of $V$,  the quantity 
\[
\frac{1}{tM}\sum_{m = 1}^M \Psi_k[g_{D_0,V_0}]({m t/M}, r,\upsilon)
\]
 (for sufficiently large $M, k, t$) provides nothing more than a normalised histogram of occupation of the path of the NBP or NRW. 

\smallskip

Strictly speaking, a fresh set of Monte Carlo cycles needs to be performed for each `pixel pair' $(D_0,V_0)$. Taking a histogram of occupation across all pixel pairs means that the Monte Carlo simulations on each pixel pair site are highly corollated. For example, in the setting of the NRW, high occupancy in one region of $D\times V$ will necessarily imply lower occupancy in other regions of $D\times V$. Hence, the natural estimator induced by \eqref{lefteigenestimate} has some immediate deficiencies. On the other hand, for the NBP, although the occupancies of individual pixel pairs are highly corollated, the phenomenon of branching ensures that the the underling stochastic process explores the space $D\times V$ relatively well over multiple Monte Carlo cycles. 

\smallskip

Figure \ref{histogram} gives an example of the occupation histogram given by $g(r,\upsilon) = \mathbf{1}_{(r\in D_0, \upsilon \in V)}$ where $D_0$ varies over small subsets of $D$. In this example, $D$ is a square, with 4 circular `fuel rods' where the fissile rates are substantially higher than in the surrounding domain. We also give a second example with $g(r,\upsilon) = \mathbf{1}_{(r\in D_0, \upsilon\in V')}$, where $D_0$ varies over small subsets of $D$ but $V'\subset V$ is a subset of velocities of $V$ (hence the apparent smear of `heat' in a particular direction in the figure).

\begin{figure}[h!]
\includegraphics[height= 5cm, width=5cm]{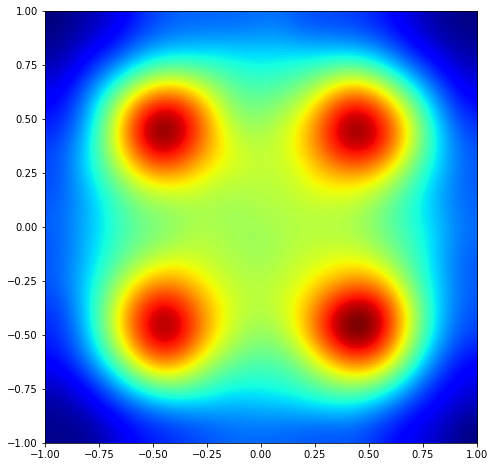}
\includegraphics[height= 5cm, width=5cm]{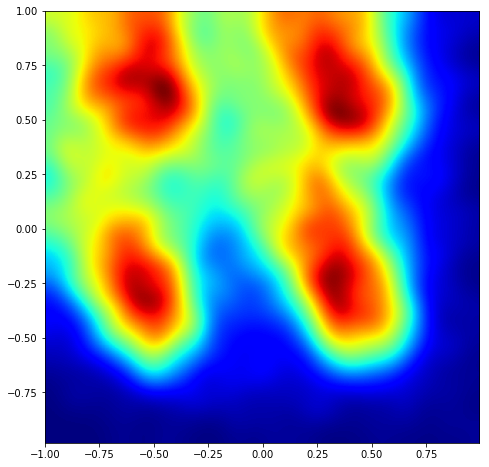}
\includegraphics[height= 5cm, width=5cm]{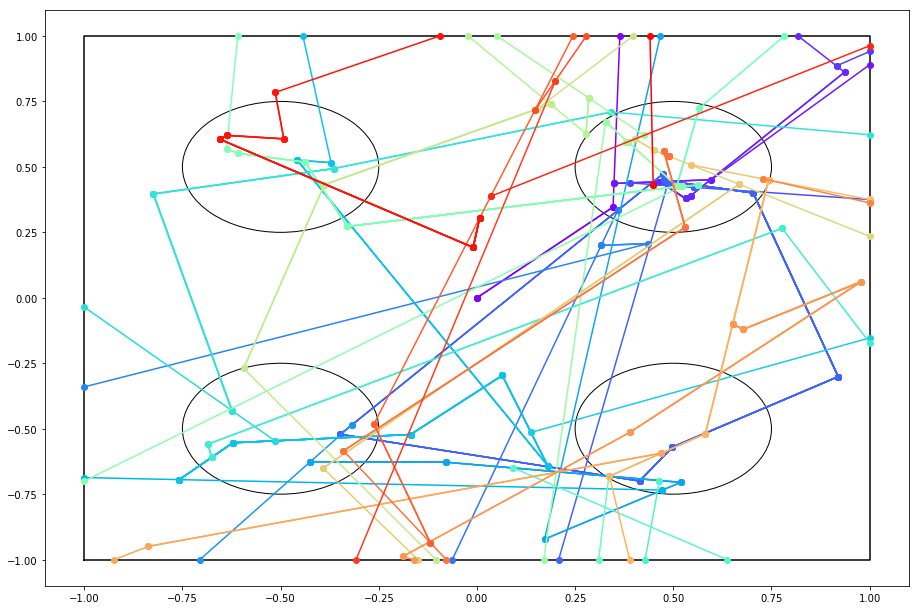}
\caption{\rm Estimates of $\tilde\varphi(r,V): = \int_V \tilde\varphi(r,\upsilon')\d\upsilon'$  represented as a `heat map' (left), $\tilde\varphi(r,V'): = \int_{V'} \tilde\varphi(r,\upsilon')\d\upsilon'$ for $V'\subseteq V$  represented as a `heat map' (middle) and some of the contributing underlying branching neutron particles (right). The domain $D$ is rectangular and the cross sections are constant except on four circular domains, where fissile rates are greater.}
\label{histogram}
\end{figure}

\section{Complexity Analysis  of Basic NBP Monte Carlo}\label{complexitySection}
We turn our attention now to analysing the complexity of some of the algorithms presented in the previous section in the setting that the chosen stochastic process is the physical NBP. Our principal focus will be on the first one, which estimates $\lambda_*$.
Our interest lies in the rate of convergence of the Monte Carlo estimates from the previous section and the computational cost associated therewith.

We start by looking at the mean-squared convergence of the Monte Carlo estimator suggested by  \eqref{lambda}. Its proof is given in Appendix D.
\begin{theo}[NBP Monte Carlo convergence for $\lambda_*$]
\label{cor:var}
There exist constants $\kappa_{[i]}:=\kappa_{[i]}(g, r, \upsilon)$ such that the following estimates hold for all $t>0$ and $k\ge 1$:
\begin{enumerate}[(i)]
\item If $\lambda_{\ast}=0$, then
\[
\mathbb{E}_{\delta_{(r,\upsilon)}} \Big[\Big(\left(\Psi^{\texttt{\emph{br}}}_k[g](t, r, \upsilon)\right)^{1/t}-\me^{\lambda_{\ast}}\Big)^{2}\Big]\le \frac{\kappa_{[1]} t}{k} +  \frac{\kappa_{[0]}}{t^{2}}\,.
\]
\item
If $\lambda_{\ast}> 0$, then 
\[
\mathbb{E}_{\delta_{(r,\upsilon)}} \Big[\Big(\left(\Psi^{\texttt{\emph{br}}}_k[g](t, r, \upsilon)\right)^{1/t}-\me^{\lambda_{\ast}}\Big)^{2}\Big]\le \frac{\kappa_{[2]}}{k}+\frac{\kappa_{[0]}}{t^2}\,.
\]
\item
If $\lambda_{\ast} < 0$, then
\[
\mathbb{E}_{\delta_{(r,\upsilon)}}\Big[\Big(\left(\Psi^{\texttt{\emph{br}}}_k[g](t, r, \upsilon)\right)^{1/t}-\me^{\lambda_{\ast}}\Big)^{2}\Big]\le  \frac{\kappa_{[3]} \me^{-\lambda_{\ast}t}}{k}+\frac{\kappa_{[0]}}{t^{2}}\,.
\]
\end{enumerate}
Let $C_{[0]}:= \langle \tilde \varphi, g\rangle \varphi(r,\upsilon)$, and $C_{[i]}$ are given in Lemma~\ref{lem:NBPConv}. If $C_{[0]} \neq 1$, the inequalities hold for all $t,k$ sufficiently large when the constants are chosen, for any $\eta>0$, to be:
\begin{align*}
  \kappa_{[0]} & := 2 \left[\ln\left(C_{[0]}\right) \me^{\lambda_\ast}\right]^2 + \eta  \\
  \kappa_{[i]} & := 2 C_{[i]} C_{[0]}^{-2} + \eta  \qquad \qquad \qquad  i = 1,2,3.
\end{align*}
In particular, this implies that for each $(r, \upsilon)$ and $g\in L^+_{\infty}(D\times V)$, there exists a sequence $k(t)$ increasing in $t$ such that $k(t) \to \infty$ and such that
\begin{equation}
\label{cv-L2}
\lim_{t\to\infty} \mathbb{E}_{\delta_{(r,\upsilon)}} \Big[\Big(\left(\Psi^{\texttt{\emph{br}}}_{k(t)}[g](t, r, \upsilon)\right)^{1/t}-\me^{\lambda_{\ast}}\Big)^{2}\Big] = 0 
\end{equation}
and the Monte Carlo estimator suggested by \eqref{lambda} is asymptotically unbiased with mean-squared convergence.
\end{theo}

The above estimates suggest that in the critical and sub-critical cases there is a trade-off between increasing $t$ to reduce the second term on the right-hand side of the inequality, and an increasing effect on the first term. By also increasing $k$, it is possible to make the error sufficiently small. 
\smallskip

We also want to understand the required computational cost for the associated degree of accuracy in the preceding theorem.  To compute the cost of the simulation, we need an indexation for all the particles which have been present in the system.

Suppose that $\mathcal U$ is a labelling system that uniquely identifies (or `names') neutrons that ever come into existence in the NBP (for example one could use e.g. Ulam Harris labelling, cf. \cite{LeGall}).
For each $u\in \mathcal U$, write $(r^{(u)}_t, \upsilon^{(u)}_t)_{b_u\le t<d_u}$ for the path of the particle $u$, where $0\le b_u<d_u\le \infty$ stand for the birth and death times of $u$. Recall that $r^{(u)}$ is piecewise linear; write $b_u= t^{(u)}_0 < t^{(u)}_1 < \cdots $ for the jump times of $\upsilon^{(u)}_t$ (the moments when scatterings occur). 
For $f, g\in L^{+}_{\infty}(D\times V)$, define the cost per simulation by
\[
C_t[f,g]:=\sum_{u\in \mathcal U}g\Big(r^{(u)}_{b_u},  \upsilon^{(u)}_{b_u}\Big) \un_{(b_u\le t)}+ \sum_{u\in \mathcal U}\sum_{i\ge 1}f\Big(r^{(u)}_{t^{(u)}_i},  \upsilon^{(u)}_{t^{(u)}_i}\Big)\un_{(b_{u}\le t, t^{(u)}_i\le t)},
\]
for $t\geq 0.$
\smallskip

The major part of the MC simulation suggested by \eqref{lambda} consists of simulating the variables relevant to the scattering and fission events. The  random functional $(C_t[f,g], t\geq0)$, with an appropriate choice of $f, g$, can be then used to evaluate the cost (if we assume the computational effort of simulating any random variable is counted as one unit) of a single simulation with a time horizon $t$. For instance, if we take $f=\mathbf 1$ and $g=\mathbf 0$, then $C_t[\mathbf 1, \mathbf 0]$ counts the total number of scattering events up to time $t$, which can be used as a rough estimate of the CPU-time of the program; on the other hand, taking $f=\mathbf 0, g=\mathbf 1$ yields an estimate for the memory cost as $C_t[\mathbf 0, \mathbf 1]$ counts the number of particles that have appeared before $t$.  
\smallskip

The process $(C_{t}[f,g], t\geq 0)$ is a pure jump increasing process whose Doob--Meyer decomposition can be written as $C[f,g]=A+M$, where $M$ is a martingale with respect to the natural filtration of $X$ and $A$ is the bounded variation compensator of $C[f,g]$. The latter can be easily computed noting that rates at which scattering and fission occurs are given by the cross sections, it is straightforward to derive 
\begin{equation}\label{eq:rate}
A_t = \int_0^t \langle \sigma_{\texttt{s}}\pi_{\texttt{s}}[f]+\sigma_{\texttt{f}}\pi_{\texttt{f}}[g], X_{s}\rangle \d s, \qquad t\geq 0,
\end{equation}
where for $(r, \upsilon)\in D\times V$ and $h\in L^+_\infty(D\times V)$, 
\begin{equation}
\pi_{\texttt{s}}[h](r, \upsilon)=\int_V h(r, \upsilon')\pi_{\texttt{s}}(r, \upsilon, \upsilon')\d \upsilon' \quad \text{and} \quad \pi_{\texttt{f}}[h](r, \upsilon)=\int_V h(r, \upsilon')\pi_{\texttt{f}}(r, \upsilon, \upsilon')\d \upsilon'. 
\label{pih}
\end{equation}
Such computations are not unfamiliar in the setting of branching Markov processes, superprocesses or fragmentation processes. See for example \cite{E, Li, Bertbook2}. We are led very easily to the following estimate of simulation cost per sample as a simple application of Theorem \ref{CVtheorem}.

Doob's Maximal inequality for the representation $C[f,g] = M+A$ also gives, for example,
\begin{equation}\label{eq: qv}
\mathbb E_{\delta_{(r, \upsilon)}}\big[\sup_{s\le t}\big(C_\emph{s}[f, g]-A_s\big)^{2}\big]   \le 4\int_0^t \psi_s[\sigma_{\texttt{\emph{s}}} \pi_{\texttt{\emph{s}}}[f^2]+\sigma_{\texttt{{f}}} \pi_{\texttt{{f}}}[g^2]](r, \upsilon)\d s, \qquad t\geq 0.
\end{equation}
This shows that the mean squared error of the growth of the cost function $C[f,g]$ relative to the compensator $A$  mimics the growth of 
\[
\mathbb E_{\delta_{(r, \upsilon)}}\big[C_{t}[f, g]\big] = \mathbb E_{\delta_{(r, \upsilon)}}[A_t] =\int_0^t \psi_s[\sigma_{\texttt{\emph{s}}} \pi_{\texttt{\emph{s}}}[f]+\sigma_{\texttt{{f}}} \pi_{\texttt{{f}}}[g]](r, \upsilon)\d s,\qquad t\geq 0.
\]
 \begin{lem}[NBP Expected simulation cost per sample]\label{thm: cost}
For $f, g\in L^+_\infty(D\times V)$
we have the following.
\begin{enumerate}[(i)]
\item
If $\lambda_\ast = 0$, then, as $t\to\infty$,
\begin{equation}\label{eq: cost_asy_cr}
\mathbb E_{\delta_{(r, \upsilon)}}\big[C_{t}[f, g]\big] {\sim} \langle \sigma_{\texttt{\emph{s}}} \pi_{\texttt{\emph{s}}}[f]+\sigma_{\texttt{\emph{f}}} \pi_{\texttt{\emph{f}}}[g], \,\tilde\varphi\rangle\varphi(r, \upsilon)   t =: \kappa_{[4]} t\,,
\end{equation}
\item
If $\lambda_\ast > 0$, then, as $t\to\infty$,
\begin{equation}\label{eq: cost_asy_sur}
\mathbb E_{\delta_{(r, \upsilon)}}\big[C_{t}[f, g]\big] {\sim} \langle \sigma_{\texttt{\emph{s}}} \pi_{\texttt{\emph{s}}}[f]+\sigma_{\texttt{\emph{f}}} \pi_{\texttt{\emph{f}}}[g], \,\tilde\varphi\rangle\varphi(r, \upsilon)   \frac{\me^{\lambda_{\ast}t}}{\lambda_{\ast}} = \kappa_{[4]} \frac{\me^{\lambda_{\ast}t}}{\lambda_{\ast}} \,,
\end{equation}
\item
If $\lambda_\ast < 0$, then, as $t\to\infty$, $\mathbb E_{\delta_{(r, \upsilon)}}\big[C_{t}[f, g]\big]\nearrow \mathbb E_{\delta_{(r, \upsilon)}}\big[C_{\infty}[f, g]\big] =: \kappa_{[5]}<\infty$. 
\end{enumerate}
\end{lem}

\medskip
Of course, if we run $k$ cycles of Algorithm \ref{A2}, the cost of these cycles is then a $k$-multiple of that of a single cycle as stipulated in Lemma~\ref{thm: cost}. It is also worth noting that both $ \sigma_{\texttt{\emph{s}}} \pi_{\texttt{\emph{s}}}[f]+\sigma_{\texttt{\emph{f}}} \pi_{\texttt{\emph{f}}}[g]<\infty$ and $ \sigma_{\texttt{\emph{s}}} \pi_{\texttt{\emph{s}}}[f^2]+\sigma_{\texttt{\emph{f}}} \pi_{\texttt{\emph{f}}}[g^2]<\infty$, for $f, g\in L^+_\infty(D\times V)$.

\medskip
We conclude with the following consequence of Theorem~\ref{cor:var} and Lemma~\ref{thm: cost}. Specifically, we want to identify the optimal choice of $k,t$ to minimise total computational cost for  a given level of accuracy. To this end,  we define 
\[
\texttt{Cost}(k, t) = k\times\sup_{f, g\in L_\infty^+(D\times V): \norm{f}_\infty, \norm{g}_\infty\leq 1}C_t[f, g]
\]
\begin{theo}[NBP Monte Carlo complexity] \label{thm:NBPComplexity}
  There is a choice of $k,t$ such that
  \begin{equation} \label{eq:lesseps}
    \mathbb{E}_{\delta_{(r,\upsilon)}} \Big[\Big(\left(\Psi^{\texttt{\emph{br}}}_k[g](t, r, \upsilon) \right)^{1/t}- \me^{\lambda_{\ast}}\Big)^{2}\Big]\le \varepsilon^2
  \end{equation}
  with the following total simulation cost:
  \begin{enumerate}[(i)]
  \item If $\lambda_{\ast} = 0$ then \eqref{eq:lesseps} holds in the limit as $\varepsilon \to 0$ with
    \begin{equation*}
      \mathbb E_{\delta_{(r, \upsilon)}}\big[\emph{\texttt{Cost}}(k, t)\big]  \le C \varepsilon^{-4},
    \end{equation*}
    for some $C$, and the optimal choice of $k,t$ are given by $t = \sqrt{2\kappa_{[0]}} \varepsilon^{-1}$ and $k = \kappa_{[0]} \kappa_{[1]}^{-1}t^3$.
  \item If $\lambda_{\ast} > 0$ then \eqref{eq:lesseps} holds in the limit as $\varepsilon \to 0$ with
    \begin{equation*}
      \mathbb E_{\delta_{(r, \upsilon)}}\big[\emph{\texttt{Cost}}(k,
      t)\big]  \le C \exp\left(\lambda_{\ast} \sqrt{\kappa_{[0]}} \varepsilon^{-1}\right),
    \end{equation*}
    for some $C$, and the optimal choice of $k,t$ corresponds to $t \approx \sqrt{\kappa_{[0]}} \varepsilon^{-1}$ and $k \approx  \lambda_{\ast} \kappa_{[0]}^{1/2}\kappa_{[2]} \varepsilon^{-3} /2$.
  \item If $\lambda_{\ast} < 0$ then \eqref{eq:lesseps} holds in the limit as $\varepsilon \to 0$ with
    \begin{equation*}
      \mathbb E_{\delta_{(r, \upsilon)}}\big[\emph{\texttt{Cost}}(k, t)\big]  \le C \exp\left(|\lambda_{\ast}| \sqrt{\kp{0}} \varepsilon^{-1}\right),
    \end{equation*}
    for some $C$, and the optimal choice of $k,t$ corresponds to $t \approx \sqrt{\kp{0}} \varepsilon^{-1}$ and $k \approx |\lambda_{\ast}| \kappa_{[0]}^{1/2}\kappa_{[3]} \varepsilon^{-3} \exp(|\lambda_{\ast}| t)/2$.
  \end{enumerate}
\end{theo}

\begin{proof}
(i) Following Theorem~\ref{cor:var} and Lemma~\ref{thm: cost}, it follows that we want to solve the optimisation problem:
    \begin{equation*}
      \text{minimize: } \quad tk \quad \text{ subject to: } \quad \frac{t \kappa_{[1]} }{k} +
      \frac{\kappa_{[0]}}{t^2}\le \varepsilon^2, \quad k, t > 0.
    \end{equation*}
    Introducing a Lagrangian variable $\xi \ge 0$ to penalise the constraint, we see that we expect the minimum to occur when
    \begin{align*}
      k + \frac{\kappa_{[1]} \xi}{k} - 2 \kappa_{[0]} \xi t^{-3} = 0, \qquad
      t - \frac{t \kappa_{[1]}\xi}{k^2} = 0.
    \end{align*}
    Rearranging the second expression, and noting that $t>0$, we deduce $k^2 = \kappa_{[1]}\xi$, and hence substituting back into the first expression, that $k = \kappa_{[1]} \kappa_{[0]}^{-1} t^3$. Substituting this into the constraint, we deduce that $t = \sqrt{2\kappa_{[0]}} \varepsilon^{-1}$, and the expression for the optimal cost follows immediately.

\medskip

(ii)  Similarly to the previous case, we need to solve the problem:
    \begin{equation*}
      \text{minimize: } \quad \me^{\lambda_{\ast}t} k \quad \text{ subject to: } \quad \frac{\kappa_{[2]}}{k} +
        \frac{\kappa_{[0]}}{t^2}\le \varepsilon^2, \quad k, t > 0.
    \end{equation*}
    Proceeding as above to introduce a Lagrangian penalisation, and optimising over $t$ and $k$, we deduce that $2 \xi \kappa_{[0]} = \lambda_{\ast} k t^3 \me^{\lambda_{\ast}t} = 2 \kappa_{[0]} \kappa_{[2]}^{-1}k^2 \me^{\lambda_{\ast}t}$, and hence $k = \kappa_{[2]} \kappa_{[0]}^{-1}\lambda_{\ast} t^3/2$. Substituting into the constraint, we deduce that the optimal $t$ satisfies $t^3 \varepsilon^2 - \kappa_{[0]} t - 2 \kappa_{[0]} \lambda_{\ast}^{-1} = 0$. It follows for $\varepsilon$ sufficiently small that $t \le \sqrt{\kappa_{[0]}} \varepsilon^{-1}+\eta$ for some $\eta >0$. Recalling that $\kappa_{[0]}$ is already defined only up to some small constant, we can assume $\eta=0$ and absorb an $\varepsilon^{-3}$ term into the same constant. The remainder of the argument then follows directly as above.

\medskip

(iii) The case where $\lambda_{\ast} < 0$ follows with an almost
    identical argument; in this case the optimisation problem is:
    \begin{equation*}
      \text{minimize: } \quad k \quad \text{ subject to: } \quad \frac{\kp{3} \me^{-\lambda_{\ast}t} }{k} +
        \frac{\kp{0}}{t^2} \le \varepsilon^2, \quad k, t > 0,
    \end{equation*}
    and the answer follows from an essentially identical set of steps to (ii).
\end{proof}

\section{Complexity Analysis  of Basic NRW Monte Carlo}\label{NRWMCCA}
Next we turn our attention to the setting of the Monte Carlo simulation based on the NRW estimator, in other words, via the representation \eqref{phi}. The picture in this setting is typically less competitive than in the the NBP estimator setting, however, this approach will offer additional flexibility in terms of possible algorithmic approaches to finding $\lambda_*$. Recall that $\overline{\beta} := \textstyle{\sup_{r\in D, \upsilon\in V}\beta(r,\upsilon)}$, and introduce also $\underline{\beta} := \textstyle{\inf_{r\in D, \upsilon\in V}\beta(r,\upsilon)}$.

\begin{theo}[NRW Monte Carlo convergence for $\lambda_*$]\label{NRWvar}
For all $\lambda_*\in\R$ and $g\in L^+_\infty(D\times V)$ there exists $\kpt{1} >0$ and $\lambda_1 \in \R$ such that  
\[
\mathbf{E}_{(r,\upsilon)} \Big[\Big(\left(\Psi^{\texttt{\emph {rw}}}_k[g](t, r, \upsilon)\right)^{1/t}-\me^{\lambda_{\ast}}\Big)^{2}\Big]\le
\kpt{1}\frac{1}{k}{\rm e}^{(\lambda_1 -2\lambda_*)t} + \frac{\kp{0}}{t^2}, 
\]
as $t\to\infty$. Moreover, $(2\lambda_* \vee( \lambda_* + \underline\beta)) \leq \lambda_1 \leq \lambda_* + \overline\beta$.
\end{theo}
It is immediately obvious from the above result that, since $\lambda_1-2\lambda_*\geq 0$ there is no setting in which the NRW Monte Carlo offers better asymptotic accuracy that the NBP Monte Carlo algorithm does. The reason for this is quite simple. It is a straightforward consequence of the previously observed fact that a large majority of simulations for the NRW Monte Carlo estimator will return zero contribution to the estimate on account of the indicator in \eqref{phi}.

The constant $\kpt{1}$ can be computed using convergence results in the appendix (see \eqref{prime}) as $2 \me^{2\lambda_{\ast}}\left( \varphi_1(r,v) \langle \tilde{\varphi}_1,g^2\rangle - (\varphi(r,v) \langle \tilde{\varphi},g\rangle)^2 \mathbf{1}_{\lambda_1 = 2 \lambda_{\ast}}\right)/C_{[0]}^2$. Here $\varphi_1, \tilde{\varphi}_1$ are analogues of $\varphi, \tilde{\varphi}$ for a related semigroup specified in the appendix.

\smallskip

In terms of  computational cost, it is again a simple exercise to identify the analogue of Lemma \ref{thm: cost}. 
In this setting we need only define the cost of each simulation by the function 
\begin{equation}
C_t[f] = \sum_{j} f(R_{t_j},\Upsilon_{t_j})\mathbf{1}_{(t_j \leq t)}, \qquad t\geq 0,
\label{RWcostfunction}
\end{equation}
where $0\leq t_0< t_1<\cdots$ are the scatter times of the process $(R,\Upsilon)$ and $f\in L^{+}_{\infty}(D\times V)$. As in the setting of Lemma \ref{thm: cost}, we can write $C[f]=A+M$, where $M$ is a martingale with respect to the natural filtration of $(R,\Upsilon)$ and $A$ is the bounded variation compensator of $C[f]$. As before, the scattering rates of $(R,\Upsilon)$ together with standard computations allow us to write, for $ t\geq 0$,
\begin{equation}
A_t = \int_0^{t\wedge \tau^D}  \alpha\pi[f](R_s, \Upsilon_s)\d s
, \qquad t\geq 0,
\label{AcostRW}
\end{equation}
where $\pi[h]$ is defined similarly to \eqref{pih}.

\smallskip

On account of the fact that 
\[
\alpha(r,\upsilon) = \sigma_{\texttt{s}}(r,\upsilon)\int_V \pi_{\texttt{s}}(r, \upsilon, \upsilon')\d \upsilon' +
\sigma_{\texttt{f}}(r,\upsilon) \int_V \pi_{\texttt{f}}(r, \upsilon, \upsilon')\d \upsilon',
\]
we can use the assumptions (H1) and (H2) to deduce that 
\[
\sup_{r\in D, \upsilon\in V} \alpha(r,\upsilon) \leq \overline c,
 \]
 for constant $0<
 \overline{c}<\infty$. We can think of the trajectory of process $(R,\Upsilon)$ as a series of randomly placed sticks laid end to end in the domain $D$. Suppose that $S_1,S_2, \cdots, S_N$ are the successive lengths of these sticks, where $N$ is the last stick which touches the boundary of $D$ (with the understanding that $N = \infty$ if this does not happen).  The estimate above ensures that we can construct a sequence of iid random variables, which are exponentially distributed with rate $1$, say ${\bf e}_1,{\bf e}_2,\cdots$, such that $S_i\geq \upsilon_{\texttt{min}}{\bf e}_i/\overline{c}$, for $i = 1,\cdots, N$. 
 It now follows that  the probability that the length of any given  stick is bounded below by ${\rm diam}(D) = \sup\{|x-y|: x,y\in \partial D\}$ is uniformly bounded below by $\exp(- {\rm diam}(D)\overline{c}/\upsilon_{\texttt{min}})$.  Accordingly, it follows that $N$ is upper bounded by $\Gamma$, where $\Gamma$ is a geometrically distributed random variable with parameter $\exp(- {\rm diam}(D)\overline{c}/\upsilon_{\texttt{min}})$. Moreover, we also have that 
 \[
 \tau^D \leq \Gamma {\rm diam}(D)/\upsilon_{\texttt{min}},
 \]
so that, in particular, $\sup_{r\in D, \upsilon\in V}\mathbb{E}_{\delta_{(r, \upsilon)}}[\tau^D]<\infty$.
We thus have the following result for the computational cost.

 \begin{lem}[NRW Expected simulation cost]\label{thm: costrw}
 For $f\in L^+_\infty(D\times V)$
\[
\lim_{t\to\infty}\mathbf{E}_{(r,\upsilon)}\big[C_{t}[f]\big] =  \mathbf{E}_{(r,\upsilon)}\left[ \int_0^{ \tau^D}  \alpha\pi[f](R_s, \Upsilon_s)\d s\right]<\infty.
\]
 \end{lem}
 
 We should of course note again that, if there are $k$ cycles in the Monte Carlo simulation, then the cost grows no faster than proportional to $k$, irrespective of the time horizon.

\smallskip

As with the previous section, we conclude with a corollary which gives the complexity of the Monte Carlo algorithm. On this occasion 
\[
\texttt{Cost}(k, t) = k\times\sup_{f\in L_\infty^+(D\times V): \norm{f}_\infty\leq 1}C_t[f],
\]
for $k\geq 1$, $t\ge0$.

\begin{cor}[NRW Monte Carlo complexity] \label{cor:NRWComplex} There is a choice of $k,t$ such that
  \begin{equation} \label{eq:Mex}
    \mathbb{E}_{\delta_{(r,\upsilon)}} \Big[\Big(\left(\Psi^{\texttt{\emph{rw}}}_k[g](t, r, \upsilon)\right)^{1/t}-\me^{\lambda_{\ast}}\Big)^{2}\Big]\le \varepsilon^2.
  \end{equation}
  Moreover, \eqref{eq:Mex} holds in the limit as $\eps \to 0$ with 
  \begin{equation*}
    \mathbb E_{\delta_{(r, \upsilon)}}\big[\emph{\texttt{Cost}}(k, t)\big]  \le C \exp\left( (\lambda_1-2\lambda_{\ast}) \sqrt{\kappa_{[0]}} \varepsilon^{-1}\right),
  \end{equation*}
  for some $C>0$, and the optimal choice of $k,t$ corresponds to $t \approx \sqrt{\kp{0}} \varepsilon^{-1}$ and $k \approx (\lambda_1-2\lambda_{\ast}) \sqrt{\kp{0}} \kpt{1} \varepsilon^{-3} \exp((\lambda_1-2\lambda_{\ast}) t)/2$.
 
\end{cor}

\begin{proof}
From Lemma \ref{thm: costrw} we have, for all $t>0$,
$
 \mathbb E_{\delta_{(r, \upsilon)}}\big[{\texttt{Cost}}(k, t)\big] \leq C' k,
$
for some constant $C'$, and hence the objective is to:
\begin{equation*}
  \text{minimize: } \quad k \quad \text{ subject to: } \quad \frac{\kpt{1}\me^{(\lambda_1-2\lambda_{\ast}) t} }{k} +
    \frac{\kp{0}}{t^2} \le \varepsilon^2, \quad k, t > 0.
\end{equation*}
This is essentially identical to the minimisation carried out in (iii) of Theorem~\ref{thm:NBPComplexity}, and the result follows.
\end{proof}

\section{Importance sampling with NRW Monte Carlo}\label{importance}

One common method for improving the efficiency of Monte Carlo estimation schemes is through the use of importance sampling. Recalling, for example, \eqref{Psirw}, the term ${\rm e}^{\int_0^t\beta(\texttt{r}_s^i, \texttt{v}_s^i)\D s}\mathbf{1}_{\{t <\texttt{t}_{\rm end}^i\}}$ in the estimator can be interpreted as the weight of the $i$th particle. In particular, considering the simple case where $g \equiv 1$, one wants to apply importance sampling to reduce the variance of this particle weight as much as possible, by simulating paths from a suitably chosen probability distribution.

In this context, we will show that the correct way to think about importance sampling in the case of the NTE is through the use of Doob's $h$-transform. In this section, we show how to $h$-transform the NRW, and prove that under mild assumptions, this gives an unbiased estimate of the semi-group $(\psi_t; t \ge 0)$.

\smallskip


Appealing to the Markov property together with Theorem \ref{CVtheorem}, suppose we write $\mathcal{G}_t = \sigma((R_s, \Upsilon_s), s\leq t)$ for $t\geq 0$. Then, for $0\leq s\leq t<\infty$,
\begin{align}
&\mathbf{E}_{(r,\upsilon)}
\left[\left.{\rm e}^{-{\lambda_*} t}{\rm e}^{\int_0^t   \beta(R_u,\Upsilon_u)\d u}\frac{\varphi(R_t, \Upsilon_t)}{\varphi(r,\upsilon)}\mathbf{1}_{( t<\tau^{D})}\right|\mathcal{G}_s\right]\notag\\
& = \mathbf{1}_{( s<\tau^{D})}
{\rm e}^{-{\lambda_*} s}{\rm e}^{\int_0^s   \beta(R_u,\Upsilon_u)\d u}\notag\\
&\hspace{2cm}\times\mathbf{E}_{(r',\upsilon')}
\left[{\rm e}^{-{\lambda_*} (t-s)}{\rm e}^{\int_0^{t-s}   \beta(R_u,\Upsilon_u)\d u}\frac{\varphi(R_{t-s}, \Upsilon_{t-s})}{\varphi(r,\upsilon)}\mathbf{1}_{( t-s<\tau^{D})}\right]_{(r', \upsilon') = (R_s,\Upsilon_s)}\notag \\
&= \mathbf{1}_{( s<\tau^{D})}
{\rm e}^{-{\lambda_*} s}{\rm e}^{\int_0^s   \beta(R_u,\Upsilon_u)\d u} {\rm e}^{-{\lambda_*} (t-s)}
\psi_{t-s}[\varphi](R_s, \Upsilon_s)/\varphi(r,v)\notag\\
&=\mathbf{1}_{( s<\tau^{D})}
{\rm e}^{-{\lambda_*} s}{\rm e}^{\int_0^s   \beta(R_u,\Upsilon_u)\d u}\varphi(R_s, \Upsilon_s)/\varphi(r,v).
\label{martingale}
\end{align}
Together with the boundedness of $\varphi$ and the Markov property this tells us that  
\[
{\rm e}^{-{\lambda_*} t}{\rm e}^{\int_0^t   \beta(R_u,\Upsilon_u)\d u}\frac{\varphi(R_t, \Upsilon_t)}{\varphi(r,\upsilon)}\mathbf{1}_{( t<\tau^{D})}, \qquad t\geq 0,
\]
is a martingale.  

\smallskip

As such we can use this martingale to effect a change of measure on the path space of $(R, \Upsilon) = ((R_t, \Upsilon_t), t\geq 0)$ via the relation 
\begin{equation}\label{COM}
\mathbf{P}^\varphi_{(r,\upsilon)}(A, \, t<\tau^{D}) = \mathbf{E}_{(r,\upsilon)}\left[{\rm e}^{-{\lambda_*} t}{\rm e}^{\int_0^t   \beta(R_s,\Upsilon_s)\d s}\frac{\varphi(R_t, \Upsilon_t)}{\varphi(r,\upsilon)}\mathbf{1}_{(A, \, t<\tau^{D})}\right] = {\rm e}^{-{\lambda_*} t}\frac{\psi_t[\varphi \mathbf{1}_A](r,\upsilon)}{\varphi(r, \upsilon)}
\end{equation}
for all $(r,\upsilon)\in D\times V$ and $A\in\sigma((R_s,\Upsilon_s), s\leq t)$. The theory of Doob $h$-transforms dictates that the process $(R, \Upsilon)$ under $\mathbf{P}^\varphi_{(r,\upsilon)}$, $r\in D$, $\upsilon\in V$, is again a Markov process, see \cite{CW}. It is conservative only if we can prove that 
$\mathbf{P}^\varphi_{(r,\upsilon)} (\tau^{D} =\infty)=1$, for all $(r,\upsilon)\in D\times V$, which is to say, the process $(R,\Upsilon)$ is trapped in $D$. 
However this is easily verified from \eqref{COM} by taking $A$ to be the entire sample space and invoking the martingale property to deduce that $\mathbf{P}^\varphi_{(r,\upsilon)} (t<\tau^{D} )=1$, for all $t\geq 0$.
Moreover, since $\tilde\varphi, \varphi\in L^+_\infty(D\times V)$, we can  normalise $\tilde\varphi$ and $\varphi$ so that $\langle\tilde\varphi, \varphi\rangle = 1$, for bounded and measurable $g: \bar D\times V\to [0,\infty)$ we have
\[
\lim_{t\to\infty}\mathbf{E}^\varphi_{(r,\upsilon)}[g(R_t,\Upsilon_t)] =\lim_{t\to\infty} {\rm e}^{-{\lambda_*} t}\frac{\psi_t[\varphi g](r,\upsilon)}{\varphi(r, \upsilon)}
=
\langle g,  \varphi\tilde\varphi \rangle,
\]
indicating that $\varphi\tilde\varphi$ is the stationary distribution of $(R, \Upsilon)$  under $\mathbf{P}^\varphi_{(r,\upsilon)}$, $r\in D$, $\upsilon\in V$.

\smallskip

As is usual with such space-time changes of measure, the action of its infinitesimal generator, say ${\bL}_\varphi$, can be heuristically calculated by applying the $h$-transformation to the generator of the process $(R,\Upsilon)$, say $\bL$, with potential $\beta -{\lambda_*}$ under $\mathbf{P}_{(r,\upsilon)}$, $(r,\upsilon)\in \bar D\times V$.
Given the discussion in Section \ref{m21subsection}, we have the action of the generator $\bL$ satisfying 
\begin{equation}
\bL f(r,\upsilon) = \upsilon\cdot\nabla f(r,\upsilon)+ \alpha(r,\upsilon)\int_V \left( f(r,\upsilon') - f(r,\upsilon)\right) \pi(r,\upsilon,\upsilon')\d \upsilon',
\label{Ell}
\end{equation}
for all $f$ in the domain of $\bL$, written $\text{Dom}(\bL)$; e.g. we can take $f$ as continuous in $r$ and differentiable in $\upsilon$.
Hence, on $D\times V$,
\begin{align}
{\bL}_\varphi f (r,\upsilon)&= \left({\bL} + \beta -{\lambda_*} \right)^{\varphi}f(r,\upsilon)\notag\\
&= \frac{1}{\varphi(r,\upsilon)} {\bL}(\varphi f)(r,\upsilon) +\beta f(r,\upsilon) -{\lambda_*} f(r,\upsilon)\notag\\
&=\upsilon\cdot\nabla f(r,\upsilon)  +\frac{\upsilon\cdot\nabla \varphi(r,\upsilon) }{\varphi(r,\upsilon)}f(r,\upsilon) + \beta f(r,\upsilon) -{\lambda_*} f(r,\upsilon)\notag\\
&\hspace{1cm}+ \frac{\alpha(r,\upsilon)}{\varphi(r,\upsilon)}\int_V\left( f(r, \upsilon') - f(r,\upsilon)\right) \varphi(r,\upsilon') \pi(r,\upsilon,\upsilon')\d\upsilon'\notag\\
&\hspace{1cm}+f(r,\upsilon) \frac{\alpha(r,\upsilon)}{\varphi(r,\upsilon)} \int_V\left( \varphi(r,\upsilon') - \varphi(r,\upsilon) \right) \pi(r,\upsilon,\upsilon')\d\upsilon'\notag\\
&=\upsilon\cdot\nabla f(r,\upsilon)   + \frac{({\bL}+\beta-{\lambda_*})\varphi(r,\upsilon)}{\varphi(r,\upsilon)}f(r,\upsilon) \notag\\
&\hspace{1cm}+\alpha(r,\upsilon)\int_V\left( f(r, \upsilon') - f(r,\upsilon)\right)\frac{ \varphi(r,\upsilon')}{\varphi(r,\upsilon)} \pi(r,\upsilon,\upsilon')\d\upsilon'
\label{againwithh}
\end{align}
for all bounded and continuous  $f\in \text{Dom}(\bL)$. From Theorem \ref{CVtheorem}, we can interpret \eqref{leftandright} as heuristically meaning that  $(\bL + \beta-\lambda_*)\varphi=0$. Taking this on face value, we get 
\[
{\bL}_\varphi f (r,\upsilon)=\upsilon\cdot\nabla f(r,\upsilon)  +\int_V\left( f(r, \upsilon') - f(r,\upsilon)\right)\alpha(r,\upsilon)\frac{ \varphi(r,\upsilon')}{\varphi(r,\upsilon)} \pi(r,\upsilon,\upsilon')\d \upsilon'.
\]
In other words, ${\bL}_\varphi$ corresponds to an $(\alpha_\varphi, \pi_\varphi)$-NRW, where, for $r\in D$ and $\upsilon\in V$, 
\[
\alpha_\varphi(r,\upsilon) = \alpha(r,\upsilon)\int_V\frac{ \varphi(r,\upsilon')}{\varphi(r,\upsilon)} \pi(r,\upsilon,\upsilon')\d \upsilon'\text{ and }
\pi_\varphi(r,\upsilon, \upsilon') = \frac{ \varphi(r,\upsilon')\pi(r,\upsilon,\upsilon')}{\textstyle{\int_V} \varphi(r,\upsilon'')\pi(r,\upsilon,\upsilon'')\d \upsilon''} 
\]

An interesting observation in light of this change of measure is that, appealing to \eqref{COM}, by choosing $g\in L_\infty^+(D\times V)$, we have that 
\begin{equation}
 \psi_t[g](r,\upsilon)=
\mathbf{E}^\varphi_{(r,\upsilon)}\left[{\rm e}^{{\lambda_*} t}\frac{\varphi(R_t,\Upsilon_t)^{-1}}{\varphi(r, \upsilon)^{-1}} g(R_t,\Upsilon_t)\right] .
\label{perfect}
\end{equation}
This leads to the thought experiment. If we knew $\varphi$ and ${\lambda_*}$ and we take $g \equiv \varphi$, then the term inside the expectation on the right-hand side is simply a constant. As a consequence, $\psi_t[\varphi](r,v)$ can be computed exactly with a single Monte Carlo trajectory. In particular, we would be interested in a guess that would inform a Doob $h$-transformed process that would produce estimates for $\psi_t[g]$, which involves averaging a low-variance term that is close to the quantity inside the expectation in \eqref{perfect}.
The problem with this thought experiment is, of course, that we are aiming to generate Monte Carlo estimates of $\psi_t[g](r,\upsilon)$ with a view to estimating the quantities ${\lambda_*}$ and $\varphi$.

\smallskip

The idea of the thought experiment is not lost however. Suppose instead we were to make an educated guess for  $\varphi$ and ${\lambda_*}$. In particular, we would be interested in a guess that would inform a Doob $h$-transformed process that we could simulate with relatively few (if any) paths that leave the domain $D$. The next theorem formalises this idea (which also gives a rigorous basis to the discussion above by setting $h= \varphi$).

\begin{theo}\label{doobhthrm}
Suppose that $h\in L^+_\infty(D\times V)$   such that $h\geq 0$ on $D\times V$,  $h|_{\partial D^+}=0$ and 
\begin{equation}
\label{h1}\inf_{r\in D}\int_V h(r,\upsilon')\d\upsilon'>0.
\end{equation} 
Define 
\begin{equation}
{\ebJ} h(r,\upsilon)  =\alpha(r,\upsilon)\int_V\left(h(r,\upsilon') -h(r,\upsilon) \right)\pi(r,\upsilon,\upsilon')\d\upsilon',
\label{genactionJ}
\end{equation}
fon $D\times V$.
 Suppose that we denote by $(T_i, i\geq 0)$, the scattering times of $(R,\Upsilon)$ with $T_0 = 0$. Moreover, let $N_t = \sup\{i: T_i\leq t\}$.
Then 
\begin{equation}
\left.\frac{\d \mathbf{P}^h_{(r,\upsilon)}}{\d \mathbf{P}_{(r,\upsilon)}}\right|_{\sigma((R_s,\Upsilon_s), s\leq t)}:  = 
\exp\left(-\int_0^t  \frac{{\ebJ}h (R_s,\Upsilon_s)}{h (R_s,\Upsilon_s)}\d s\right)\prod_{i=1}^{N_t}\frac{h(R_{T_{i}}, \Upsilon_{T_{i}})}{h(R_{T_i}, \Upsilon_{T_{i-1}})}
\mathbf{1}_{(t<\tau^{D})}
\label{hCOM}
\end{equation}
characterises a neutron random walk whose expectation semigroup is given by 
\begin{equation}
\psi^h_t[g](r,\upsilon)  =\eU_t[g](r,\upsilon) + \int_0^{t} \eU_s\big[\ebJ_h \psi^h_{t-s}[g]\big](r,\upsilon)\d\upsilon'.
\label{psihsemig}
\end{equation}
where the scattering operator is given by 
\begin{equation}
\ebJ_h g(r,\upsilon)  = \int_V\left(g(r,\upsilon') -g(r,\upsilon) \right)\alpha(r,\upsilon)\frac{h(r,\upsilon')}{h(r,\upsilon)}\pi(r,\upsilon,\upsilon')\d\upsilon'
\label{J_h}
\end{equation}

on $D$.
Moreover, 
\begin{align}
&\psi_t[g](r,\upsilon)\notag\\
&= \mathbf{E}^h_{(r,\upsilon)}\left[
\exp\left(\int_0^t  \frac{{\ebJ}h (R_s,\Upsilon_s)}{h (R_s,\Upsilon_s)} +\beta(R_s,\Upsilon_s)\d s\right)\prod_{i=1}^{N_t}\frac{h(R_{T_i}, \Upsilon_{T_{i-1}})}{h(R_{T_{i}}, \Upsilon_{T_{i}})}
g(R_t, \Upsilon_t)
\mathbf{1}_{(t<\tau^{D})}\right]
\label{3rd}
\end{align}
for $t\geq 0$ and if we additionally assume that 
\begin{equation}
\inf_{r\in D,\upsilon\in V}  \lim_{s\to\kappa_{r,\upsilon}^D}\frac{|\upsilon| (\kappa_{r,\upsilon}^D-s)}{h (r+\upsilon s,\upsilon)} >0
\label{inf}
\end{equation}
where we recall that
$\notag
\kappa_{r,\upsilon}^{D} := \inf\{t>0 : r+\upsilon t\not\in D\}$,
then the term $\mathbf{1}_{(t<\tau^D)}$ in \eqref{hCOM}  and \eqref{3rd} can be removed and $(R,\Upsilon)$ under $\mathbf{P}^h_{(r,\upsilon)}$, $r\in D, \upsilon\in V$, is conservative.
\end{theo}

\begin{remark}\label{convenient}\rm
We also note from \eqref{psihsemig}, i.e. the expectation semigroup of   $(R,\Upsilon)$ under $\mathbf{P}^h_{(r,\upsilon)}$, $r\in D, \upsilon\in V$, that it is formally associated to the generator 
\begin{equation}
{\bL}_h g(r,\upsilon)  = \upsilon\cdot\nabla g(r,\upsilon)+\int_V\left(g(r,\upsilon') -g(r,\upsilon) \right)\alpha(r,\upsilon)\frac{h(r,\upsilon')}{h(r,\upsilon)}\pi(r,\upsilon,\upsilon')\d\upsilon'
\label{genaction}
\end{equation}
on $D\times V$.
\smallskip

Suppose, moreover, that we write $\bT = \upsilon\cdot\nabla$. If we further assume that $\bT h$ is well defined, for example, one could typically assume that $\bT$ is defined weakly in $ L_2(D\times V)$, then 
we can also write the term 
\begin{equation}
\frac{h(r+\upsilon t,\upsilon)}{h(r,\upsilon)} = \exp\left( \log h(r+\upsilon t,\upsilon) -\log h(r,\upsilon) \right)
= \exp\left(\int_0^t \frac{\bT h(r+\upsilon s,\upsilon)}{h(r+\upsilon s,\upsilon)}\d s\right).
\label{h/h}
\end{equation}
As such, a piecewise application of \eqref{h/h} allows us to convert the right-hand side of \eqref{3rd} into  
\begin{equation}
\psi_t[g](r,\upsilon)= h(r,\upsilon) \mathbf{E}^h_{(r,\upsilon)}\left[
\exp\left(\int_0^t  \frac{{\bL}h (R_s,\Upsilon_s)}{h (R_s,\Upsilon_s)} +\beta(R_s,\Upsilon_s)\d s\right)
\frac{g(R_t, \Upsilon_t)}{h(R_t, \Upsilon_t)}
\mathbf{1}_{(t<\tau^{D})}\right].
\label{4th}
\end{equation}
Similarly the change of measure \eqref{hCOM} can be identified alternatively as 
\begin{equation}
\left.\frac{\d \mathbf{P}^h_{(r,\upsilon)}}{\d \mathbf{P}_{(r,\upsilon)}}\right|_{\sigma((R_s,\Upsilon_s), s\leq t)}:  = 
\exp\left(-\int_0^t  \frac{{\bL}h (R_s,\Upsilon_s)}{h (R_s,\Upsilon_s)} \d s\right)
\frac{h(R_t, \Upsilon_t)}{h(r,\upsilon)}
\mathbf{1}_{(t<\tau^{D})}
\label{hCOM2}
\end{equation}
It is worth remarking that the representation of the semigroup $(\psi_t, t\geq 0)$ in \eqref{4th} takes on a more convenient form that in \eqref{3rd}. However, technically speaking, it requires more assumptions on the function $h$.
\hfill$\diamond$\end{remark}

\begin{proof}[Proof of Theorem \ref{doobhthrm}]
 Let us introduce $\psi^h_t[g](r,\upsilon) = \mathbf{E}^h_{(r,\upsilon)}[g(R_t, \Upsilon_t)\mathbf{1}_{(t<\tau^D)}]$, for $t\geq 0$, $r\in D$, $\upsilon\in V$ and $g\in L_\infty(D\times V)$, and note that we may also write
\[
\psi^h_t[g](r,\upsilon)  = \mathbf{E}_{(r,\upsilon)}\left[\exp\left(-\int_0^t  \frac{{\bJ}h (R_s,\Upsilon_s)}{h (R_s,\Upsilon_s)}\d s\right)\prod_{i=1}^{N_t}\frac{h(R_{T_{i}}, \Upsilon_{T_{i}})}{h(R_{T_i}, \Upsilon_{T_{i-1}})}
g(R_t, \Upsilon_t)\mathbf{1}_{(t<\tau^{D})}\right].
\]
Using standard arguments centred around the Markov property applied at the first scattering time, it is easy to show that $(\psi^h_t, t\geq0)$ is an expectation semigroup.

\smallskip

Next,  by conditioning on the first step of the NRW under $\mathbf{P}_{(r,\upsilon)}$, we get 

\begin{align}
&\psi^h_t[g](r,\upsilon) \notag\\
& = \exp\left(-\int_0^t \alpha(r+\upsilon s, \upsilon) +\frac{{\bJ}h (r+\upsilon s,\upsilon)}{h (r+\upsilon s,\upsilon)}\d s\right) 
g(r+\upsilon t,\upsilon)\mathbf{1}_{(t<\kappa_{r,\upsilon}^{D})}\notag\\
&\hspace{1cm}+\int_0^{t}\mathbf{1}_{(s<\kappa_{r,\upsilon}^{D})} \alpha(r+\upsilon s, \upsilon) \exp\left(-\int_0^s \alpha(r+\upsilon u, \upsilon) +\frac{{\bJ}h (r+\upsilon u,\upsilon)}{h (r+\upsilon u,\upsilon)}\d u\right)
\notag\\
&\hspace{5cm}
\int_{V}
\psi^h_{t-s}[g](r+\upsilon s,\upsilon')
\frac{h(r+\upsilon s,\upsilon')}{h(r+\upsilon s ,\upsilon)}\pi(r+\upsilon s, \upsilon, \upsilon')\d\upsilon' \d s.
\label{conserve}
\end{align}
Next note that from the assumptions (H1) and (H2) and \eqref{h1},
\[
\frac{\bJ h (r,\upsilon)}{h(r,\upsilon)} + \alpha(r,\upsilon)= \frac{\alpha(r,\upsilon)}{h(r,\upsilon)}\int_V h(r,\upsilon') \pi(r,\upsilon, \upsilon') \geq \frac{1}{h(r,\upsilon)}\underline{\alpha}\underline\pi\int_V h(r,\upsilon') \d\upsilon' ,
\]
where $\underline\alpha = \inf_{r\in D, \upsilon \in V} \alpha(r,\upsilon)>0$ and $\underline\pi = \inf_{r\in D, \upsilon,\upsilon' \in V} \pi(r,\upsilon, \upsilon')>0$
Together with assumption \eqref{inf} and and \eqref{h1}, this ensures that 
\[
\lim_{t\uparrow\kappa_{r,\upsilon}^D} \int_0^{t} \left\{ \frac{{\bJ}h (r+\upsilon s,\upsilon)}{h (r+\upsilon s,\upsilon)}  + \alpha (r+\upsilon s,\upsilon) \right\} \d s = \infty
\]
for all $r\in D$, $\upsilon\in D$. Since $\alpha$ is uniformly bounded from above, it follows that the indicators in 
 \eqref{conserve} are not necessary.

\smallskip

  Now appealing to Lemma 1.2, Chapter 4 in \cite{Dynkin2}, we can remove exponential potential terms at the expense of introducing an additive potential term. 
After application of the aforesaid lemma, we have
\begin{align}
\psi^h_t[g](r,\upsilon) &= 
g(r+\upsilon t,\upsilon)\notag\\
&\hspace{1cm}+\int_0^{t\wedge \kappa_{r,\upsilon}^{D}} \alpha(r+\upsilon s, \upsilon) \int_{V}
\notag \psi^h_{t-s}[g](r+\upsilon s,\upsilon')\frac{h(r+\upsilon s,\upsilon')}{h(r,\upsilon)}\pi(r+\upsilon s, \upsilon, \upsilon')\d\upsilon'\notag\\
&\hspace{3cm}-\int_0^{t\wedge \kappa_{r,\upsilon}^{D}} \left(\alpha(r+\upsilon s, \upsilon) + \frac{{\bJ}h (r+\upsilon s,\upsilon)}{h (r+\upsilon s,\upsilon)} \right)\psi^h_{t-s}[g](r+\upsilon s,\upsilon)\d s.
\label{long1}
\end{align}
Noting again that 
\begin{equation}
\alpha(r, \upsilon)+\frac{{\bJ}h (r,\upsilon)}{h (r,\upsilon)} =
\alpha(r,\upsilon) \int_V\frac{ h(r,\upsilon')}{h(r,\upsilon)}\pi(r,\upsilon,\upsilon')\d\upsilon'
\label{Th/h}
\end{equation}
and recalling the definition \eqref{J_h}, 
plugging \eqref{Th/h} back into \eqref{long1}, we get 
\begin{align}
\psi^h_t[g](r,\upsilon)  &= 
g(r+\upsilon t,\upsilon)+\int_0^{t\wedge \kappa_{r,\upsilon}^{D}} \bJ_h \psi^h_{t-s}[g](r+\upsilon s,\upsilon)\d\upsilon'\notag\\
&=\U_t[g](r,\upsilon) + \int_0^{t} \U_s\big[\bJ_h \psi^h_{t-s}[g]\big](r,\upsilon)\d\upsilon'.
\label{h-semigroup}
\end{align}
The equation \eqref{h-semigroup} together with the semigroup property shows that the change of measure \eqref{hCOM}  makes $(R,\Upsilon)$ under $\mathbf{P}^h_{(r,\upsilon)}$, $r\in D, \upsilon\in V$  the law of a NRW whose expectation semigroup is formally associated to the infinitesimal generator \eqref{genaction}.

 \smallskip
 
 To verify conservativeness, it suffices to take $g = 1$ in \eqref{h-semigroup} and note that, since $\bJ 1 = 0$, then $\psi^h_t[g](r,\upsilon) = 1$ is a solution to \eqref{h-semigroup} for all $t$. Similarly to the mild equation \eqref{mild}, the solution to \eqref{h-semigroup} is unique (thanks to a simple application of Gr\"onwall's Lemma), see for example \cite{SNTE}. Thus we have conservativeness.
  \end{proof}

As a small side remark, we note that the representation in Theorem  \ref{doobhthrm} is an alternative form of Feynman-Kac representation which works with multiplicative potentials instead of exponential additive potentials. Examples where this has been used for other Markov systems are discussed in \cite{CS, KP}.

\section{Complexity analysis  of $h$-NRW Monte Carlo and the shape of $h$}\label{h-complexity}
Under the assumptions of Theorem \ref{doobhthrm} (not necessarily including \eqref{inf}) and Remark \ref{convenient} in particular through the representation \eqref{4th},  we see a third approach to simulating $\lambda_*$. More precisely, 
with the estimate \eqref{lambda} 
we have the alternative estimator 
\begin{equation}
\Psi_k[g](t, r, \upsilon) = \Psi^{h\texttt{-rw}}_k[g](t, r, \upsilon) =h(r,\upsilon) \frac{1}{k}\sum_{i= 1}^k  
\exp\left(\int_0^t  \frac{{\bL}h (\texttt{r}_s^i, \texttt{v}_s^i)}{h (\texttt{r}_s^i, \texttt{v}_s^i)} +\beta(\texttt{r}_s^i, \texttt{v}_s^i)\d s\right)
\frac{g(\texttt{r}_t^i, \texttt{v}_t^i)}{h(\texttt{r}_t^i, \texttt{v}_t^i)} \mathbf{1}_{(t <\texttt{t}_{\rm end}^i)}.
\label{Psihrw}
\end{equation}
where $((\texttt{r}_t^i, \texttt{v}_t^i), t\geq 0)$ are independent copies of an $\alpha^h\pi^h$-NRW, with
\begin{equation}
\alpha^h(r,\upsilon)\pi^h(r,\upsilon,\upsilon') = \alpha(r,\upsilon)\frac{h(r,\upsilon')}{h(r,\upsilon)}\pi(r,\upsilon,\upsilon'), 
\qquad r\in D, \upsilon,\upsilon'\in V.
\label{ahpih}
\end{equation}
(See Algorithm  \ref{alg2} in the Appendix.)

\smallskip

Following the analysis in Section \ref{NRWMCCA}, we use similar calculations to control the variance and cost of the estimator associated to \eqref{Psihrw}. For variance control, we have the following result, whose conclusion does not depend on the sign of $\lambda_*$.

\begin{theo}[$h$-NRW Monte Carlo  convergence for $\lambda_*$]\label{hNRWvar} Suppose that  the assumptions of Theorem 
\ref{doobhthrm}, including \eqref{inf}, and of Remark \ref{convenient} hold. Additionally, suppose
\begin{equation}
 \sup_{r\in D,\upsilon\in V}\frac{\ebL h(r,\upsilon)}{h(r,\upsilon)}<\infty.
\label{supLh/h}
\end{equation}
 Then for $g\in L^+_\infty(D\times V)$ such that $g\leq h$, there exists a constant $\kpt{2}>0$ such that  
\begin{equation}
\mathbf{E}_{(r,\upsilon)}^h \Big[\Big(\left(\Psi^{h\texttt{\emph {-rw}}}_k[g](t, r, \upsilon)\right)^{1/t}-\me^{\lambda_{\ast}}\Big)^{2}\Big]\le
\frac{\kpt{2}{\rm e}^{(\lambda_2 -2\lambda_*)t}}{k} + \frac{\kp{0}}{t^2}, 
\label{minmise''}
\end{equation}
as $t\to\infty$, where $ \lambda_* + \underline{\varsigma} \leq \lambda_2 \leq \lambda_* + \overline{\varsigma}$ and
\begin{equation}
-\infty\leq  \inf_{r\in D,\upsilon\in V}\frac{(\ebL +\beta) h(r,\upsilon)}{h(r,\upsilon)}=:\underline{\varsigma}\leq \overline{\varsigma} : = \sup_{r\in D,\upsilon\in V}\frac{(\ebL +\beta) h(r,\upsilon)}{h(r,\upsilon)}.
 \label{betaprimes}
\end{equation}

\end{theo}

Next we turn our attention to the cost analysis in the spirit of Lemma \ref{thm: costrw}.

\begin{lem}[$h$-NRW Expected simulation cost]\label{thm: costrwh} Suppose that the conditions of Theorem \ref{hNRWvar} are met and additionally that $\underline\varsigma>-\infty$. Then, for $f\in L^+_\infty(D\times V)$
  and some constant $K_f>0$,
  \[
    \limsup_{t\to\infty}{\rm e}^{-(\lambda_*-\underline\varsigma)t}\mathbf{E}_{(r,\upsilon)}^h\big[C_{t}[f]\big]
    <K_f,.
  \]
  Moreover, $\mathbf{E}_{(r,\upsilon)}^h\big[C_{t}[f]\big]$ grows at most linearly in $t$ if either $\underline\varsigma = \lambda_*$, or if the following are satisfied:
  \begin{enumerate}
  \item[(i)] $f\in L^+_\infty(D\times V)$ satisfies $\pi^h[f](r,v) \ge c_0> 0$ for all $(r,v) \in D \times V$;
  \item[(ii)] there exist $\Omega_0 \subseteq D \times V$, $p_0 >0$ and $\delta >0$ such that
    \begin{enumerate}
    \item $\int_{(r,v') \in \Omega_0} \pi^h(r,v,v') \, \d v' \ge p_0$;
    \item $\alpha^h(r+vs,v) \le \delta^{-1}$ for all $s \in [0,\delta]$, for all $(r,v) \in \Omega_0$.
    \end{enumerate}
  \end{enumerate}
\end{lem}
\begin{remark} \label{rem:hisphi} \rm
  Heuristically speaking, the upper bound in Theorem \ref{hNRWvar} is optimised by choosing $h = \varphi$. Indeed, in that setting, $(\bL + \beta)\varphi = \lambda_*\varphi$ on $L_2(D\times V)$. Hence, assuming this is a pointwise equality, it follows that $\lambda_2 = 2\lambda_*$, matching the lower bound for $\lambda_2$, thereby minimising the upper bound in \eqref{minmise''}. Note also that in the same setting, $\underline\varsigma = \lambda_*$ and hence the growth of the expected cost is linear.

  Of course, in practice the function $\varphi$ is generally unknown, and a quantity that we want to learn through Monte Carlo methods. In the context of importance sampling, knowing $\varphi$ would be equivalent to knowing the optimal weights for the simulated random variable.
\hfill$\diamond$\end{remark}

The sufficient condition \eqref{inf} is strongly indicative of linear behaviour of $h$ for directional derivatives towards the physical boundary of $D$. These features should be built into any first estimate of $h$. We also note that the main theorem in underlying theory, i.e. Theorem \ref{CVtheorem},  requires that the domain $D$ is convex. In other contexts, where the NTE is interpreted as an abstract Cauchy problem on a Banach space, it  sometimes additionally assumed that the cross sections on $D\times V$ are also piecewise continuous, or even more simply, piecewise constant. This is a very natural assumption for the application of nuclear reactor core design. See for example Figure \ref{cake} which shows a typical simulation in a slice of a nuclear reactor core which has convex symmetries.

\smallskip

Let us assume momentarily that $D$ is a convex polyhedron. That is to say, it is a domain bounded by a finite number of tangent planes. Below we will  construct two possible families of choices of $h$. Both are built using a combination of hyperplanes. 

\begin{example}[Urts functions]\label{Urtsexample1}\rm
  For each $\upsilon\in V$, one possible definition we can work with 
  \[
  h(r,\upsilon) =c\times {\rm dist}_{{\upsilon}/{|\upsilon|}}(r,\partial D),
  \]
   where, for $\omega\in \mathbb{S}_2$, ${\rm dist}_{\omega}(r,\partial D)$ is the distance from $r$ to the boundary of $D$ along the ray $(r + \omega t, t\geq 0)$ and $c>0$. Note that this is a piecewise linear function in $r$ for fixed $v$. Indeed, we can write this more precisely as
\[
h(r,\upsilon)=c\times \inf\{t>0: r+\frac{\upsilon}{|\upsilon|} t \in \partial D\} =  c|\upsilon|\inf\{s>0: r+\upsilon s \in \partial D  \} = c|\upsilon|\kappa^D_{r,\upsilon}.
\]
Note that 
\begin{equation}
 \frac{|\upsilon| (\kappa_{r,\upsilon}^D-s)}{h (r+\upsilon s,\upsilon)} =  
 \frac{\kappa_{r,\upsilon}^D-s}{c\kappa_{r+\upsilon s,\upsilon}^D} =\frac{1}{c}
 \label{fixedas1}
\end{equation}
and hence  \eqref{inf} is satisfied. 

\smallskip

In two dimensions this appears to take the form of the minimum of hyperplanes touching only the sides of $\partial D$ that $\upsilon$ is directed towards (i.e. only $r\in \partial D$ for which $\upsilon\cdot{\bf n}_{r}>0$). See the function on the left in Figure~\ref{urts}.

\begin{figure}[h!]
\includegraphics[width=0.215 \textwidth]{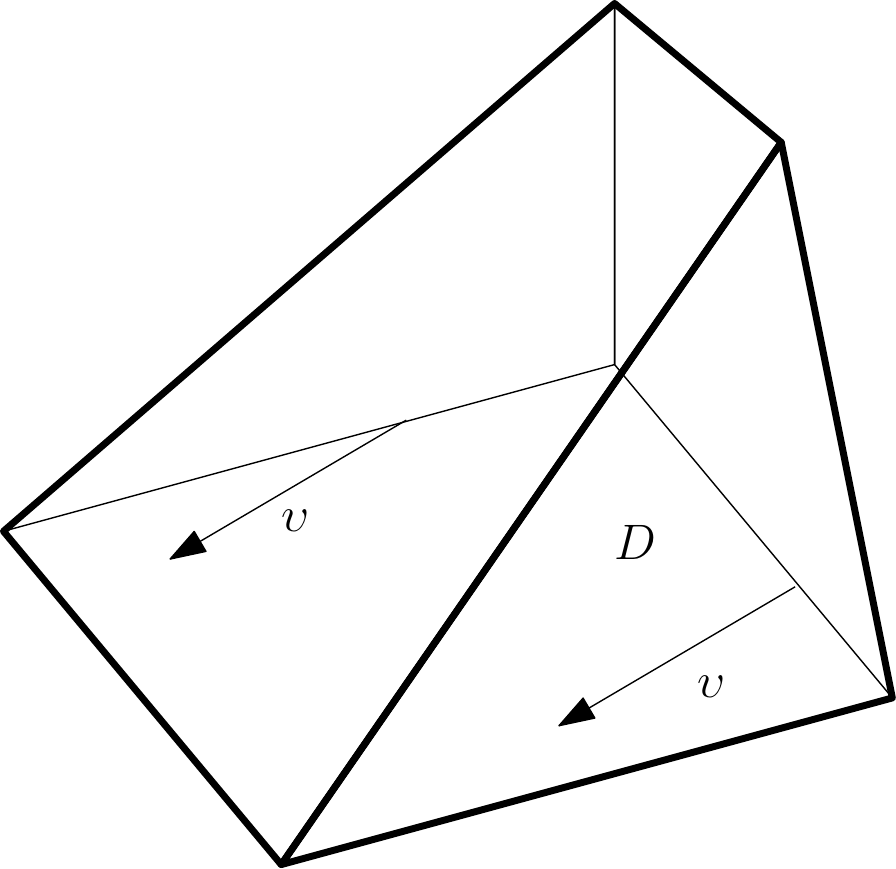}
\hspace{1cm}
\includegraphics[width=0.3 \textwidth]{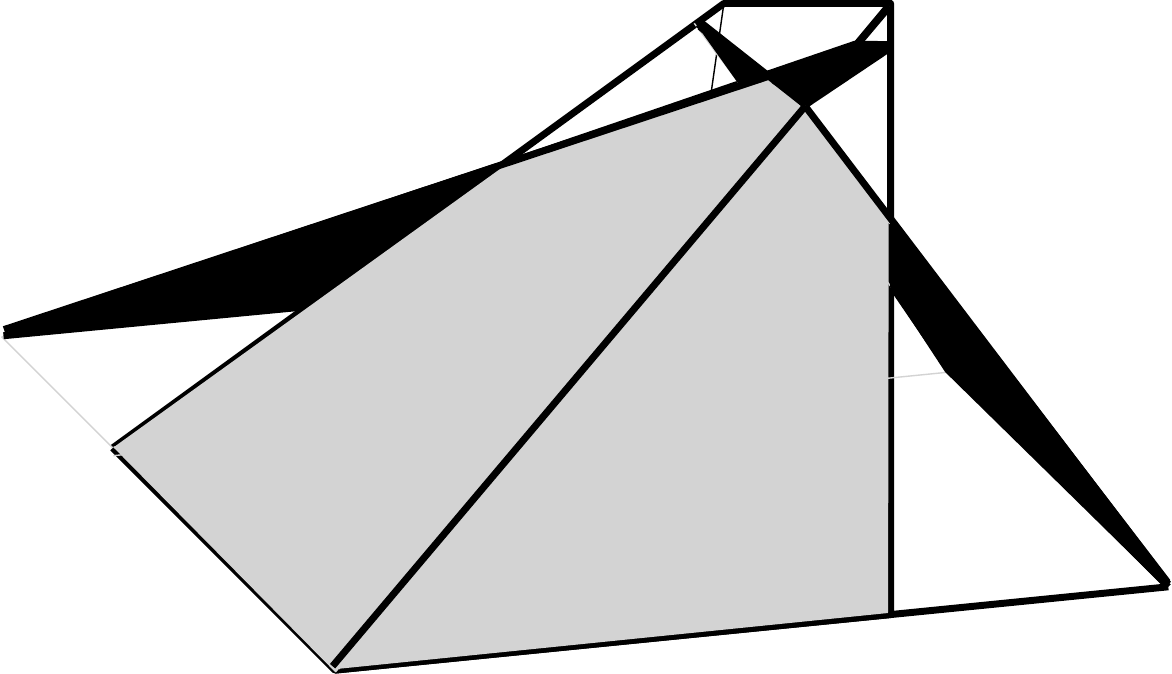}
\caption{On the left $h(r,\upsilon) = c_1{\rm dist}_{{\upsilon}/{|\upsilon|}}(r,\partial D)$, on the right, is the case of an Urts function $h(r,\upsilon) = c_1{\rm dist}_{{\upsilon}/{|\upsilon|}}(r,\partial D)\wedge c_2 {\rm dist}_{{\upsilon}/{|\upsilon|}}(r+ r_\upsilon ,\partial D)$. In each figure, the `base' represents the domain $D$, and the height represents the value of the function $h$.}
\label{urts}
\end{figure}

\smallskip

Recalling \eqref{J_h}, this proposal for $h$ thus imposes an instantaneous repulsive rate for $r$ close to boundary points $r'\in \partial D$ such that $\upsilon\cdot{\bf n}_{r'}>0$. However if $r$ is positioned far from the aforesaid boundary points but close to one of the boundary points $r'\in \partial D$ such that $\upsilon\cdot{\bf n}_{r'}<0$, we should also expect a milder repulsion from this boundary. 
We can thus adjust our proposed approximation to $\varphi$ by choosing it to be such that it receives a gentler repulsion from the boundary which is diametrically opposed to the direction of travel. We thus propose a first estimate of $\varphi$ taking the linear form 
\begin{equation}
h(r,\upsilon) =\min\{c_1
 {\rm dist}_{{\upsilon}/{|\upsilon|}}(r,\partial D)\, ,\,c_2 {\rm dist}_{-{\upsilon}/{|\upsilon|}}(r+ r_\upsilon ,\partial D)\},
\label{2consts}
\end{equation}
where 
$r_\upsilon$ can be dependent on $\upsilon$ or just taken as a fixed constant 
and $c_1,c_2>0$ are constants.
The role of $r_\upsilon$ is to provide the `gentle repulsion' away from the boundaries as alluded to above.
 The second term in the minimum is a simple measure of how far the boundary might be after scattering, although naturally more sophisticated choices could be made, at the expense of more complexity in subsequent calculations.

\medskip

We call this family of approximations for $\varphi$ Urts functions, named after the teepee-like dwellings of the Mongolian Siberian Tsaatan tribes which are similarly shaped when $D$ is a two-dimensional convex polyhedral domain; see Figure \ref{urts}. As before, it is easy to verify that $\inf_{r\in D}\int_V h(r,\upsilon')\d\upsilon'>0$.
\hfill$\diamond$\end{example}


Unfortunately, neither of the previous two examples necessarily respect the condition \eqref{supLh/h}. To see why, 
suppose we consider the setting of an Urts function with $c_1 = c_2 = 1$. For \eqref{supLh/h} to hold, we would need, in particular, that, for $r\in \partial D$ and $\upsilon\in V$ such that $\upsilon\cdot {\bf n}_r>0$ (i.e. the velocity direction of $\upsilon$ is pointing {\it out} of the domain $D$), the limit
\begin{align}
&\lim_{s\to0}\frac{ \bT h(r-s\upsilon/|\upsilon|, \upsilon) + \bJ h(r-s\upsilon/|\upsilon|,\upsilon)}{h(r-s\upsilon/|\upsilon|,\upsilon)}\notag\\
& =  \lim_{s\to 0}\left( \alpha(r, \upsilon )\int_{V}\frac{h(r-s\upsilon/|\upsilon|,\upsilon')}{h(r-s\upsilon/|\upsilon|,\upsilon)} \pi(r,\upsilon,\upsilon')\d\upsilon'  -\frac{|\upsilon|}{h(r-s\upsilon/|\upsilon|,\upsilon)} \right)- \alpha(r,\upsilon)\notag\\
&=\lim_{s\to 0}\left( \alpha(r, \upsilon )\int_{V}\frac{1}{s} h(r-s\upsilon/|\upsilon|,\upsilon')\pi(r,\upsilon,\upsilon')\d\upsilon'  -\frac{|\upsilon|}{s} \right)- \alpha(r,\upsilon)\label{blowdown}
\end{align}
needs to exist and be finite. Observe that in the last step we used the form of the function $h$ to deduce that $h(r-s\upsilon/|\upsilon|,\upsilon) = s$ for $s$ small, $r\in \partial D$ and $\upsilon\in V$ such that $\upsilon\cdot {\bf n}_r>0$. Note that this is potentially problematic because for $\upsilon'$ in the integral on the right-hand side above such that $\upsilon'\cdot {\bf n}_r <0$, we have $h(r,\upsilon')>0$, and hence the factor $1/s$ means that, in the limit as $s\downarrow0$, if it exists, can be explosive. Indeed, on the one hand, for $s$ sufficiently close to zero,
\[
\lim_{s\to0}\int_{\{\upsilon'\cdot {\bf n}_r>0\}} \frac{h(r-s\upsilon/|\upsilon|,\upsilon')}{h(r-s\upsilon/|\upsilon|,\upsilon)} \pi(r,\upsilon,\upsilon')\d\upsilon'
 = \lim_{s\to0}\int_{\{\upsilon'\cdot {\bf n}_r>0\}} \frac{ (\upsilon\cdot {\bf n}_r)|\upsilon'|}{  (\upsilon'\cdot {\bf n}_r) |\upsilon|} \pi(r,\upsilon,\upsilon')\d\upsilon'<\infty.
\]
On the other hand, for $\upsilon'$ such that $\upsilon'\cdot  {\bf n}_r<0$, we have that $h(r,\upsilon')>0$ and hence, for $s$ sufficiently small, 
\[
\int_{\{\upsilon'\cdot  {\bf n}_r<0\}} \frac{h(r-s\upsilon/|\upsilon|,\upsilon')}{h(r-s\upsilon/|\upsilon|,\upsilon)} \pi(r,\upsilon,\upsilon')\d\upsilon'> \frac{C}{s},
\]
for some constant $C>0$, which explodes as $s\to 0$. On the other hand, so long as the supremum in \eqref{blowdown} is uniformly bounded above (i.e. the limit could be $-\infty$), then \eqref{supLh/h} is respected.


\begin{example}[Lifted Urts functions]\rm 
The previous discussion indicated a possible way of building an Urts function such that the requirement \eqref{supLh/h} is satisfied. Another way to do this is to violate the zero boundary condition that $h(r,\upsilon) = 0$ when $\upsilon \cdot {\bf n}_r>0$ and to take the Urts function of Example \ref{Urtsexample1} and add a small constant to it. That is to say, 
\[
h(r,\upsilon) =\varepsilon + \min\{
 c_1{\rm dist}_{{\upsilon}/{|\upsilon|}}(r,\partial D)\, ,\, c_2{\rm dist}_{-{\upsilon}/{|\upsilon|}}(r+ r_\upsilon ,\partial D)\},
\]
for some $0<\varepsilon\ll 1$ and $c_1,c_2>0$. For a lifted Urts function, the condition \eqref{inf} is thus not satisfied and, as a consequence, not all of the $h$-NRW will survive until the chosen time horizon.
\hfill$\diamond$\end{example}

\begin{lem}
  Suppose that the domain $D \subset \R^2$ is bounded by a regular polygon or circle, and $h$ is one of the urts functions described above. Then the conditions of (ii) of Lemma~\ref{thm: costrwh} are satisfied; in particular, the cost $\mathbf{E}_{(r,\upsilon)}^h\big[C_{t}[f]\big]$ will grow at most linearly whenever $f$ is uniformly bounded below by a strictly positive constat.
\end{lem}

\begin{proof}
  For simplicity, we consider the case where $D = [-L,L]\times[-L,L]$. The proof in other cases is similar.

  Fix $\eta>0$, and define
  \begin{align*}
    \Omega_1 & := ([-L,L-\eta]\times [-L,L-\eta]) \times \{ v \in V: v/|v| = (\cos(\theta),\sin(\theta)), \theta \in [\pi/8,3\pi/8]\},\\
    \Omega_2 & := ([-L+\eta,L]\times [-L,L-\eta]) \times \{ v \in V: v/|v| = (\cos(\theta),\sin(\theta)), \theta \in [5\pi/8,7\pi/8]\}\\
    \Omega_3 & := ([-L+\eta,L]\times [-L+\eta,L]) \times \{ v \in V: v/|v| = (\cos(\theta),\sin(\theta)), \theta \in [9\pi/8,11\pi/8]\}\\
    \Omega_4 & := ([-L,L-\eta]\times [-L+\eta,L]) \times \{ v \in V: v/|v| = (\cos(\theta),\sin(\theta)), \theta \in [13\pi/8,15\pi/8]\}
  \end{align*}
  and $\Omega_0 := \Omega_1 \cup \Omega_2 \cup \Omega_3 \cup \Omega_4$. Then it follows from each of the definitions of the urts functions that $h$ is bounded below on $\Omega_0$, and bounded above on $D\times V$. Hence we deduce that (ii).(a) holds. Similarly, the fact that $h$ is bounded below on $\Omega_0$, and above on $D \times V$ guarantees that (ii).(b) holds.

  Note in particular that (ii).(a) guarantees also condition 1 of Lemma~\ref{thm: costrwh} when $f$ is bounded below, uniformly away from zero.
\end{proof}

We conclude with the corresponding version of Corollary~\ref{cor:NRWComplex}.

\begin{cor}[$h$-NRW Monte Carlo complexity]  Suppose all conditions of Theorem~\ref{hNRWvar} are satisfied, and we are in the case of Lemma~\ref{thm: costrwh} where $\mathbf{E}_{(r,\upsilon)}^h\big[C_{t}[f]\big]$ grows at most linearly. Then there is a choice of $k,t$ such that
  \begin{equation} \label{eq:Mex2}
    \mathbb{E}_{\delta_{(r,\upsilon)}} \Big[\Big(\left(\Psi^{\texttt{\emph{h-rw}}}_k[g](t, r, \upsilon)\right)^{1/t}-\me^{\lambda_{\ast}}\Big)^{2}\Big]\le \varepsilon^2.
  \end{equation}
  Moreover, \eqref{eq:Mex2} holds in the limit as $\eps \to 0$ with 
  \begin{equation*}
    \mathbb E_{\delta_{(r, \upsilon)}}\big[\emph{\texttt{Cost}}(k, t)\big]  \le C \exp\left( (\lambda_2-2\lambda_{\ast}) \sqrt{\kappa_{[0]}} \varepsilon^{-1}\right),
  \end{equation*}
  for some $C>0$, and the optimal choice of $k,t$ corresponds to $t \approx \sqrt{\kp{0}} \varepsilon^{-1}$ and $k \approx (\lambda_2-2\lambda_{\ast}) \sqrt{\kp{0}} \kpt{2} \varepsilon^{-3} \exp((\lambda_2-2\lambda_{\ast}) t)/2$.
 
\end{cor}

The proof follows in a similar manner to Corollary~\ref{cor:NRWComplex}.

\section{Discussion and numerical comparison of methods}\label{discussion}
\subsection{Relation to Existing Monte Carlo methods}

There are many interesting connections between the methods we have discussed in this paper and existing Monte Carlo methods discussed in the literature. Of particular relevance for our discussion (and in particular in comparison with the NBP and NRW methods described), are various Monte Carlo methods which have been proposed for simulating from the quasi-stationary distribution of a given Markov process. We highlight some relevant strands fo research below:
\begin{itemize}
\item A common approach to simulate from the QSD of a killed Markov process is to use \emph{Fleming-Viot} particle systems. In this approach, $N$ particles are simulated independently from a (killed) Markov process, up to the first time one of the particles is killed, and the killed particle is replaced by a copy of a randomly chosen surviving particle. See \eg{} \cite{burdzy_flemingviot_2000,groisman_simulation_2012,champagnat_convergence_2018}. Under fairly weak conditions, this process converges to the QSD as the simulation time, and the number of particles go to infinity. Such methods are not immediately applicable in the case of the NTE, however, since they generally require a condition on the particles that the killing times for independent particles cannot coincide. In general this is not true for the NTE. In addition, existing methods do not handle branching, although this can be incorporated into the NRW setting by increasing the killing rate when the branching rate is bounded above. Below we will also discuss how other approaches based on interacting particle systems may also be applied in our setting. 

\item Another, potentially interesting alternative approach to simulation from the QSD involves stochastic approximation approaches, \cite{benaim_stochastic_2018, mailler_stochastic_2020}. Again, these methods look to  simulate the QSD of a killed process. In this setting, rather than simulating many particles and using other particles to resurrect killed particles, killed particles are resurrected at some point in their history. In this manner, a single, potentially infinite trajectory is sampled, and a quantity based on the occupation measure (\cf{} Section~\ref{sec:occupation}) is used to estimate the QSD. Existing results are not directly applicable to our setting, either because the results are in discrete time, or diffusion models in continuous time (\cite{wang_approximation_2020}) however it would appear to be an interesting question for future work to see if these methods can be fruitfully adapted to the NTE setting.

\item There are a number of recent papers in the Markov Chain Monte Carlo (MCMC) literature which look to simulate from a given measure, typically motivated by applications in Bayesian statistics. In these applications, and unlike the NTE, the measure is generally known (at least, up to a normalising constant), and the challenge is to construct Markov processes which realise the given measure as their (quasi-)stationary distribution. In \cite{pollock_quasi-stationary_2020} the measure is constructed as the QSD of a diffusion model, and interacting particle methods applied to simulate from the QSD. Other approaches (\cite{vanetti_piecewise-deterministic_2018,pollock_quasi-stationary_2020}) highlight the benefit of using piecewise deterministic Markov processes (PDMPs) as an efficient method to rapidly converge to the QSD, and show how to efficiently constrict PDMPs which have a desired stationary measure. Since the NTE is an example of a PDMP, it would be interesting to understand the extent to which numerical methods for the NTE can borrow from, or inspire ideas for these novel MCMC methods.
\end{itemize}

\subsection{Towards Improved Monte Carlo Methods for the NTE}

The results of the previous section (in particular Remark~\ref{rem:hisphi}) suggests that extremely efficient methods for estimating $\lambda_*$ can be developed if the eigenfunction $\varphi$ is known. Unfortunately, finding the optimal $\varphi$ is generally at least as hard as finding the optimal $\lambda_*$, so this is not immediately helpful. However, importance sampling can still be highly effective in reducing the variance of the numerical scheme even when the function $\varphi$ is not known. The key question is how to construct informative functions $h$ which might approximate $\varphi$ well.
\smallskip

Moreover, in practical implementations, the naive approach suggested in Theorem~\ref{thm:NBPComplexity} is generally not amenable in practice. Considering the asymptotic results given here, one first observes that it is heavily desirable to be in the $\lambda_*=0$ regime, which has linear growth, rather than the cases when $\lambda_* \neq 0$, which both have exponential growth in complexity for a given level of accuracy. It follows that it is numerically beneficial to try to get the system to an equilibrium state, for example, by introducing additional killing or branching (although the optimal rate is of course unknown, it could, for example, be estimated numerically from simulations). Moreover, numerical methods which try to mollify excess particle deaths or births through repopulating from the existing population can also be effective at keeping the system close to equilibrium. Such methods lead naturally to considering interacting particle systems (e.g. \cite{del_moral_feynman-kac_2004}), and Sequential Monte Carlo (SMC) methods (\cite{DoucetSMC}). In most industrial-standard software\footnote{For example, MCNP \cite{werner_mcnp_2018}, MONK \cite{richards_recent_2019} Serpent \cite{leppanen_serpent_2015}, and Tripoli-4 \cite{brun_tripoli-4_2015}}, numerical implementations follow a modified version of classical SMC methods known as power iteration (\cite[Ch. 6.III]{lux_monte_1991}, \cite{brown_monte_2016}) where particles are randomly selected from a `birth store', representing the current particle population, with unused particles being stored for possible later use. The eigenvalue is estimated by adjusting the birth-rate in a manner that We are not aware of any rigorous theoretical justification or analysis of these methods, but it would seem that they would bear significant similarity to standard SMC methods, and we would expect similar theoretical behaviour.  \smallskip

The benefits of combining SMC methods with importance sampling has been investigated by (e.g. \cite{WhiteleyTwisted}). It is in this combination that we expect the results in Sections~\ref{importance} and \ref{h-complexity} to be of most benefit. Specifically, we would anticipate using SMC methods to sample/resample from a population of particles, which themselves undergo motion according to an $h$-transformed version of the NRW. We aim to write more about this in forthcoming work.
\smallskip

In practice, in the nuclear engineering literature, a large range of different variance reduction methods are used. One class of variance reduction methods appears to be closely related to the methods introduced in Sections~\ref{importance} and \ref{h-complexity}. In the nuclear engineering literature, these variance reduction methods are called \emph{zero-variance} calculations, see \cite{christoforou_zero-variance-based_2011}, building on earlier work stretching back to the 1940s. Most of the nuclear engineering literature considers zero-variance schemes in the context of shielding problems, where the goal is to estimate the neutron flux at a detection site, based on a fixed source of particles, and where fissile behaviour is typically not dominant; \cite{christoforou_zero-variance-based_2011} is a notable exception to this, although unlike the current paper, this paper considers the case of generational growth, not time-dependent growth, and the presentation is not mathematically rigorous. To the best of our knowledge, there is no work in the nuclear engineering literature which considers the zero-variance solutions to time-dependent criticality problems, although these are current areas of active research, see for example \cite{zoia_alpha_2014,faucher_variance-reduction_2019,mancusi_towards_2021}.

\smallskip

A further difficulty that arises in the use of $h$-transformed motion for the estimation of $\lambda_*$ is that the estimates can be dominated by large-deviation effects. Specifically, although the estimates remain unbiased, for large $t$ a substantial contribution to the expectation in \eqref{3rd} comes from rare particles with large weights (the exponential, product and indicator terms in the expression). For example, \cite{jack_ergodicity_2020}, in a slightly different setting, connects the limiting value $\lambda$ to a maximisation problem, which in our setting reads as:
\begin{equation*}
  \lambda_* \ge \lim_{T \to \infty}\left\{ \mathbb{E}^h\left[ \beta(R_T,\Upsilon_T)\right] - \frac{1}{T} \mathbb{E}^h\left[ \log \left( \left.\frac{\mathrm{d}\mathbb{P}^h}{\mathrm{d}\mathbb{P}}\right|_T\right) \right]\right\}.
\end{equation*}
Moreover, the optimal choice of $h = \varphi$ will attain equality in this expression. By manipulation of this expression, it should be possible to develop iterative schemes which improve $h$ to get better estimates of $\lambda_*$. In a different setting, this is the approach taken recently by \cite{HengControlledSMC}. We leave the complete discussion of how these methods may be implemented in the case of the NTE to future work.

\smallskip
In the rest of this section we provide some basic numerical experiments\footnote{The numerical simulations have all been implemented in Python. The code can be downloaded from \url{people.bath.ac.uk/mapamgc/}.}  that demonstrate the convergence rates observed in the previous sections. We concentrate on an analytically tractable one-dimensional case, as well as a two dimensional setting, which is not analytically tractable. We conclude with some discussion of the relative strengths and weaknesses of the methods.

\subsection{One dimensional slab reactor} We consider a particular example where $D$ is an open interval {\color{blue} c.f. \cite{christoforou_zero-variance-based_2011,zoia_alpha_2014}}.  Up to some linear transformation of the variables, we can always assume that $D=\{(r_1, r_2, r_3): -L<r_1<L\}$ for some fixed $(r_2, r_3)\in \mathbb R^2$ and $L\in (0, \infty)$. 
Clearly in this case $(\psi_t, t\geq 0)$ is merely a function of the first coordinate.  
We can therefore suppress the dependency on the second and third coordinates and reduce the problem entirely to its one dimensional form.  
In other words, we can reduce the entire problem to the setting $D = (-L,L)$.
\smallskip

The interest of this example lies in that computations of the leading eigenvalue and the associated eigenfunctions in this case are tractable so that we can compare numerical results with theoretical values. 
Throughout this section, we assume that
\begin{align}\label{eq:1d}
&    V=\{-\upsilon_0, \upsilon_0\}, \quad   \sigma_{\texttt{s}}(r, \upsilon)\equiv \sigma_{\texttt{s}}, \quad \sigma_{\texttt{f}}(r, \upsilon)\equiv \sigma_{\texttt{f}}, \\ \notag
& \qquad  \pi_{\texttt{s}}(r, \upsilon, \upsilon')=
\delta_{-\upsilon}(\d \upsilon')
, \quad  \pi_{\texttt{f}}(r, \upsilon, \d\upsilon')=2
\delta_{\upsilon}(\d \upsilon')
\end{align}
where $\upsilon_0, \sigma_{\texttt{s}}, \sigma_{\texttt{f}}\in (0, \infty)$. Note, as a small technical matter, because the velocity space $V$ consists of just two elements, we need to replace $\pi_{\texttt s}(r,\upsilon,\upsilon')\d\upsilon'$ by $\pi_{\texttt s}(r,\upsilon,d\upsilon')$ and a similar replacement for $\pi_{\texttt f}$. Observe that we have $\sigma=\sigma_{\texttt{s}}+\sigma_{\texttt{f}}$ and the two measures $\pi_{\texttt{f}}(r, \upsilon,\d  \upsilon')$ and $\pi_{\texttt{s}}(r, \upsilon, \d \upsilon')'$ are singular with respect to one another. Note that under 
Assumption \eqref{eq:1d} the operator $\bA = {\bT} + {\bS}+{\bF}$ takes the following simple form: 
\begin{equation}
\left\{
\begin{array}{rll}
{\bT}{f}(r, \upsilon)  &=  \upsilon\cdot\nabla{f}(r, \upsilon)   &\text{ (transport) }\\
&\\
{\bS}{f}(r, \upsilon) &= \sigma_{\texttt{s}}\big({f}(r, -\upsilon)-{f}(r, \upsilon)\big) &\text{ (scattering) }\\
&\\
{\bF}{f}(r, \upsilon) & =  \sigma_{\texttt{f}}\big( 2 {f}(r, \upsilon)-{f}(r, \upsilon)\big) =  \sigma_{\texttt{f}}{f}(r, \upsilon) &\text{ (fission)}
\end{array}
\right.
\label{1D_backwards_operators}
\end{equation}
for all $f\in L_2^+(D\times V)$ and $(r, v)\in D\times V$, where $p\in(1,\infty)$. 
\smallskip

The main condition of Theorem \ref{CVtheorem}
requires that $\sigma_{\texttt f}(2-1) = \sigma_{\texttt f}>0$, which is automatically the case. 
Moreover, Theorem 8.1 of \cite{MultiNTE} allows us to otherwise   identify 
  the eigentriple $(\lambda_\ast, \varphi, \tilde\varphi)$  via the relations
  \begin{equation}\label{eq:eigen}
\langle \tilde\varphi, \bA  g\rangle = \lambda_*\langle \tilde\varphi,  g\rangle\text{ and  }\langle f , \bA\varphi\rangle  = \lambda_* \langle f, \varphi \rangle,
\qquad f,g\in L^+_2(D\times V),
\end{equation}
where the eigenfunctions are identified this way only uniquely within $L_2^+(D\times V)$ (which is sufficient for our purposes to identify the eigenfunctions in $L_\infty^+(D\times V)$).

\begin{prop}\label{intervalprop}
The leading eigenvalue $\lambda_\ast$ and the associated eigenfunctions $\varphi, \tilde\varphi$ as specified in Theorem \ref{CVtheorem} are given as follows.
\begin{itemize}
\item[(i)]
If $\theta:={\upsilon_0}/{2L\sigma_{\mathtt{s}}}>1$, 
let $x_{\ast}$ be the unique non negative solution to the equation
\begin{equation}\label{eq:fixedpt1}
\frac{\sinh(x)}{x}=\theta. 
\end{equation}
Then the leading eigenvalue is given as 
\begin{equation}\label{eq:1D_eigen1}
\lambda_\ast=\sigma_{\mathtt{f}}-\sigma_{\mathtt{s}}-\sqrt{\sigma_{\mathtt{s}}^2+\big(\tfrac{\upsilon_0 x_\ast}{2L}\big)^2}. 
\end{equation}
And we have 
\[
\varphi(r, \upsilon)= \phi(r)\1_{\{\upsilon=\upsilon_0\}}+\phi(-r)\1_{\{\upsilon=-\upsilon_0\}}, \quad \tilde\varphi(r, \upsilon)= \phi(-r)\1_{\{\upsilon=\upsilon_0\}}+\phi(r)\1_{\{\upsilon=-\upsilon_0\}},
\]
where
\begin{equation}\label{eq:phi1}
\phi(r)=\frac{\sinh\big(\tfrac{1}{2} x_\ast(1-\frac{r}{L})\big)}{\sinh(\tfrac{1}{2} x_\ast)}, \qquad r\in (-L, L). 
\end{equation}
\item[(ii)]
If $\theta:={\upsilon_0}/{2L\sigma_{\mathtt{s}}}=1$, then the leading eigenvalue is given as 
\begin{equation}\label{eq:1D_eigen2}
\lambda_\ast=\sigma_{\mathtt{f}}-2\sigma_{\mathtt{s}}.  
\end{equation}
And we have
\[
\varphi(r, \upsilon)= \phi(r)\1_{\{\upsilon=\upsilon_0\}}+\phi(-r)\1_{\{\upsilon=-\upsilon_0\}}, \quad \tilde\varphi(r, \upsilon)= \phi(-r)\1_{\{\upsilon=\upsilon_0\}}+\phi(r)\1_{\{\upsilon=-\upsilon_0\}},
\]
where
\begin{equation}\label{eq:phi2}
\phi(r)=1-\frac{r}{L}\, , \qquad r\in (-L, L). 
\end{equation}
\item[(iii)]
If $\theta:={\upsilon_0}/{2L\sigma_{\mathtt{s}}}<1$, then let $x_{\ast}$ be the smallest non negative solution to the equation
\begin{equation}\label{eq:fixedpt2}
\frac{\sin(x)}{x}=\theta. 
\end{equation}
Then the leading eigenvalue is given as 
\begin{equation}\label{eq:1D_eigen3}
\lambda_\ast=\sigma_{\mathtt{f}}-\sigma_{\mathtt{s}}-\sgn(\cos(x_{\ast}))\sqrt{\sigma_{\mathtt{s}}^2-\big(\tfrac{\upsilon_0 x_\ast}{2L}\big)^2}. 
\end{equation}
And we have 
\[
\varphi(r, \upsilon)= \phi(r)\1_{\{\upsilon=\upsilon_0\}}+\phi(-r)\1_{\{\upsilon=-\upsilon_0\}}, \quad \tilde\varphi(r, \upsilon)= \phi(-r)\1_{\{\upsilon=\upsilon_0\}}+\phi(r)\1_{\{\upsilon=-\upsilon_0\}},
\]
where
\begin{equation}\label{eq:phi3}
\phi(r)=\frac{\sin\big(\tfrac{1}{2} x_\ast(1-\frac{r}{L})\big)}{\sin(\tfrac{1}{2}x_\ast)}, \qquad r\in (-L, L). 
\end{equation}
\end{itemize}
\end{prop}

Before turning to the proof of this proposition, there are a number of remarks we should make. 

\smallskip

First, note that the condition $\theta=1$ can be written ${\upsilon_0}/{\sigma_{\mathtt{s}}}=2L$, thus the different forms of the lead eigenfunction correspond intuitively as to whether the average distance that a  particle travels between scatterings is greater than, or less than, the width of the strip. 
Second, the symmetric form of $\varphi$ above can be readily seen from the invariance of the system under the reflection $(r, v)\mapsto (-r, -v)$. 
%
Third, as seen in the previous discussion, the eigenfunctions $\varphi$ and $\tilde\varphi$ are unique up to a multiplicative constant. Here, we fix this constant by requiring $\varphi(0, \upsilon_0)=\tilde\varphi(0, -\upsilon_0)=1$. 

\subsection{Numerical experiments with the one dimensional slab reactor}
For this setting,  we can easily develop numerical solutions for the NBP, NRW and $h$-NRW, albeit that we must decide which $h$ function to work with in advance. As $V$ consists of just two elements, the setting of Urts functions is easy to work with.

\smallskip

For convenience, assume that $ \upsilon_0/2L\sigma_{\texttt s}<1$, then we will work with an Urts function of the form 
\begin{equation}\label{ex:h}
h_{1}(r, \upsilon_0) = \left(r+L+\frac{\upsilon_0}{\sigma_{\texttt s}}\right)\wedge (L-r)\text{ and }h_{1}(r, -\upsilon_0) =  (r+L)\wedge\left(L-r+\frac{\upsilon_0}{\sigma_{\texttt s}}\right),
\end{equation}
for $-L<r<L$. 
In that case, from \eqref{betadef}, we have that $\beta(r,\upsilon) = \sigma_{\texttt{f}}$. Note, moreover, that, for example, when  $(r,\upsilon) = (L,\upsilon_0)$,
\[
\bL h_{1}(L,\upsilon_0) = (\bA - \beta)h_{1}(L,\upsilon_0)=  -\upsilon_0+ \sigma_{\texttt{s}}\left({h_{1}}(L, -\upsilon_0)-h_{1}(L,\upsilon_0)\right) =0.
\]
Hence, we can easily verify that 
\[
\sup_{r\in D, \upsilon\in\{-\upsilon_0,\upsilon_0\}}\frac{\bL h_{1}(r,\upsilon)}{h_{1}(r,\upsilon)}<\infty.
\] 
We will also consider the choices of the Urt functions proposed similar to Examples \ref{Urtsexample1}. 
\begin{equation}\label{ex:h'}
h_{2}(r, \upsilon_0) =  L-r \quad\text{and}\quad h_{3}(r, \upsilon_{0})=\left(r+L+\frac{\upsilon_0}{\sigma_{\texttt s}}\right) (L-r),
\end{equation}
and $h_{i}(r, -\upsilon_{0})=h_{i}(-r, \upsilon_{0})$, for $i=2, 3$.

\bigskip

{\bf NBP estimators. } 
Recall from Theorem \ref{cor:var} that the branching estimator $\Psi^{\texttt{br}}_k[g](t, r, \upsilon)$ has different behaviours for supercritical, critical and subcritical systems. While estimations improve as time grows in a supercritical situation, this is not the case for subcritical systems (Fig.~\ref{fig:sub-br}) since most genealogies have become extinct at large times. Moreover, simulations (Fig.~\ref{fig:simu-br}) suggest that in the supercritical case for fixed $k$ the convergence rate is indeed  $1/t$. 

\begin{figure}[htp]
  \includegraphics[width=\textwidth]{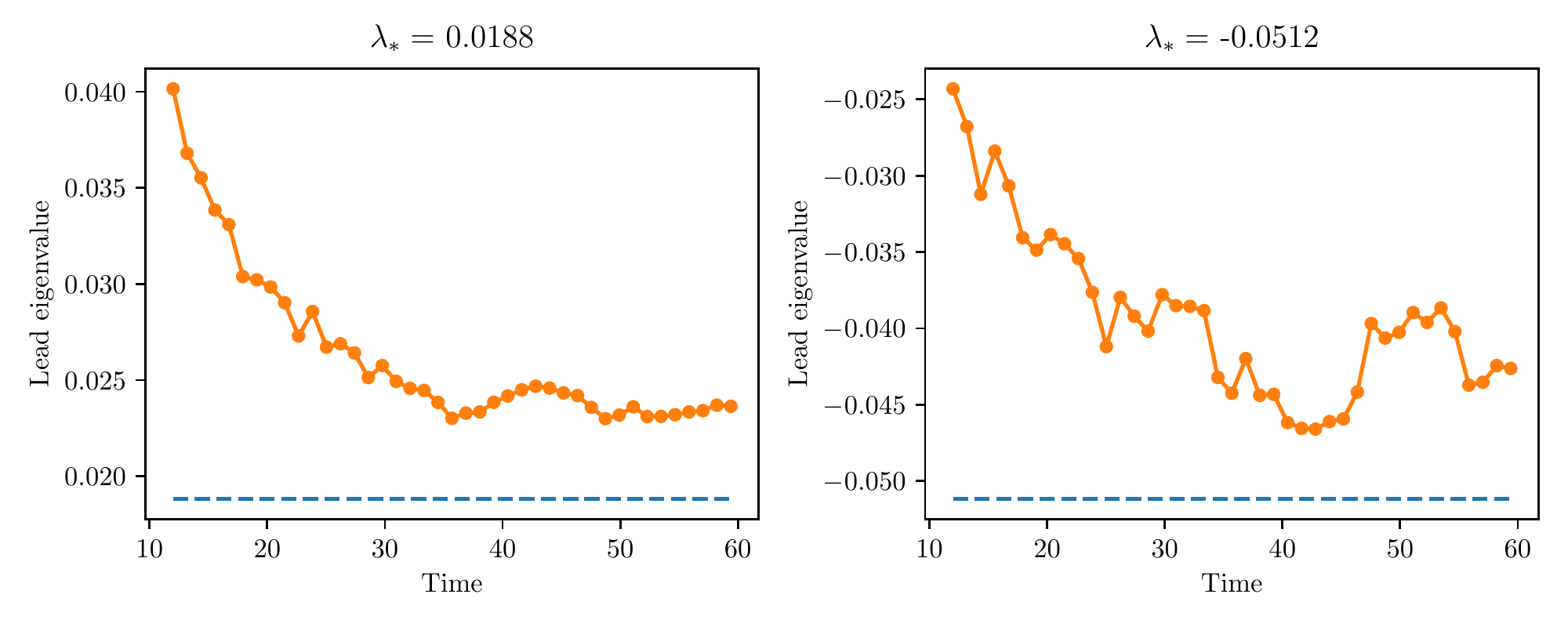}
  \caption{\label{fig:sub-br} {\bf Supercritical vs.~subcritical. } Estimates of $\lambda_*$ from the branching estimator in a supercritical case (left) and a subcritical one (right). }
\end{figure}

\begin{figure}[htp]
  \includegraphics[width=\textwidth]{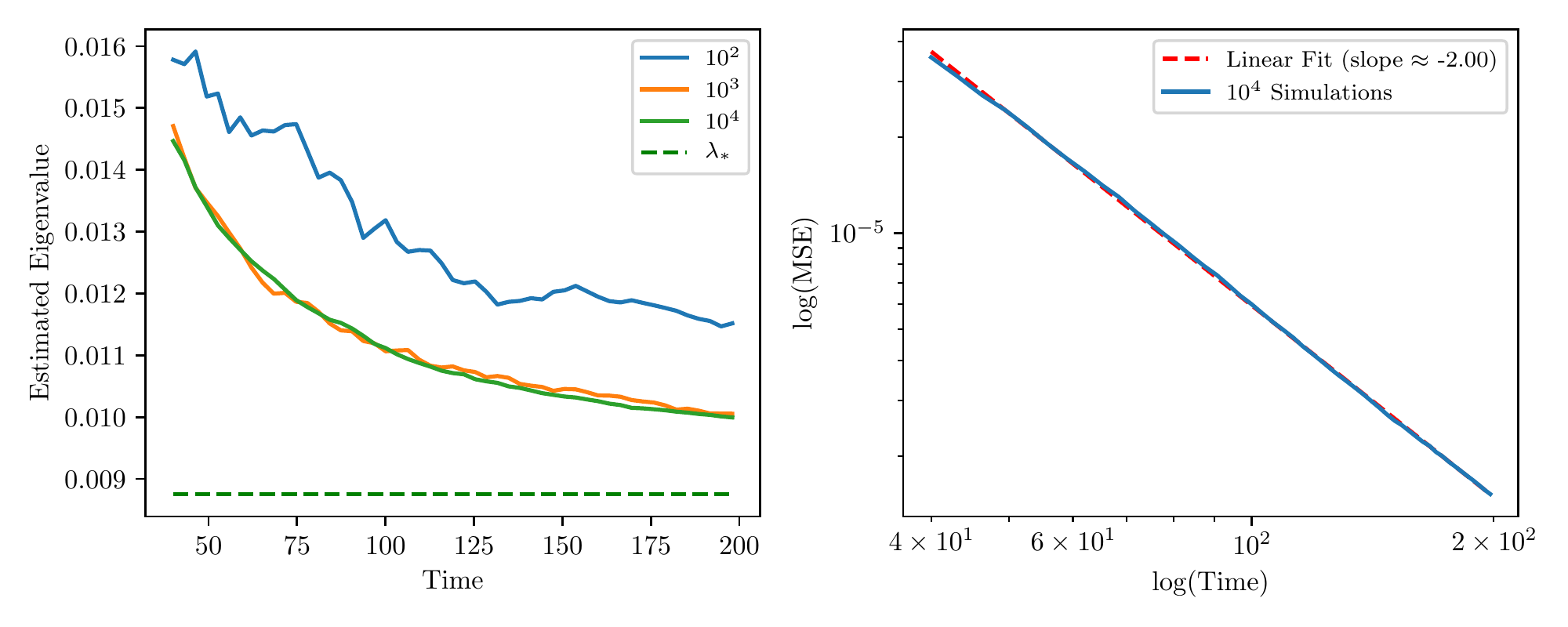}
\caption{\label{fig:simu-br} {\bf Convergence rate of branching estimators. } (Left) Estimates of $\lambda_*$ from the branching estimator $\frac1t\Psi^{\texttt{\emph{br}}}_k[g](t, r, \upsilon)$in a supercritical situtation; the dashed line depicts the value of $\lambda_{\ast}$ and different lines correspond to experiments with different values of $k$. (Right) The function $t\mapsto (\frac1t\Psi^{\texttt{\emph{br}}}_k[g](t, r, \upsilon)-\lambda_{\ast})^{2}$ in log-log scale; the slope is approximately $-2.00$. }
\end{figure}

Next, we illustrate how our estimates of the simulation costs can be applied in this setting with two natural choices of $(f, g)$: one being $f=\mathbf 0$ and $g=\mathbf 1$, which, as discussed previously, estimates the memory cost as the function $C[\mathbf 0, \mathbf 1](t)$ counts precisely the number of particles that have appeared before time $t$, the other chosen to be $f=\mathbf 1$ and $g=\mathbf 0$ for an estimate of the CPU-time. 
Recall that in the supercritical case Lemma \ref{thm: cost} predicts an exponential growth at rate $\lambda_{\ast}$. 
Note also that in the current setting, the various constants appearing in Lemma \ref{thm: cost} can be easily computed from the solutions given in Proposition \ref{intervalprop}. See Fig.~\ref{fig:sub-cost} and \ref{fig:cost-br} for the numerical results.

\begin{figure}[htp]
  \includegraphics[width=\textwidth]{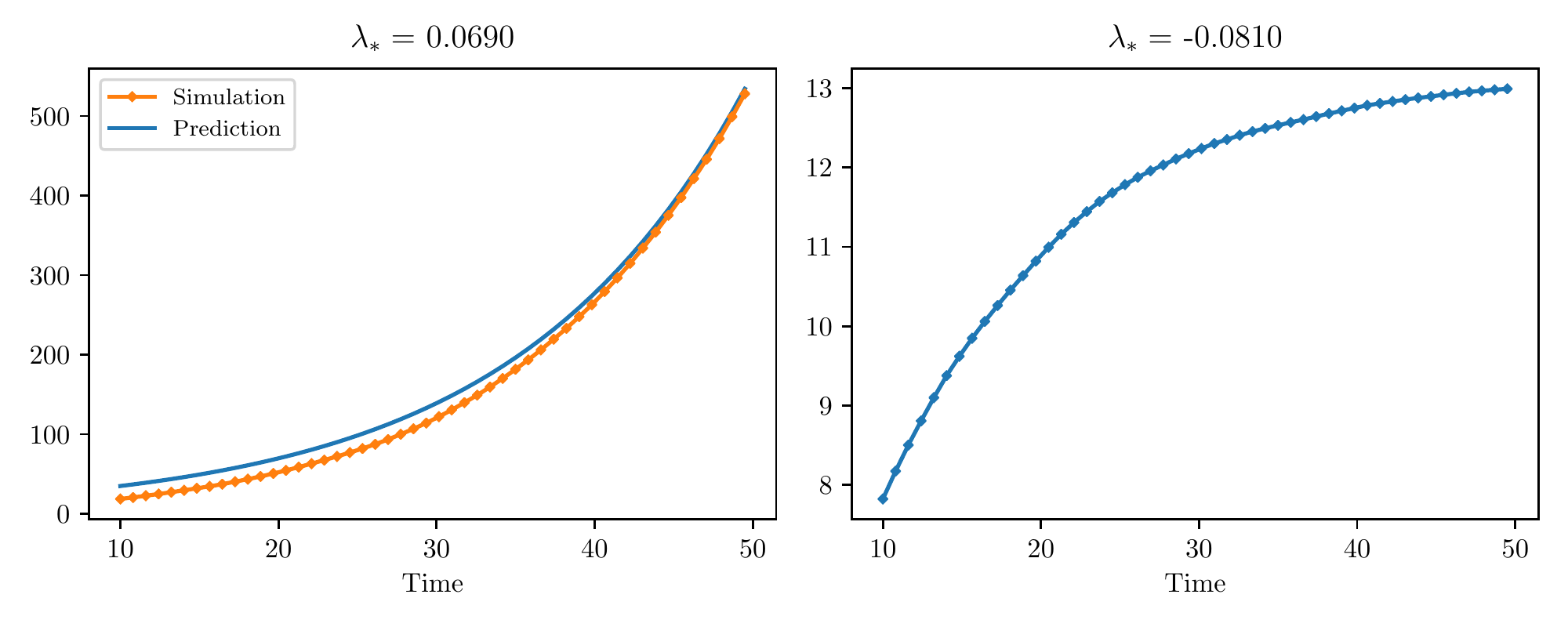}
  \caption{\label{fig:sub-cost} {\bf Supercritical vs.~subcritical. } Mean memory costs in a supercritical case (left) and a subcritical case (right) for the NBP. In the supercritical case, this is compared to the prediction derived from \eqref{eq: cost_asy_sur}.}
\end{figure}

\begin{figure}[htp]
  \includegraphics[width=\textwidth]{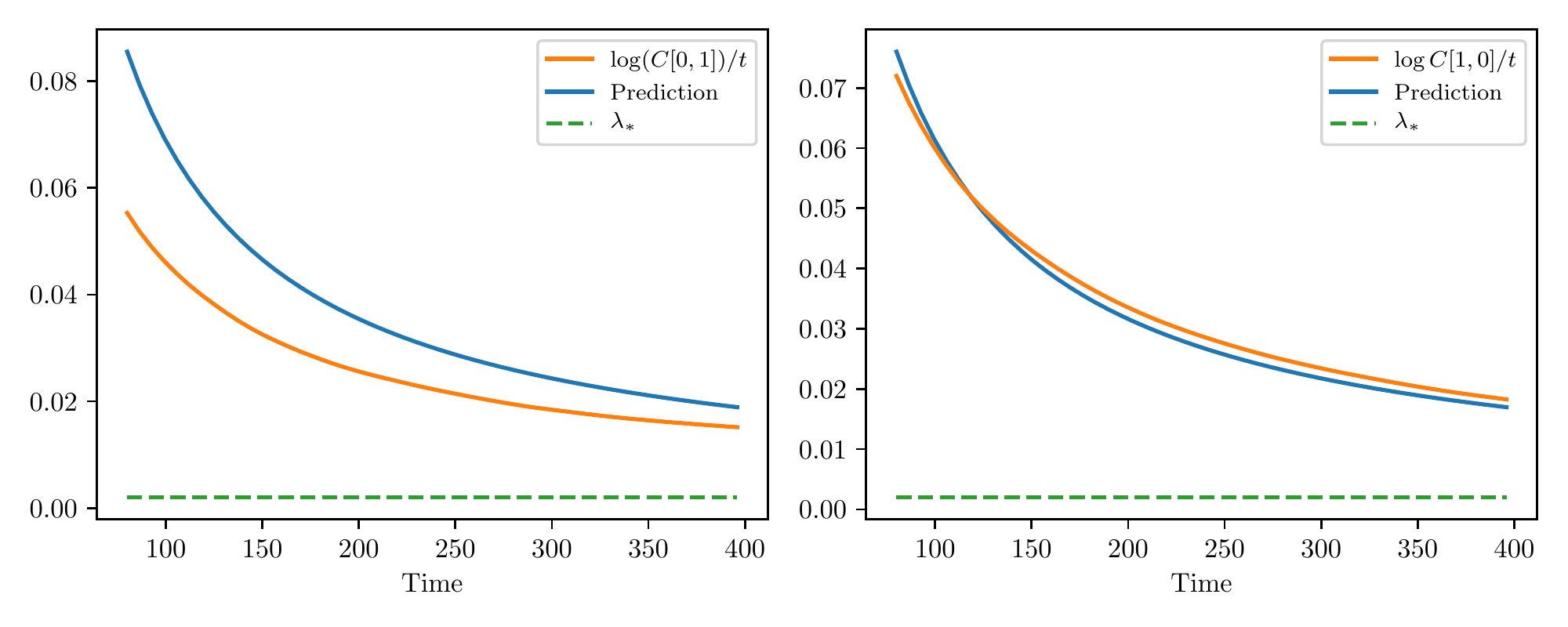}
  \caption{\label{fig:cost-br} (Left) The orange curve is the plot of $t\mapsto \frac1t\log C[\mathbf 0, \mathbf 1](t)$ averaged over $500$ simulations of the NBP. The blue line depicts the predicted growth rate given in the RHS of \eqref{eq: cost_asy_sur} and the green line indicates the value of $\lambda_{\ast}$. (Right) The orange curve is the plot of $t\mapsto \frac1t\log C[\mathbf 1, \mathbf 0](t)$ averaged over $500$ simulations of the NBP. The blue line depicts the predicted growth rate given in the RHS of \eqref{eq: cost_asy_sur} and the green line indicates the value of $\lambda_{\ast}$. }
\end{figure}

\medskip {\bf NRW and $h$-NRW estimators. } Since the particles are typically killed on the boundary after some period of time, the basic NRW estimator $\Psi^{\texttt{rw}}_{k}[g](t, r, \upsilon)$ performs worse than the branching estimator (Fig.~\ref{fig:NRW}).
\smallskip

The $h$-NRW estimators depends on the choice of $h$. We see this by considering three cases based on \eqref{ex:h} and \eqref{ex:h'}.  See Fig.~\ref{fig:1d-hRW} for a plot of the $h$'s and a comparison of the estimators they yield.  The pictures there suggest that both $h_{1}$-NRW and $h_{3}$-NRW estimators outperformed the NBP estimator in this example. Note also, based on the discussion in the previous section, that the `poor' choice of $h$, \emph{viz.} $h_2$, systematically underestimates $\lambda_*$.

\begin{figure}
  \includegraphics[width=\textwidth]{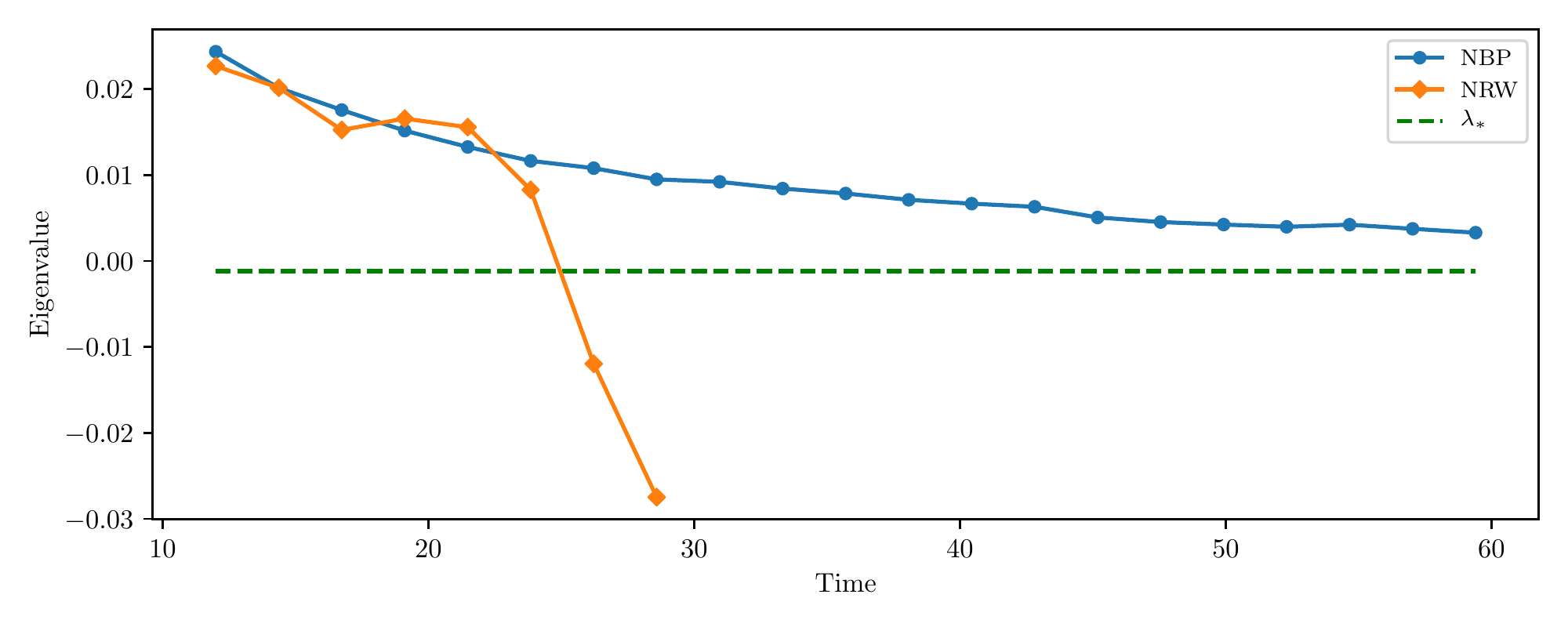}
\caption{\label{fig:NRW} {\bf NBP vs.~NRW. } Estimates of $\lambda_*$ for the NBP and NRW estimators in a supercritical case. In the NRW case, the plot stops after all particles are killed. The instability is due to the very small numbers of surviving particles at large times. The true eigenvalue here is $\lambda_\ast \approx -0.00118$.}
\end{figure}

\begin{figure}
  \includegraphics[width=\textwidth]{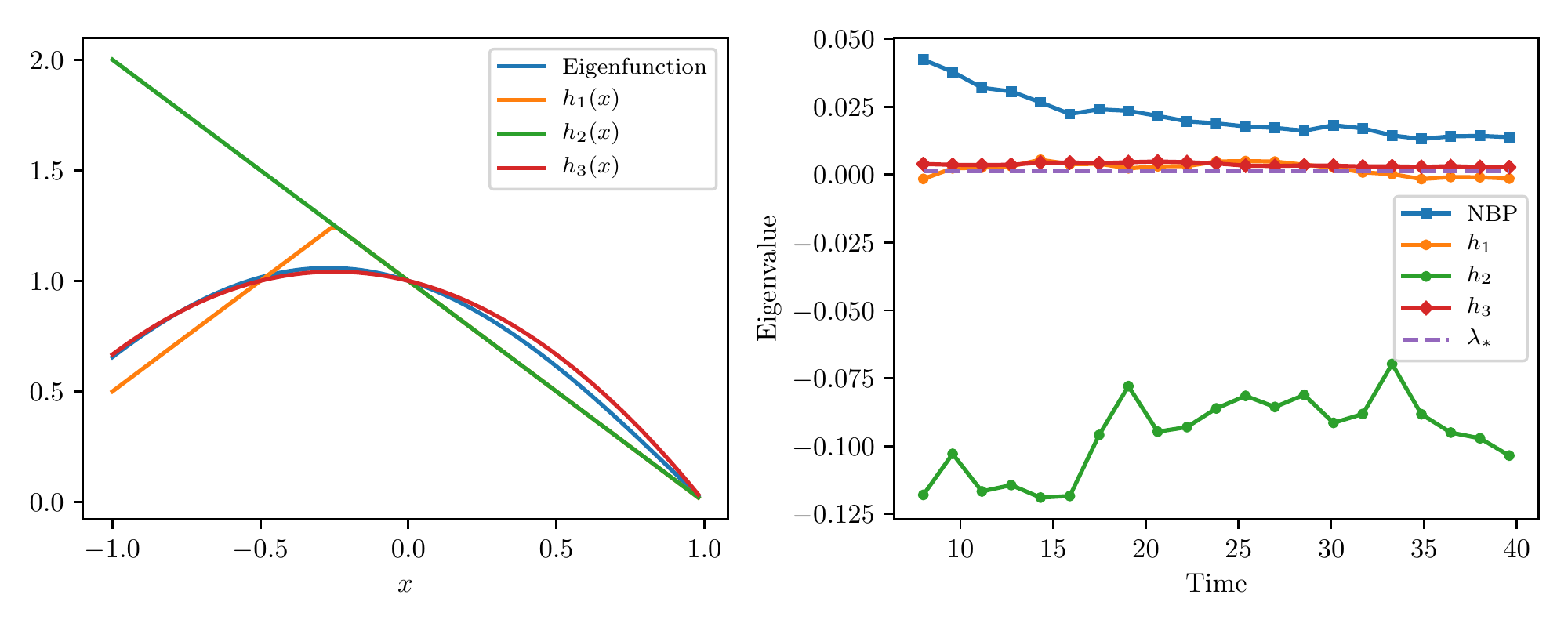}
  \caption{\label{fig:1d-hRW} {\bf  $h$-NRW estimators. } (Left) Plots of three different choices of $h$ compared to the true eigenfunction (in blue). (Right) Comparison of these three $h$-NRW eigenvalue estimators, together with the NBP estimator ($k=1000$ simulations). }
\end{figure}

\subsection{A two dimensional reactor with four rods} 
We compare the results above with a 2D particle simulation. In this case, we model our environment as a square reactor with 4 rods, where most branching happens, and with particles killed on exit from the square. We assume all scattering events are uniform, and the particles have constant speed.

\smallskip
Numerically, the picture is similar to the 1D case. The basic NRW estimator suffers from the rapid loss of the particles (Fig.~\ref{fig:NRW}). For the $h$-NRW, we use the choice of  $h$ given in \eqref{2consts}.
Estimates over time for the eigenvalue by the branching estimator and the $h$-NRW estimator can be found in Fig.~\ref{fig:2d-hRW}.

\begin{figure}[htp]
  \includegraphics[width=0.45\textwidth]{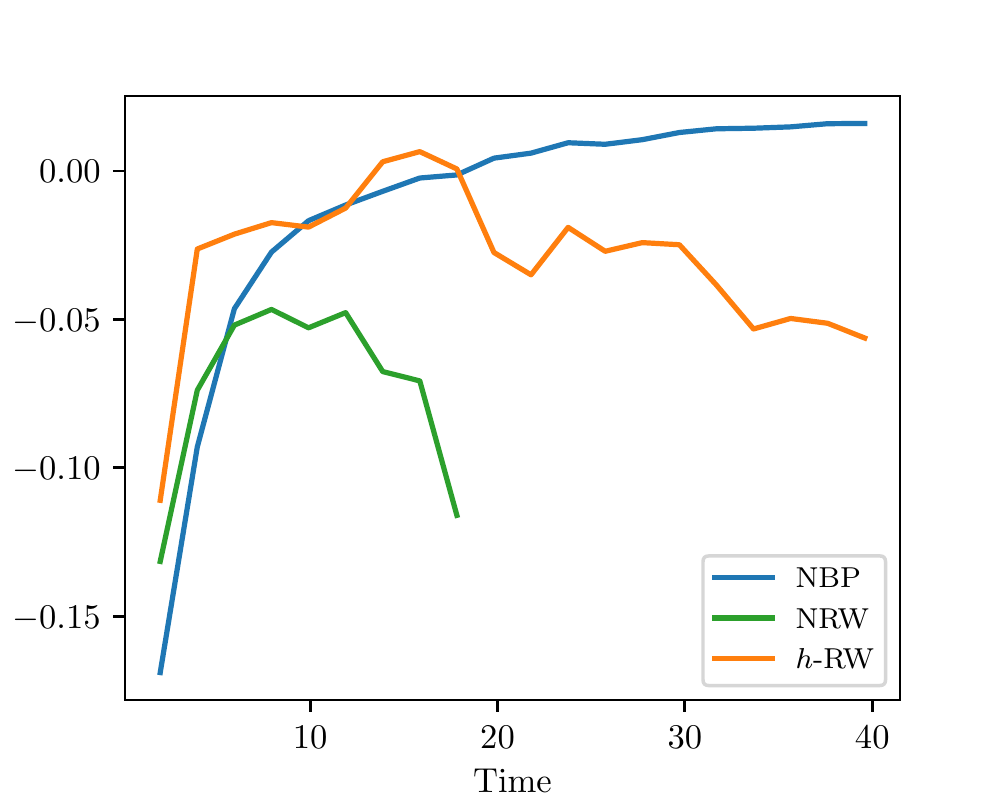} 
  \includegraphics[width=0.45\textwidth]{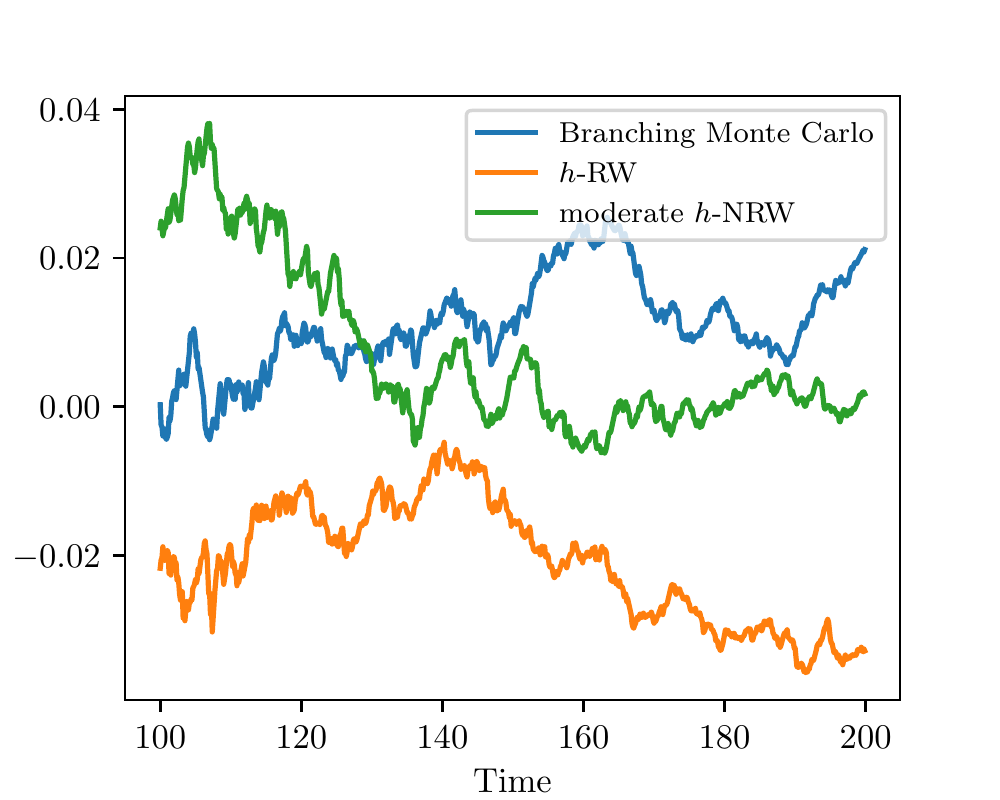} 
\caption{\label{fig:2d-hRW} {\bf $h$-NRW estimators for 2D reactor. } With the choice of $h$ given in \eqref{2consts}. Here we choose $\eta = 0.02, C = 0.1$. (Left) The estimates for the NBP, NRW and $h$-NRW cases. The particles in the NRW (started with 50,000 particles) all die out by time 20. The $h$-transform estimate is based on 200 particles, the NRW estimate is for 400 particles. The drop-off in the estimates of the $h$-transform method appear to be due to weight-degeneracy of the particles. (Right) The eigenvalue estimated using a particle filter approach. Particles are resampled based on the weight of current particles to minimise the effects of weight degeneracy. The `moderate h-NRW' uses a mild version of the $h$-transform which is closer to the NRW.}
\end{figure}
\smallskip
In Figure~\ref{fig:2d-hRW} we see the advantages and disadvantages of the $h$-transform method. In the simplest case (left), where particles are simulated until death, the NRW suffers from extinction of its particles, even though a large number of particles (50,000) are initially simulated. As a result, the estimate of $\lambda_*$ is very unreliable, and there is little hope to get the particles close to equilibrium before most particles die out. The $h$-NRW case does much better on this count, with particles surviving until large times, however the particles in this case suffer from weight degeneracy, where a few particles have large weight (the exponential, product and indicator terms in \eqref{3rd}), and many particles have comparatively small weight. As a result, the final computations of the average will be dominated by the small number of particles with large weight, leading to wasted computational effort.

\smallskip
To mitigate the weight degeneracy of the particles, we implement a particle filter algorithm. Roughly speaking, this means that at fixed times, we look at our collection of particles, and randomly choose to discard/keep particles based on the weight of individual particles, for example, by choosing particles in the next generation by independently selecting particles from the current population proportional to their current weights. In our implementation, we use a simple multinomial resampling scheme. It is known that more sophisticated resampling schemes can improve the performance of the filter \cite{douc_comparison_2005}, although we do not expect substantial benefits in this setting. We also resample our particles at fixed time-steps. More sophisticated algorithms might choose only to resample when a simple measure of the weight-diversity, such as the \emph{effective sample size} drops below a chosen threshold. Since we are not trying to measure directly the algorithmic performance in this paper, we have not tried to complicate our implementation of the particle filter in this paper.

\smallskip
Particle filtering  algorithms are well studied, see for example \cite{del_moral_feynman-kac_2004}, or \cite{angeli_limit_2019}. In this case, the $h$-transform methods appear to be slightly more stable than the branching process approximation, although it is hard to judge the accuracy here due to the lack of an analytical solution. In addition, in our implementation of this method, the $h$-transform approach is substantially slower, due to the complexity of simulating paths from the conditioned process, compared to the NBP. However in industrial implementations, the complexity of the NBP will increase as the environment gets more sophisticated, and so conditioning may become more competitive with the NBP approach, particularly when combined with smarter ways to calibrate the $h$ function. We leave details of these improvements as the subject of future work.

\appendix

\section{Monte Carlo algorithms}

In this appendix, we highlight briefly the approach to simulating a NRW,  a NBP and a NRW with importance sampling (respectively), which themselves feed into the basic Monte Carlo algorithms in Section \ref{MCsect} and Section \ref{importance}. 
We assume that the basic data of $D, V$ and $\sigma_{\texttt s}, \pi_{\texttt s}, \sigma_{\texttt f}, \pi_{\texttt f}$ are all given.  

\subsection{Simulating $\alpha\pi$-NRW}
We start by defining the distributions, 
\[
\mu^{\texttt{s}}_{r,\upsilon}(u,\infty):  = {\rm e}^{-\int_0^u\alpha(r+\upsilon s, \upsilon)\d s} \qquad u\geq 0, r\in D, \upsilon\in V
\]
and
\[
\eta^{\texttt{s}}_{r,\upsilon}(\d \upsilon')   = \pi(r,\upsilon, \upsilon')\d \upsilon', \qquad r\in D, \upsilon, \upsilon'\in V.
\]

The algorithm below for generating a NRW will take as input the starting position and velocity, $(r,v)$, and the terminal time $t$. For later use, we will also allow the particle to start at a positive time $\texttt{t}_0 \in [0,t)$.

\smallskip
\hrule
\begin{alg}[Simulation of an $\alpha\pi$-NRW] \label{A1}
Given an initial configuration $(r, \upsilon)$, an initial time $\texttt{\emph t}_0$ and a time horizon $t$: 

\smallskip

{\bf Step 1:} Set $n = 0$, $(\texttt{\emph r}_0,\texttt{\emph v}_0) = (r,\upsilon)\in D\times V$. 

\smallskip

{\bf Step 2:} Given the value of $\texttt{\emph t}_n,\texttt{\emph r}_{n}, \texttt{\emph v}_{n}$, sample $\delta_{n+1}$ with law $\mu^{\texttt{\emph s}}_{\texttt{\emph r}_{n},\texttt{\emph v}_{n}}$ and set  
\[
\texttt{\emph t}_{n+1} = \texttt{\emph t}_n+\delta_{n+1}, \quad \texttt{\emph r}_{{n+1}} = \texttt{\emph r}_{n} + \texttt{\emph v}_{n}\delta_{n+1}.
\]

{\bf Step 3:} Check if $\texttt{\emph r}_{{n+1}}\in D$ and $\texttt{\emph t}_{n+1} < t$. If true, sample $\texttt{\emph v}_{{n+1}}$  with law $\eta^{\texttt{s}}_{\texttt{\emph r}_{n+1}, \texttt{\emph v}_{n}}$, increase $n$ by one and go to Step 2. If false, set 
\[
\texttt{\emph t}_{\rm{end}} = \min\{\texttt{\emph t}_{n}+\inf\{s>0: \texttt{\emph r}_{{n}}+\texttt{\emph v}_{{n}}s\notin D\},t\}.
\]

\smallskip

{\bf Step 4:} Return $((\texttt{\emph t}_{k}, \texttt{\emph r}_{k}, \texttt{\emph v}_{k}), 0 \le k \le n)$ and $\texttt{\emph t}_{\rm{end}}$.

\smallskip

\end{alg}
\hrule
\smallskip

We note that there are further algorithmic simplifications within the above description. For example, sampling from the distribution $\mu^{\texttt{s}}_{r,\upsilon}$ can be done as follows.  Suppose that $\mathbf{e}$ is a unit mean exponential random variable. For each $r\in D$ and $\upsilon \in V$, the law $\mu_{r,\upsilon}(\cdot)$ agrees with that of $ \inf\{t>0: \textstyle{\int_0^t}\beta(r+\upsilon s, \upsilon)\d s>{\mathbf{e}}\}$. In the case that $\beta$ is bounded above by a constant, this is further simplified through the use of a simple rejection sampling algorithm.

\smallskip

The above algorithm returns a sequence of variables $\texttt{z} := \{((\texttt{t}_k, \texttt{r}_k, \texttt{v}_k), {0\le k \le n}), \texttt{t}_{\rm{end}} \}$, for some $n$, and a final time $\texttt{t}_{\rm{end}} \le t$. Together these fully determine the trajectory, say $((\texttt{r}_s^{\texttt{z}}, \texttt{v}_s^{\texttt{z}}), s < \texttt{t}_{\rm{end}})$, of a neutron random walk, in the following way:
\begin{equation} \label{eq:trajdefn}
  \texttt{r}_s^{\texttt{z}} = \texttt{r}_k + \texttt{v}_k(s-\texttt{t}_k) \quad\text{and}\quad \texttt{v}_s^{\texttt{z}}=\texttt{v}_k, \quad \text{if}\quad \texttt{t}_k\le s< \texttt{t}_{k+1}\wedge \texttt{t}_{\rm{end}},
\end{equation}
where we take $\texttt{t}_{n+1} = \infty$.  By simulating many particles in this manner, the value of the semigroup may be estimated via \eqref{Psirw}.

\subsection{Simulating $(\sigma_{\texttt {s}}, \pi_{\texttt{s}}, \sigma_{\texttt{f}}, \mathcal{P})$-NBP when only $\sigma_{\texttt {s}}, \pi_{\texttt{s}}, \sigma_{\texttt{f}}, \pi_{\texttt{f}}$ is given} Building on Section \ref{A1}, we have the following algorithm which generates the branching particle system. We need the following notation. Let $\texttt{z}$ be a trajectory as defined above. Then we define a measure on $[\texttt{t}_{0}^\texttt{z},\infty) \cup \{\infty\}$ by:
\[
  \mu^{\texttt{f}}_{\texttt{z}}((u,\infty) \cup \{\infty\}):  = {\rm e}^{-\int_{\texttt{t}_{0}^\texttt{z}}^{u\wedge \texttt{t}_{\rm{end}}^\texttt{z}}\sigma_{\texttt{f}}(\texttt{r}_s^{\texttt{z}} , \texttt{v}_s^{\texttt{z}} )\d s},  \qquad u\geq \texttt{t}_{0}^\texttt{z}
\]
and a measure on $V$ by
\[
\eta^{\texttt{f}}_{\texttt{z},t}(\d \upsilon')   = \frac{\pi_{\texttt{f}}(\texttt{r}_t^{\texttt{z}} , \texttt{v}_t^{\texttt{z}},\upsilon')\d \upsilon'}{\pi_{\texttt{f}}(\texttt{r}_t^{\texttt{z}} , \texttt{v}_t^{\texttt{z}}, V)}, \qquad
 \upsilon'\in V.
\]

\smallskip

In order to simulate a NBP with given data $\sigma_{\texttt s}, \pi_{\texttt s}, \sigma_{\texttt f}, \pi_{\texttt f}$, we need to supply additional information, as the family of densities $\pi_{\texttt f}(r,\upsilon,\cdot)$, $r\in D, \upsilon\in V$ only tells us the mean behaviour of the point process of outgoing fission velocities, whose probabilities were denoted by $\mathcal{P}_{(r,\upsilon)}$, $r\in D, \upsilon\in V$.
Moreover, there is no unique way to choose $\mathcal{P}_{(r,\upsilon)}$, $r\in D, \upsilon\in V$ given $\pi_{\texttt f}(r,\upsilon,\cdot)$, $r\in D, \upsilon\in V$.

\smallskip
 
This does not present a problem however, to the contrary it presents an opportunity. As all of our estimators that are based on the NBP are built around the notion of mean growth, we are at our liberty to choose a convenient $(\mathcal{P}_{(r,\upsilon)}, r\in D, \upsilon\in V)$, and, as we shall see below, an obvious choice is that of a Poisson random field with intensity density given by $\pi_{\texttt f}(r,\upsilon,\cdot)$, $r\in D, \upsilon\in V$. As a small remark, this choice is unrealistic as a physical model given the assumption (H2), however for the synthetic purposes of Monte Carlo simulation, it is very convenient both practically and mathematically.

\smallskip
Below, we give the algorithm to produce the paths of the NBP. Note that, unlike the NRW, the output of the NBP is a random number of trajectories that may have birth times $b^i = \texttt{t}_{0}^i$ and death times $\texttt{t}_{\rm{end}}^i$ which may be strictly postive, and strictly before the end time $t$. The object returned will therefore be a set of trajectories (via \eqref{eq:trajdefn}) of the form $\{((\texttt{t}_k^i, \texttt{r}_k^i, \texttt{v}_k^i), {0\le k \le n^i}), \texttt{t}_{\rm{end}}^i\}_{i = 1}^N$, where $N$ is the (random) total number of particles in the branching process.
\smallskip

\hrule
\begin{alg}[Simulation of a $(\sigma_{\texttt {s}}, \pi_{\texttt{\emph s}}, \sigma_{\texttt{f}}, \pi_{\texttt{f}})$-NBP]
\label{A2}
Given an initial configuration $(r_1, \upsilon_1),\cdots (r_\ell, \upsilon_ \ell)$ and a time horizon $t$:

\smallskip {\bf Step 1:} Set $\ell^0 = \ell$, $b_i = 0, i = 1,\cdots, \ell^0$, $\mathcal{X}_0=\{(r_1, \upsilon_1, b_1),\cdots, (r_\ell, \upsilon_\ell, b_{\ell} )\}$ and $n = 0$.

\smallskip
{\bf Step 2:} For each $i \in \{1, \dots, \ell^n\}$ and corresponding $(r_i,\upsilon_i, b_i)\in \mathcal{X}_n$ such that $b_i <t$,  run Algorithm \ref{A1} to produce a $\sigma_{\texttt {\emph s}}\pi_{\texttt{\emph s}}$-NRW with initial configuration $(r_i, \upsilon_i)$, birth time $b_i$ and terminal time $t$.  Denote the resulting NRW by $\texttt{\emph z}^{i,n}:= \{((\texttt{\emph r}^{i,n}_k, \texttt{\emph t}^{i,n}_k, \texttt{\emph v}^{i,n}_k), 0\leq k\le m^{i,n}), \texttt{\emph t}^{i,n}_{\rm{end}}\}$.

\smallskip {\bf Step 3:} For each $i \in \{1, \dots, \ell^n\}$ and corresponding $(r_i,\upsilon_i, b_i)\in \mathcal{X}_n$ such that $b_i <t$, sample $\gamma^{i,n}$ with law $\mu^{\texttt{\emph f}}_{\texttt{\emph z}^{i,n}}$ 
and then, if $\gamma^{i,n}\le\texttt{\emph t}_{\rm end}^{i,n}$:
\begin{enumerate}
\item Sample $\xi^{i,n}$ independent variables $u^{i,n}_1, \dots, u^{i,n}_{\xi^{i,n}}$ with probability density $\eta^{\texttt{\emph f}}_{\texttt{\emph z}^{i,n}}$, where $\xi^{i,n}$ is Poisson distributed with parameter $\pi_{\texttt{\emph f}}(\texttt{\emph r}^{i,n}_{\gamma^{i,n}}, \texttt{\emph v}^{i,n}_{\gamma^{i,n}}, V)$; and
\item Trim the path $\texttt{\emph z}^{i,n}$ at time $\gamma^{i,n}$, that is, set $\tilde{m}^{i,n}$ to be the largest integer $k \le m^{i,n}$ such that $\texttt{t}_{k}^{i,n} \le \gamma^{i,n}$, and redefine $\texttt{\emph z}^{i,n}:= \{((\texttt{\emph r}^{i,n}_k, \texttt{\emph t}^{i,n}_k, \texttt{\emph v}^{i,n}_k), 0\leq k\le \tilde{m}^{i,n}), \gamma^{i,n}\}$.
\end{enumerate}

\smallskip
{\bf Step 4:} Let $\mathcal{X}_{n+1}$ be the set: $\{(\texttt{\emph r}^{i,n}_{\gamma^{i,n}}, u^{i,n}_{j}, \gamma^{i,n}); 1 \le i \le \ell^n, 1 \le j \le \xi^{i,n}, \texttt{\emph t}^{i,n}_0 < t\}$. Set $\ell^{n+1} := |\mathcal{X}_{n+1}|$.

\smallskip

{\bf Step 5:} If $\ell^{n+1} >0$, go to Step 2. Otherwise stop and return $\mathcal{X} = \{\texttt{\emph z}^{i,n}; 1 \le i \le \ell^n, n \in \N\}$.

\end{alg}
\hrule
\smallskip

In this algorithm, we then estimate the semigroup using \eqref{def: Psi_br}, so we have (in the notation of the algorithm above):
\begin{equation*}
  \Psi_k[g](t, r, \upsilon) = \frac{1}{k} \sum_{\texttt{z} \in \mathcal{X}} g(\texttt{r}_t^{\texttt{z}}) \mathbf{1}_{(\texttt{t}_{\rm end}^{\texttt{z}} \ge t)}.
\end{equation*}

%
%
%

\newpage

\begin{wrapfigure}{r}{0.5\textwidth}
  \begin{center}
\includegraphics[height=5.5cm]{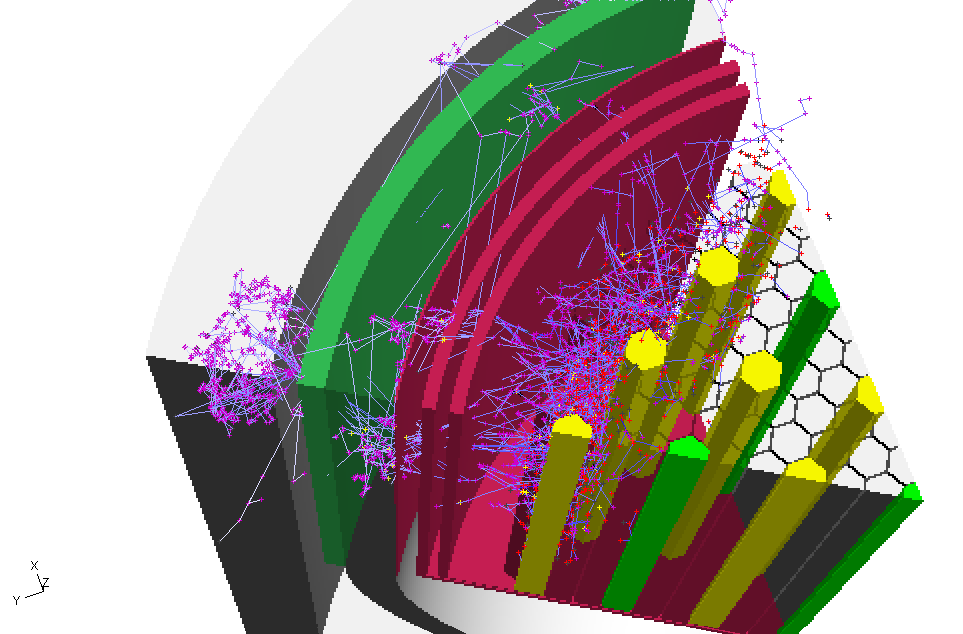}
  \end{center}
  \vspace{-10pt}
\caption{An example of a slice through a mock nuclear reactor design in which some neutron path simulations are depicted. The reactor core displays symmetry.}
\label{cake}
\end{wrapfigure}
Let us make some practical notes on the simulation of a $(\sigma_{\texttt {s}}, \pi_{\texttt{s}}, \sigma_{\texttt{f}}, \pi_{\texttt{f}})$-NBP, pertaining in particular to how such simulations are performed in industrial software. The geometric structure of a virtual reactor will be designed on a special CAD tool. After that, according to the physical properties of the materials used, the values of the four cross-sections $(\sigma_{\texttt {s}}, \pi_{\texttt{s}}, \sigma_{\texttt{f}}, \pi_{\texttt{f}})$ are mapped from data libraries into the different geometrical regions. This generates numerically the piecewise constant functions $(\sigma_{\texttt {s}}, \pi_{\texttt{s}}, \sigma_{\texttt{f}}, \pi_{\texttt{f}})$ across the entire space-velocity domain $D\times V$. The construction of this numerical representation of the cross-sections in a virtual nuclear reactor can be stored in a file whose size\footnote{This is information which has been shared with us by our industrial partner at the ANSWERS group in {\it Jacobs}.} is as little as 150MB. Thereafter, the NBP simulation calls upon the values in this file 
as it samples from e.g. $\mu^{\texttt{f}}_{r,\upsilon}$ and $\eta^{\texttt{f}}_{r,\upsilon}$.

\subsection{Simulating $\alpha\pi$-NRW Monte Carlo with importance sampling}

In conclusion, a second proposed efficiency to the Monte Carlo simulation of solutions to the NTE  is to make an educated guess at the shape of $\varphi$, the function $h$ and then to adapt Algorithm \ref{A1} as follows.
\smallskip

\hrule
\begin{alg}\label{alg2}
$\mbox{ }$
\smallskip

Run Step 1 - 4 of Algorithm \ref{A1} replacing the quantities $\alpha$, $\beta$ and $\pi$ by
\[
\alpha^h(r,\upsilon) =
\frac{\alpha(r,\upsilon)\int_Vh(r,\upsilon')\pi(r,\upsilon,\upsilon') \,\d\upsilon'}{h(r,\upsilon)}, \qquad 
\beta^h (r,\upsilon) = \dfrac{{\ebJ} h  (r,\upsilon)}{h  (r,\upsilon)} + \beta (r,\upsilon) 
\]
and
\[
\pi^h(r,\upsilon,\upsilon') =
\frac{h(r,\upsilon')}{\int_Vh(r,\upsilon')\pi(r,\upsilon,\upsilon')\, \d\upsilon'}\pi(r,\upsilon,\upsilon'),
\]
noting that $\texttt{\emph t}^{\rm{end}}= t$ always.
\end{alg}
\hrule
\smallskip

In this case, the Monte Carlo estimate of the operator $\Psi_k[g]$ is then given by \eqref{Psihrw}.

\section{A many-to-two lemma for NBP}

Recall the notation of \eqref{PP}, \eqref{Erv} and define, for $r\in D$, $\upsilon\in V$, $f, g\in L^+_{\infty}(D\times V)$,
\begin{equation}
\eta_{\texttt{f}}[f, g](r, \upsilon)= \sigma_{\texttt f}(r,\upsilon)\mathcal{E}_{(r,\upsilon)}\left[
\sum_{{i,j = 1},{i\neq j}}^N   f(\upsilon_i)g(\upsilon_j)
\right].
\label{etadef}
\end{equation}
As an abuse of notation, we will also write, for $r\in D$, $\upsilon\in V$, $f\in L_{\infty}(D\times V)$, 
\[
\eta_{\texttt{f}}[f](r, \upsilon):=\eta_{\texttt{f}}[f, f](r, \upsilon)=\sigma_{\texttt f}(r,\upsilon)\mathcal{E}_{(r,\upsilon)}\left[
\sum_{{i,j = 1},{i\neq j}}^N   f(\upsilon_i)f(\upsilon_j)
\right]. 
\]
We have the following two point correlation formula (also called the many-to-two formula; cf \cite{HR}).
\begin{lem}[Many-to-two]\label{M2F}
Suppose that $f,g\in L^{+}_{\infty}(D\times V)$. Then 
\begin{align}
\label{M22}
\mathbb E_{\delta(r, \upsilon)}\Big[\langle f, X_{t}\rangle \langle g, X_t\rangle \Big]&=
\psi_{t}[fg](r, \upsilon)+\int_{0}^{t} \psi_{s}\Big[ \eta_{\texttt{\emph f}}[\psi_{t-s}[f], \psi_{t-s}[g]]\Big](r, \upsilon)\d s. 
\end{align}
\end{lem}
\begin{proof}
The proof follow the  arguments in Harris \& Roberts \cite{HR}, which  can be easily adapted to the current context. 

\smallskip

As a first step, we note that, for  $h\in L^{+}_{\infty}(D\times V)$ and $H: \mathbb R_{+}\times D\times V\to \mathbb R_{+}$ continuous and bounded, the function 
\begin{align}
u_{t}(r, \upsilon) &:=  \mathbf{E}_{(r,\upsilon)}\left[{\rm e}^{\int_0^t\beta(R_s, \Upsilon_s)\d s}h(R_t, \Upsilon_t) \mathbf{1}_{(t < \tau^D)}\right]+ \mathbf{E}_{(r,\upsilon)}\left[\int_{0}^{t\wedge \tau^{D}} \!\!\!\!\!\!\!\! H(s, R_{s}, \Upsilon_{s}){\rm e}^{\int_0^s\beta(R_u, \Upsilon_u)\d u} \d s\right]\notag\\
&=\psi_t[h](r,\upsilon) + \mathbf{E}_{(r,\upsilon)}\left[\int_{0}^{t\wedge \tau^{D}} \!\!\!\!\!\!\!\! H(s, R_{s}, \Upsilon_{s}){\rm e}^{\int_0^s\beta(R_u, \Upsilon_u)\d u} \d s\right], 
\label{ut}
\end{align}
for $t\geq 0$, $r\in D$ and  $\upsilon \in V$,
uniquely solves the integral equation
\begin{equation}
u_{t}(r, \upsilon)  = \bU_{t}[h](r, \upsilon)+\int_{0}^{t} \bU_{s}\Big[H(s, \cdot) +(\bS+\bF)u_{t-s}\Big](r, \upsilon) \d s.
\label{hasunique}
\end{equation}
We only sketch the proof of this fact for the sake of brevity.  To see why this is the unique solution, it suffices to return to the proof of e.g.  Lemma 6.1 in \cite{MultiNTE} and note that the approach given there works equally well here. That is to say, we first condition the right-hand side of \eqref{ut} on the first fission or scattering event, whichever comes first, to generate a recursion in $u$. This can then be manipulated further with the help of Lemma 1.2, Chapter 4 in \cite{Dynkin2} to obtain \eqref{hasunique}. Finally uniqueness of the latter follows by a standard argument appealing to Gr\"onwall's Lemma.

\smallskip

To complete the proof of \eqref{M22},  it suffices to consider the case $f=g$, as the general form will follow from the polarisation
\[
\mathbb E_{\delta(r, \upsilon)}\Big[2\langle f, X_{t}\rangle \langle g, X_t\rangle \Big]=\mathbb E_{\delta(r, \upsilon)}\Big[\langle f+g, X_{t}\rangle ^{2} \Big]-\mathbb E_{\delta(r, \upsilon)}\Big[\langle f, X_{t}\rangle ^{2} \Big]-\mathbb E_{\delta(r, \upsilon)}\Big[\langle g, X_{t}\rangle ^{2} \Big]. 
\]
To this end, denote by $w_{t}(r, \upsilon)=\mathbb E_{\delta(r, \upsilon)}[\langle g, X_{t}\rangle ^{2}]$, for $t\geq 0$, $r\in D$, $\upsilon\in V$. For convenience, write $\Pi_t(r,\upsilon) = \exp(-\int_0^t \sigma(r+\ell \upsilon, \upsilon)\d \ell)$, for $r\in D$, $\upsilon\in V$.
Once again, by conditioning on the first scattering or fission event and then applying  Lemma 1.2, Chapter 4 in \cite{Dynkin2}, we get 
\begin{align*}
w_{t}(r, \upsilon) &=\Pi_t(r,\upsilon) 
\bU_{t}[g^{2}](r, \upsilon)+\int_{0}^{t}\bU_{s}\Bigg[\sigma_{\texttt{s}}\Pi_s \int_V w_{t-s}(r,\upsilon')\pi_{\texttt{s}}(\cdot,\cdot, \upsilon')\d\upsilon'\Bigg](r, \upsilon)\d s
  \\
&+\int_{0}^{t}\bU_{s}\Bigg[\sigma_{\texttt{f}} \Pi_s \mathcal{E}_{(r,\upsilon)}\left[
\sum_{i,j = 1, i \neq j}^N  \psi_{t-s}[g](\cdot, \upsilon_i)\psi_{t-s}[g](\cdot, \upsilon_j) +\sum_{i = 1}^N  w_{t-s}[g](\cdot, \upsilon_i)
\right]\Bigg](r, \upsilon)\d s \\
& = \bU_{t}[g^{2}](r, \upsilon)+ \int_{0}^{t}\bU_{s}[(\bS+\bF)w_{t-s}](r, \upsilon)\d s+\int_{0}^{t}\bU_{s}\Big[\eta_{\texttt{f}}[\psi_{t-s}[g]] \Big](r, \upsilon)\d s.
\end{align*}
Using the representation of the solution to \eqref{ut} with $h=g^{2}$ and $H(s, \cdot, \cdot)=\eta_{\texttt f}[\psi_{t-s}[g]]$, we get  \eqref{M22}.
\end{proof}

\section{NBP Monte Carlo convergence}

\begin{proof}[Proof of Theorem \ref{cor:var}]
The estimator $\Psi^{\textrm{br}}_{k}$ as defined in \eqref{def: Psi_br} is unbiased. 
Essentially the proof is based around a study of its variance, which amounts to studying the variance of the branching  system, since
\begin{equation}
\mathbb E_{\delta(r, \upsilon)}\Big[\Big(\Psi^{\texttt{br}}_{k}[g](t, r, \upsilon)-\psi_{t}[g](r, \upsilon)\Big)^{2}\Big]=\frac{1}{k}\Big(\mathbb E_{\delta(r, \upsilon)}[\langle g, X_{t}\rangle^{2}]-\psi_{t}[g](r, \upsilon)^{2}\Big).
\label{just2ndmom}
\end{equation}
We shall rely upon the many-to-two formula in Lemma \ref{M2F} to compute the above.  Accordingly, we now claim as an intermediate step.
\begin{lem}\label{lem:NBPConv}
  We have for all $t\geq 0$, $k\in\mathbb{N}$,
\begin{align}\label{eq: variance}
&\mathbb E_{\delta(r, \upsilon)}\Big[\Big(\Psi^{\emph{\texttt{br}}}_{k}[g](t,r, \upsilon)-\psi_{t}[g](r, \upsilon)\Big)^{2}\Big]\\ \notag
&\qquad\qquad =\frac{1}{k}\Big\{\psi_{t}[g^{2}](r, \upsilon)+\int_{0}^{t} \psi_{s}\Big[ \eta_{\texttt{\emph f}}[\psi_{t-s}[g]]\Big](r, \upsilon)\d s - \psi_{t}[g](r, \upsilon)^{2}\Big\}. 
\end{align}
Subsequently, it has the following asymptotics as $t\to\infty$. 
\begin{enumerate}[(i)]
\item \label{cas-crit}
If $\lambda_{\ast}= 0$, then for fixed $k$, 
\[
\lim_{t\to\infty} \frac kt \cdot\mathbb E_{\delta(r, \upsilon)}\Big[\Big(\Psi^{\emph{\texttt{br}}}_{k}[g](t,r, \upsilon)-\psi_{t}[g](r, \upsilon)\Big)^{2}\Big] = \langle \eta_{\texttt{\emph  f}}[\varphi], \tilde\varphi\rangle \langle g, \tilde\varphi\rangle^{2} \varphi(r, \upsilon) := C_{[1]}(g,r,v)\,. 
\]
\item \label{cas-sur}
If $\lambda_{\ast}>0$, then for fixed $k$, 
\begin{align*}
&\lim_{t\to\infty} e^{-2\lambda_{\ast}t}k\cdot\mathbb E_{\delta(r, \upsilon)}\Big[\Big(\Psi^{\emph{\texttt{br}}}_{k}[g](t,r, \upsilon)-\psi_{t}[g](r, \upsilon)\Big)^{2}\Big] \\
&\qquad \qquad\qquad = \langle g, \tilde\varphi\rangle^{2}\Big( \int_{0}^{\infty} e^{-2\lambda_{\ast}s}\psi_{s}\big[\eta_{\texttt{\emph f}}[\varphi]\big](r, \upsilon) \d s-\varphi(r, \upsilon)^{2}\Big)  := C_{[2]}(g,r,v) \in [0, \infty).
\end{align*}
\item \label{cas-sous}
If $\lambda_{\ast} < 0$, then for fixed $k$,
\begin{align*}
&\lim_{t\to\infty} e^{-\lambda_{\ast}t}k\cdot\mathbb E_{\delta(r, \upsilon)}\Big[\Big(\Psi^{\emph{\texttt{br}}}_{k}[g](t,r, \upsilon)-\psi_{t}[g](r, \upsilon)\Big)^{2}\Big] \\
  & \qquad \qquad\quad= \varphi(r, \upsilon) \Big\{ \langle g^{2}, \tilde\varphi \rangle+\int_0^{\infty} e^{-\lambda_{\ast}s}\langle \tilde\varphi, \eta_{\texttt{\emph  f}}\big[\psi_{s}[g]\big]\rangle \d s\Big\}  := C_{[3]}(g,r,v) \in [0, \infty).
\end{align*}
\end{enumerate}
\end{lem}
\begin{proof}
The display \eqref{eq: variance} is an easy consequence of \eqref{M22}. We turn to the asymptotics. Appealing to \eqref{spectralexpsgp}, we have for any $f\in L_{\infty}^{+}(D\times V)$ and $\delta>0$, there exists some constant $K\in (0, \infty)$ and some $t_0=t_0(\delta)$ such that 
\begin{equation}\label{bd}
\sup_{t\ge 0}\|e^{-\lambda_{\ast}t}\psi_t[f]\|_\infty \le K\|f\|_{\infty}\  \text{ and } \ \|e^{-\lambda_{\ast}t}\psi_t[f]-\langle f, \tilde\varphi\rangle \varphi\|_\infty \le \delta\|f\|_{\infty}, \quad  \text{for all } t\ge t_0. 
\end{equation}
On the other hand, it follows from \eqref{etadef} that $\eta_{\texttt{f}}$ is a symmetric bilinear form, hence
\begin{equation}\label{bd-eta}
\left|\eta_{\texttt{f}}[f](r, \upsilon)-\eta_{\texttt{f}}[h](r, \upsilon)\right|=|\eta_{\texttt{f}}[f-h, f+h](r, \upsilon)|\le C \|f-h\|_\infty (\|f\|_\infty+\|h\|_\infty),
\end{equation}
where in the last inequality we have used the fact that $\|\sigma_{\texttt{f}}\|_{\infty}<\infty$ and the offspring number is also uniformly bounded by some constant $B$, so that in above $C:=\|\sigma_{\texttt{f}}\|_{\infty}\, B<\infty$. In particular, taking $h=0$ yields 
\begin{equation}\label{bd-eta'}
\left\|\eta_{\texttt{f}}[f]\right\|_{\infty}\le C \|f\|_\infty^{2}.
\end{equation}
Also, we clearly have 
\begin{equation}\label{mono}
f, g\in L_{\infty}, f\le g \quad \Rightarrow \quad \psi_{t}[f]\le \psi_{t}[g], \quad t\ge 0. 
\end{equation}

Let us look at the supercritical case, i.e.~$\lambda_{\ast}>0$. The arguments consist in extracting the dominant growth rate from \eqref{eq: variance} as $t\to\infty$. To that end, we first split the integral in \eqref{eq: variance} into two parts:
\[
\int_{0}^{t} \psi_{s}\Big[ \eta_{\texttt{f}}[\psi_{t-s}[g]]\Big](r, \upsilon)\d s = \int_{t-t_{0}}^{t} \psi_{s}\Big[ \eta_{\texttt{f}}[\psi_{t-s}[g]]\Big](r, \upsilon)\d s + \int_{0}^{t-t_{0}} \psi_{s}\Big[ \eta_{\texttt{f}}[\psi_{t-s}[g]]\Big](r, \upsilon)\d s. 
\]
Note that the first term on the right-hand side is of order $o(e^{2\lambda_{\ast}t})$. Indeed, 
\begin{align*}
\int_{t-t_{0}}^{t} \psi_{s}\Big[\eta_{\texttt{f}}[\psi_{t-s}[g]]\Big](r, \upsilon)\d s 
&\le\int_{t-t_{0}}^{t} \psi_{s}\Big[ \eta_{\texttt{f}}\big[K\|g\|_{\infty}e^{\lambda_{\ast}(t-s)}\big]\Big](r, \upsilon)\d s \qquad \text{by \eqref{bd} }\\
&\le \int_{t-t_{0}}^{t} \psi_{s}\Big[CK^{2}\|g\|_{\infty}^{2}e^{2\lambda_{\ast}(t-s)}\Big](r, \upsilon) \d s \qquad \text{by \eqref{bd-eta'}} \\
&\le CK^{3}\|g\|_{\infty}^{2} e^{2\lambda_{\ast}t}\int_{t-t_{0}}^{t} e^{-\lambda_{\ast} s}\d s   \qquad \text{by \eqref{mono} and \eqref{bd}}\\
&=o(e^{2\lambda_{\ast}t}), \quad \text{ as } t\to\infty. 
\end{align*}
On the other hand, for the second term, we have
\begin{align*}
&\Big|\int_{0}^{t-t_{0}} \psi_{s}\Big[ \eta_{\texttt{f}}[\psi_{t-s}[g]]\Big](r, \upsilon)\d s-\langle g, \tilde\varphi\rangle^{2} \int_{0}^{t-t_{0}}e^{2\lambda_{\ast}(t-s)}\psi_{s}\Big[ \eta_{\texttt{f}}[\varphi]\Big](r, \upsilon)\d s\Big|\\
&\le \int_{0}^{t-t_{0}} \psi_{s}\Big[ \Big|\eta_{\texttt{f}}\big[\psi_{t-s}[g]\big]-\eta_{\texttt{f}}\big[\langle g, \tilde\varphi \rangle e^{\lambda_{\ast}(t-s)}\varphi\big]\Big|\Big](r, \upsilon)\d s\\
&\le \int_{0}^{t-t_{0}} \psi_{s}\Big[e^{2\lambda_{\ast}(t-s)}\delta CK \|g\|_{\infty}(\|g\|_{\infty}+|\langle g,\tilde{\varphi}\rangle|\|\varphi\|_{\infty})\big] \Big] \d s \quad \text{by \eqref{bd}, \eqref{bd-eta}}\\
&\le \delta CK\|g\|_{\infty}(\delta\|g\|_{\infty}+2|\langle g,\tilde{\varphi}\rangle|\|\varphi\|_{\infty}) e^{2\lambda_{\ast}t}\int_{0}^{t-t_{0}}e^{-\lambda_{\ast}s}\d s \quad \text{by \eqref{bd}}\\
&=O(\delta e^{2\lambda_{\ast}t}), \quad \text{ as } t\to\infty. 
\end{align*}
Since $\delta$ is arbitrary, when combined with the fact that the integral 
\[
\int _{0}^{\infty}e^{-2\lambda_{\ast}s}\psi_{s}[\eta_{\texttt{f}}[\varphi]](r, \upsilon) \d s\le CK\|\varphi\|^{2}_{\infty}\int_{0}^{\infty}e^{-\lambda_{\ast}s}\d s<\infty,
\]
the above implies 
\[
\int_{0}^{t-t_{0}} \psi_{s}\Big[ \eta_{\texttt{f}}[\psi_{t-s}[g]]\Big](r, \upsilon)\d s \overset{t\to\infty}{\sim} \langle g, \tilde\varphi\rangle^{2}e^{2\lambda_{\ast}t}\int _{0}^{\infty}e^{-2\lambda_{\ast}s}\psi_{s}[\eta_{\texttt{f}}[\varphi]](r, \upsilon) \d s. 
\]
The asymptotics in the supercritical case then easily follows. 
Next, we consider the subcritical case, i.e.~$\lambda_{\ast}<0$. We start with a change of variable:
\[
\int_{0}^{t}\psi_{s}\Big[ \eta_{\texttt{f}}\big[\psi_{t-s}[g]\big]\Big](r, \upsilon) \d s = \int_{0}^{t}\psi_{t-s}\Big[ \eta_{\texttt{f}}\big[\psi_{s}[g]\big]\Big](r, \upsilon) \d s. 
\]
Note that by \eqref{bd} and \eqref{bd-eta'}, 
\[
\int_{t-t_0}^{t}\psi_{t-s}\Big[ \eta_{\texttt{f}}\big[\psi_{s}[g]\big]\Big](r, \upsilon) \d s\le CK^{2}\|g\|_{\infty}^{2}\int_{t-t_{0}}^{t} e^{2\lambda_{\ast}s+\lambda_{\ast}(t-s)}\d s= o(e^{\lambda_{\ast}t}),  \quad \text{as } t\to\infty,
\]
since $\lambda_{\ast}<0$. 
On the other hand, applying \eqref{bd} to $\psi_{t-s}$ and $\eta_{\texttt{f}}\big[\psi_{s}[g]\big]$, noting the latter is bounded by $CK^{2} e^{2\lambda_{\ast}s}\|g\|_{\infty}^{2}$ as consequence of \eqref{bd-eta'} and \eqref{mono}, we get
\begin{align*}
\bigg|\int_{0}^{t-t_0} \psi_{t-s}\Big[ \eta_{\texttt{f}}\big[\psi_{s}[g]\big]\Big](r, \upsilon)\d s &- e^{\lambda_{\ast}t}\varphi(r, \upsilon)\int_0^{t-t_0} e^{-\lambda_{\ast}s}\langle \tilde\varphi, \eta_{\texttt{f}}\big[\psi_{s}[g]\big]\rangle \d s\bigg|\\
&\le \int_0^{t-t_0} \delta e^{\lambda_{\ast}(t-s)}
\big\|\eta_{\texttt{f}}\big[\psi_{s}[g]\big]\big\|_{\infty} \d s \\
&\le \int_0^{t-t_0} \delta e^{\lambda_{\ast}(t+s)}CK^{2} \|g\|_{\infty}^{2} \d s = O(\delta e^{\lambda_{\ast}t}), \quad\text{as } t\to\infty. 
\end{align*}
Arguing as in the previous case, we conclude that 
\[
\int_{0}^{t}\psi_{t-s}\Big[ \eta_{\texttt{f}}\big[\psi_{s}[g]\big]\Big](r, \upsilon) \d s \overset{t\to\infty}{\sim}
e^{\lambda_{\ast}t}\varphi(r, \upsilon)\int_0^{\infty} e^{-\lambda_{\ast}s}\langle \tilde\varphi, \eta_{\texttt{f}}\big[\psi_{s}[g]\big]\rangle \d s,
\]
which in turn implies the result in the subcritical case. The proof in the critical case follows from similar arguments and is therefore omitted. 
\end{proof}

\smallskip We begin by noting the following estimates, which follow from the concavity of the function $h_t(x):= x^{1/t}$. Fix $x_0 \in (0,\infty)$, and consider $x \ge 0$. Then
\begin{align*}
  0 & \le  h_t(x) - h_t(x_0)  \le  (x-x_0) \frac{x_0^{1/t-1}}{t}, \qquad x \ge x_0\\
  0 & \le  h_t(x_0) - h_t(x)  \le  (x_0-x) \,x_0^{1/t-1}, \qquad x \le x_0
\end{align*}
From which we deduce the inequality for all $x \ge 0$
\begin{equation*}
  (h_t(x)-h_t(x_0))^2 \le (x-x_0)^2 \left( x_0^{1/t-1} \max\left\{1,\frac{1}{t}\right\}\right)^2.
\end{equation*}
Now returning to the  estimation of the lead eigenvalue, by appealing to the above inequality and $(x+y)^2\le 2x^2+2y^2$ we get
\begin{align}
&\mathbb E\Big[\Big(\left(\Psi^{\texttt{br}}_k[g](t, r, \upsilon)\right)^{1/t}-\me^{\lambda_\ast}\Big)^2\Big]\notag\\
&\hspace{2cm}\le \frac{2}{\min\{1,t^2\}} \frac{\mathbb E\big[\big(\Psi^{\texttt{br}}_k[g](t, r, \upsilon)-\psi_t[g](r, \upsilon)\big)^2\big]}{\psi_t[g](r, \upsilon)^{2-2/t}}+ 2\Big(\left(\psi_t[g](r,v)\right)^{1/t}-\me^{\lambda_\ast}\Big)^2.
\label{powerestimate}
\end{align}

The proof  of Theorem \ref{cor:var} now follows from the above inequality, \eqref{spectralexpsgp} and the estimates in (i) - (iii) of Lemma \ref{lem:NBPConv}. 
\end{proof}
 
\section{NRW Monte Carlo convergence}

We need the following intermediate result, before turning to the proof of Theorem \ref{NRWvar}.
\begin{lem}\label{interlem} Under the conditions of Theorem \ref{NRWvar}, we have 
\begin{equation}
\max(2\lambda_{\ast}, \lambda_{\ast}+\underline\beta)\le \lambda_1\le  \lambda_{\ast}+\overline{\beta}. 
\label{primeclaim}
\end{equation}
\end{lem}

\begin{proof}
We have, on the one hand, that 
\begin{align}
\psi^1_t[g^2](r,\upsilon)  &:=  \mathbf{E}_{(r,\upsilon)}\left[{\rm e}^{\int_0^t 2\beta(R_s, \Upsilon_s)\d s}g(R_t, \Upsilon_t)^2 \mathbf{1}_{(t < \tau^D)}\right]\notag\\
&\leq 
{\rm e}^{\bar\beta t}  \mathbf{E}_{(r,\upsilon)}\left[{\rm e}^{\int_0^t \beta(R_s, \Upsilon_s)\d s}g(R_t, \Upsilon_t)^2 \mathbf{1}_{(t < \tau^D)}\right]
= {\rm e}^{\bar\beta t} \psi_t[g^2](r,\upsilon)
\label{upper1}
\end{align}
On the other hand, by Jensen's inequaltiy, 
\begin{align}
\psi^1_t[g^2](r,\upsilon) 
&\geq 
\mathbf{E}_{(r,\upsilon)}\left[{\rm e}^{\int_0^t \beta(R_s, \Upsilon_s)\d s}g(R_t, \Upsilon_t) \mathbf{1}_{(t < \tau^D)}\right]^2= \psi_t[g](r,\upsilon)^2
\label{lower1}
\end{align}
and also that 
\begin{align}
\psi^1_t[g^2](r,\upsilon) 
&\geq {\rm e}^{\underline\beta t}  
\mathbf{E}_{(r,\upsilon)}\left[{\rm e}^{\int_0^t \beta(R_s, \Upsilon_s)\d s}g(R_t, \Upsilon_t)^2 \mathbf{1}_{(t < \tau^D)}\right]^2= {\rm e}^{\underline{\beta} t}  \psi_t[g^2](r,\upsilon)
\label{lower2}
\end{align}
By taking logarithms, dividing by $t$ and then taking $t\to\infty$ using  Theorem \ref{CVtheorem} in \eqref{upper1} and \eqref{lower1}, \eqref{lower2}  respectively,   yields the desired bounds. 
\end{proof}

\begin{proof}[Proof of Theorem \ref{NRWvar}]
In a similar spirit to \eqref{just2ndmom}, we note that 
\begin{align}
&\mathbf{E}_{(r,\upsilon)}\Big[\Big(\Psi^{\texttt{rw}}_{k}[g](t, r, \upsilon)-\psi_{t}[g](r, \upsilon)\Big)^{2}\Big]\notag\\
&\hspace{2cm}=\frac{1}{k}\Big( \mathbf{E}_{(r,\upsilon)}\left[{\rm e}^{\int_0^t 2\beta(R_s, \Upsilon_s)\d s}g(R_t, \Upsilon_t)^2 \mathbf{1}_{(t < \tau^D)}\right]-\psi_{t}[g](r, \upsilon)^{2}\Big).
\label{just2ndmomE}
\end{align}
Next we note that 
$\psi^1_t[g^2]$, which was previously defined in \eqref{upper1},
behaves similarly to $\psi_t[g]$, albeit that the potential $\beta$ is replaced by $2\beta$ and $g$ by $g^2$. Invoking Theorem \ref{CVtheorem}, taking note of Remark \ref{CVtheoremupgrade}, we thus conclude that there exists a $\lambda_1$ with accompanying eigenfunctions $\varphi_1$ and $\tilde\varphi_1$ in $L^+_\infty(D\times V)$ such that 
\begin{equation}
\sup_{g\in L^+_\infty(D\times V): \norm{g}_\infty\leq 1}  \left\|{\rm e}^{-\lambda_1 t}{(\varphi_1)}^{-1}{\psi^1_t[g^2]}-\langle\tilde\varphi_1, g^2\rangle\right\|_\infty = o({\rm e}^{-\varepsilon t}) \text{ as $t\rightarrow+\infty$.}
\label{prime}
\end{equation}
Taking advantage of \eqref{powerestimate}, the desired conclusion follows with the help of  \eqref{prime} and Lemma \ref{interlem} .
\end{proof}

\section{$\lowercase{h}$-NRW Monte Carlo convergence}
\begin{proof}[Proof of Theorem \ref{hNRWvar}]
 In that case, we note from the proof of Theorem \ref{NRWvar}, that the crux of the argument there boils down to the estimate \eqref{powerestimate} together with the analogue of \eqref{just2ndmomE}, which reads
\begin{align}
&\mathbf{E}_{(r,\upsilon)}^h\left[\Big(\Psi^{h\texttt{-rw}}_{k}[g](r, \upsilon)-\psi_{t}[g](r, \upsilon)\Big)^{2}\right]\notag\\
&\hspace{1cm}=\frac{1}{k}\Big(\mathbf{E}_{(r,\upsilon)}^h\left[\exp\left(2\int_0^t  \frac{{\bL}h (R_s,\Upsilon_s)}{h (R_s,\Upsilon_s)} +\beta(R_s,\Upsilon_s)\d s\right)\frac{g(R_t, \Upsilon_t)^2}{h(R_t, \Upsilon_t)^2}\mathbf{1}_{(t<\tau^D)}
\right]-\psi_{t}[g](r, \upsilon)^{2}\Big)\notag\\
&\hspace{1cm}\leq 
\frac{1}{k}\Big( \mathbf{E}_{(r,\upsilon)}^h\left[\exp\left(2\int_0^t  \frac{{\bL}h (R_s,\Upsilon_s)}{h (R_s,\Upsilon_s)} +\beta(R_s,\Upsilon_s)\d s\right)\mathbf{1}_{(t<\tau^D)}
\right]-\psi_{t}[g](r, \upsilon)^{2}\Big)\notag\\
&\hspace{1cm}=\frac{1}{k}\Big( \mathbf{E}_{(r,\upsilon)}\left[\exp\left(\int_0^t  \frac{{\bL}h (R_s,\Upsilon_s)}{h (R_s,\Upsilon_s)} +2\beta(R_s,\Upsilon_s)\d s\right)\frac{h(R_t, \Upsilon_t)}{h(r,\upsilon)}\mathbf{1}_{(t<\tau^D)}
\right]-\psi_{t}[g](r, \upsilon)^{2}\Big)
\label{just2ndmomF}
\end{align}
where we have used that $g\leq h$.

\smallskip

To complete the proof, appealing to \eqref{supLh/h}, we can  follow 
the reasoning in the proof of Theorem \ref{NRWvar} and Lemma \ref{interlem}. In particular,  if we write 
\[
\overline{\varsigma} : = \sup_{r\in D,\upsilon\in V}\frac{(\bL +\beta) h(r,\upsilon)}{h(r,\upsilon)}\leq 
 \sup_{r\in D,\upsilon\in V}\frac{\bL h(r,\upsilon)}{h(r,\upsilon)} + \overline{\beta}<\infty,
\]
and 
\[
\underline{\varsigma}: = \inf_{r\in D,\upsilon\in V}\frac{(\bL +\beta) h(r,\upsilon)}{h(r,\upsilon)},
\]
then there exists a $\lambda_2$  satisfying
 $ \lambda_* +\underline{\varsigma}\leq \lambda_2\leq \lambda_*+\overline{\varsigma}.$  such that the desired result holds. 
 \end{proof}

\begin{proof}[Proof of Lemma \ref{thm: costrwh}]
  We can take inspiration from the proof of Lemma~\ref{thm: costrw}, which is given immediately preceding the statement of the lemma. In particular considering \eqref{AcostRW}, when we are under the assumptions of Theorem \ref{hNRWvar}, appealing to \eqref{ahpih}, for $f\in L^+_\infty(D\times V)$ the cost function \eqref{RWcostfunction} satisfies
  \begin{align}
    \mathbf{E}_{(r,\upsilon)}^h\big[C_{t}[f]\big]
    & = \mathbf{E}_{(r,\upsilon)}^h\left[\int_0^{t}  \mathbf{1}_{(s<\tau^D)}\alpha_h\pi_h[f](R_s, \Upsilon_s) \d s\right]\label{eq:h-integral}\\
    &=\mathbf{E}_{(r,\upsilon)}^h\left[\int_0^{t} \mathbf{1}_{(s<\tau^D)}\alpha(R_s, \Upsilon_s) \int_V f(R_s, \upsilon') \frac{h(R_s, \upsilon')}{h(R_s, \Upsilon_s)}\pi(R_s, \Upsilon_s,\upsilon')\d \upsilon'\d s\right]\notag\\
    &=\frac{1}{h(r,v)}\mathbf{E}_{(r,\upsilon)}\Big[\int_0^{t} \mathbf{1}_{(s< \tau^D)}  {\rm e}^{-\int_0^s \frac{\bL h}{h}(R_u,\Upsilon_u)\d u } \alpha(R_s, \Upsilon_s) \notag \\
    & \qquad \qquad \qquad \qquad \qquad \int_V  f(R_s,\Upsilon_s) h(R_s, \upsilon') \pi(R_s, \Upsilon_s,\upsilon')\d \upsilon'\d s\Big]\notag\\
    &\leq  \frac{||\alpha h f||_\infty}{h(r,\upsilon)}\int_0^t \mathbf{E}_{(r,\upsilon)}\left[ {\rm e}^{\int_0^s \beta(R_u,\Upsilon_u)\d u }\mathbf{1}_{(s< \tau^D)} {\rm e}^{-\int_0^s \frac{(\bL +\beta)h}{h}(R_u,\Upsilon_u)\d u } \right]\d s
      \label{intexp}
  \end{align}
  where we have used \eqref{hCOM2}. Recalling the definition in \eqref{betaprimes}, as well as the conclusion of  Theorem \ref{CVtheorem}, it is straightforward to estimate from \eqref{intexp} that 
  \[
    \limsup_{t\to\infty}{\rm e}^{-(\lambda_*-\underline\varsigma) t}\mathbf{E}_{(r,\upsilon)}^h\big[C_{t}[f]\big] <\infty.
  \]
   In the event that $\lambda_* = \underline\varsigma$, we need to be a little more careful and note that the integrand is asymptotically a constant, and, in that case, the expected cost grows no faster than linearly.

  \smallskip
  
  In the second case, where the conditions on the domain are satisfied, then the argument follows by careful consideration of \eqref{eq:h-integral}. We first define $U^f_n := \{ (r,v) | \alpha_h \pi_h[f] \in [n,n+1)\}$. Note that as a consequence of the lower bound on $f$, scatter events on this set happen with rate at least $n c_0^{-1}$. Moreover, given there is a scatter event in $U^f_n$, the scatter event will result in scatter \emph{to} $\Omega_0$ with probability at least $p_0$. For a fixed time $s$, we write $\tau_s:= \sup\{u \le s| (R_u,\Upsilon_u) \neq (R_s-(s-u) \Upsilon_s, \Upsilon_s)$, the time of the most recent scatter event. Then let
  \begin{equation*}
    A^n_t := \{(R_s,\Upsilon_s) \in \Omega_0 \text{ for all } s \in [\tau_t,t], \text{ and } (R_{\tau_t-},\Upsilon_{\tau_t-}) \in U^f_n\}
  \end{equation*}
  that is, the event that at time $t$, the process is in $\Omega_0$, and has been in $\Omega_0$ since the most recent scatter event, which happened in the set $U^f_n$. Note that the sets $A^n_t$ are disjoint for fixed $t$.

  \smallskip
  
  As a consequence of the assumption on $\alpha_h$, we can find $\epsilon = \epsilon(\delta) >0$ such that, conditional on starting at any $(r,v) \in \Omega_0$, then the expected time to the next jump is at least $\epsilon$. Then we have the (approximate!) inequality:
  \begin{equation*}
    \mathbb{E}^h_{(r,v)}\left[\int_0^{t \wedge \tau^D} \mathbf{1}_{A_s^n} \, \d s\right] \ge \mathbb{E}^h_{(r,v)}\left[\int_0^{t \wedge \tau_D} \mathbf{1}_{ (R_s,\Upsilon_s) \in U^f_n} \, \d s\right] (1+\epsilon c_0^{-1} p_0 n) 
  \end{equation*}
  where the identity follows from the fact that the process jumps out of the set $U^f_n$ into $\Omega_0$ at rate at least $c_0^{-1}p_0 n$, and once it is there, spends on average at least $\epsilon$ time units.

  Since we must have
  \begin{equation*}
    \sum_{n \ge 0} \mathbb{E}^h_{(r,v)}\left[\int_0^{t \wedge \tau^D} \mathbf{1}_{A_s^n} \, \d s\right]  \le t
  \end{equation*}
  we conclude that
  \begin{equation*}
    \sum_{n \ge 0} n \cdot \mathbb{E}^h_{(r,v)}\left[\int_0^{t \wedge \tau^D} \mathbf{1}_{ (R_s,\Upsilon_s) \in U^f_n} \, \d s\right] \le C t
  \end{equation*}
  for some constant $C$. Substituting this estimate into \eqref{eq:h-integral} gives the desired conclusion.
\end{proof}

\section{ The one-dimensional slab reactor}\label{B}

\begin{proof}[Proof of Proposition \ref{intervalprop}] 
Denote by $\Lambda$ the set  of eigenvalues for which there exists 
$\tilde\varphi\in L^+_\infty(D\times V)$ and $\varphi\in L^+_\infty(D\times V)$  such that \eqref{eq:eigen}  holds. According to Theorem \ref{CVtheorem},  $
\lambda_\ast\in \Lambda$  and $\lambda_\ast= \sup\Lambda$.

\smallskip

Suppose that $\lambda\in \Lambda$ and let $\chi$ be the associated eigenfunction:   $\bA \chi= \lambda \chi$. Let us write $f_+(r)=\chi(r, \upsilon_0)$ and $f_-(r)=\chi(r, -\upsilon_0)$. Using \eqref{1D_backwards_operators}, we can build candidates for $\chi$  in $L^+_2(D\times V)$ from the pointwise  relations
\begin{equation}\label{eq:ode}
\frac{d}{dx}
\begin{pmatrix}
f_+ \\
 f_- 
\end{pmatrix}
\quad =\quad M_\lambda 
\begin{pmatrix}
f_+ \\
f_- 
\end{pmatrix}, \quad \text{where }  
M_\lambda 
\ = \ \frac{1}{\upsilon_0}
\begin{pmatrix}
\lambda-\sigma_{\texttt{f}}+\sigma_{\texttt{s}} & -\sigma_{\texttt{s}} \\ 
\sigma_{\texttt{s}} & -(\lambda-\sigma_{\texttt{f}}+\sigma_{\texttt{s}})
\end{pmatrix},
\end{equation}
together with the boundary conditions
\begin{equation}\label{eq:bc}
f_+(L)=f_-(-L)=0
\end{equation}

To solve \eqref{eq:ode} and \eqref{eq:bc}, let us first suppose that 
\begin{equation}\label{case1}
(\lambda-\sigma_{\texttt{f}}+\sigma_{\texttt{s}})^2\ne \sigma_{\texttt{s}}^2 \quad \Longleftrightarrow\quad \lambda \ne \sigma_{\texttt{f}} \  \text{ or } \ \lambda \ne \sigma_{\texttt{f}}-2\sigma_{\texttt{s}},
\end{equation}
In this case, $M_\lambda$ has two distinct eigenvalues $\pm\alpha_\lambda$, where
\begin{equation}\label{eq: def_d_lambda}
\alpha_\lambda= \sqrt{-\mathrm{Det}(M_\lambda)}=\frac{1}{\upsilon_0}\sqrt{(\lambda-\sigma_{\texttt{f}} +\sigma_{\texttt{s}})^2-\sigma_{\texttt{s}}^2}\;\in \mathbb R_+\cup \texttt{i}\mathbb R_+. 
\end{equation}
Then by elementary results for ODEs, for fixed $\lambda$, the solutions to \eqref{eq:ode}, \eqref{eq:bc} take the following form
\begin{equation}\label{eq: sol1}
f_+(r)=C({\rm e}^{-\alpha_\lambda r}-{\rm e}^{\alpha_\lambda r-2\alpha_\lambda L}), \qquad f_-(r)=C'({\rm e}^{\alpha_\lambda r}-{\rm e}^{-\alpha_\lambda r-2\alpha_\lambda L}), 
\end{equation}
where $C, C'\in \R\setminus\{0\}$. 
Differentiating \eqref{eq: sol1}, we find
\[
\frac{d}{dx}
\begin{pmatrix}
f_+ \\
 f_- 
\end{pmatrix}
\quad =\quad 
\begin{pmatrix}
-\alpha_\lambda  \coth(2\alpha_\lambda L) & -\frac{C}{C'}\frac{\alpha_\lambda }{\sinh(2\alpha_\lambda L)} \\
\frac{C'}{C}\frac{\alpha_\lambda}{\sinh(2\alpha_\lambda L)} & \alpha_\lambda \coth(2\alpha_\lambda L)
\end{pmatrix}
\begin{pmatrix}
f_+ \\
f_- 
\end{pmatrix}
\]
Comparing this with \eqref{eq:ode}, we get
\begin{equation}\label{eq: cdis}
\frac{C}{C'}=\frac{C'}{C}\,, \quad \frac{C'}{C}\frac{\alpha_\lambda}{\sinh(2 \alpha_\lambda L)}=\frac{\sigma_{\texttt{s}}}{\upsilon_0} \quad \text{and} \quad \alpha_\lambda \coth(2\alpha_\lambda L ) = -\frac{\lambda-\sigma_{\texttt{f}}+\sigma_{\texttt{s}}}{\upsilon_0}.
\end{equation}
This first identity yields $C/C'=1$ or $-1$. This then implies 
\begin{equation}\label{eq: feq}
\bigg|\frac{2\alpha_\lambda}{\sinh(2 \alpha_\lambda L)}\bigg|=\frac{2\sigma_{\texttt{s}}}{\upsilon_0}.
\end{equation}
Recall that $\alpha_\lambda\in \mathbb R_+\cup \texttt{i}\mathbb R_+$. 
If $\alpha_\lambda \in \mathbb R_+$, then \eqref{eq: feq} has a solution if and only if $\theta={\upsilon_0}/(2\sigma_{\texttt{s}})\ge 1$. Moreover, if $\theta>1$, \eqref{eq: feq} has a unique solution in $[0, \infty)$. Denote by $x_\ast$ this unique solution. It follows from \eqref{eq: feq} that $\alpha_\lambda=\frac{x_\ast}{2 L}$. Since $x/\sinh(x)\ge 0$ for $x\in \mathbb R$, \eqref{eq: cdis} implies that 
\[
C=C' \quad \text{and} \quad \lambda-\sigma_{\mathtt{f}}+\sigma_{\mathtt{s}}<0.
\]
We then deduce from \eqref{eq: def_d_lambda} that in the case where $\theta>1$, we have
\[
 \alpha_\lambda=\frac{x_\ast}{2 L} \quad\text{and} \quad \lambda=\sigma_{\mathtt{f}}-\sigma_{\mathtt{s}}-\sqrt{\sigma_{\mathtt{s}}^2+\Big(\frac{\upsilon_0 x_\ast}{2L}\Big)^2}. 
\]
Plugging this into \eqref{eq: sol1} and renormalising $f_+, f_-$ by taking $C^{-1}=f_+(0)$, we obtain the formula \eqref{eq:1D_eigen1} and \eqref{eq:phi1} in the case $\theta>1$. 
Next, suppose that $\alpha_\lambda \in i\mathbb R_+$. In that case,  \eqref{eq: feq} has a solution if and only if $\theta={\upsilon_0}/(2\sigma_{\texttt{s}})\le 1$. Moreover, if $\theta<1$, \eqref{eq: feq} has a finite number of solutions in $\texttt{i}\mathbb R_+$, which we denote as $\texttt{i}x_k$, $1\le k\le n$, with $0<x_1<x_2<\cdots<x_n$. By a similar argument as before, for each $x_k$, $1\le k\le n$, we get
\[
 \alpha_\lambda=\frac{x_k \texttt{i}}{2L}, \quad \lambda=\sigma_{\mathtt{f}}-\sigma_{\mathtt{s}}-\sgn(\cot(x_k))\sqrt{\sigma_{\mathtt{s}}^2-\Big(\frac{\upsilon_0 x_k}{2L}\Big)^2} \quad \text{ and }\quad \frac{C}{C'}=\sgn(\cot(x_k)). 
\]
Moreover, we have $x_{1}\in (0, \pi)$. Therefore, in the case $\theta<1$, $\lambda_\ast=\max\Lambda$ is given by \eqref{eq:1D_eigen3} and we deduce the forms of the corresponding $\varphi, \tilde\varphi$ from \eqref{eq: sol1}. 
\smallskip 

Finally, we consider the case where $(\lambda-\sigma_{\texttt{f}}+\sigma_{\texttt{s}})^2=\sigma_{\texttt{s}}^2$. An elementary computation shows that  \eqref{eq:ode} and \eqref{eq:bc} have a solution in this case if and only if $\theta=1$. In that case, the solution is given by 
\[
\lambda=\sigma_{\texttt{f}}-2\sigma_{\texttt{s}}, \qquad  f_+(r)=1-\frac{r}{L} \quad \text{and} \quad  f_-(r)=1+\frac{r}{L}. 
\]
Together with our earlier remarks regarding the relationship between $\varphi$ and $\tilde\varphi$,  this concludes the proof. 
\end{proof}

\section*{Acknowledgements} This work was born out of a surprising connection that was made at the problem formulation ``Integrative Think Tank'' as part of the EPSRC Centre for Doctoral Training SAMBa in the summer of 2015. We are indebted to Professor Paul Smith from the ANSWERS modelling group at Jacobs for the extensive discussions, pictures that have appeared in this paper as well as hosting at their offices in Dorchester. We would also like to thank Emma Horton for some of the pictures and simulations that feature in this paper.

\bibliography{references}{}

\end{document}